%% file: 260705_arXiv-v2.tex
\documentclass{article}

\usepackage{a4wide} %
\usepackage{authblk} %
\usepackage[dvipsnames]{xcolor} %

\usepackage[utf8]{inputenc}

\input{header.tex}

\usepackage{tikz}
\usetikzlibrary{arrows.meta}
\usetikzlibrary{calc,intersections,through}

\definecolor{bluegray}{rgb}{0.4, 0.6, 0.8}

\title{Gradient flows on metric graphs with reservoirs:\\
Microscopic derivation and multiscale limits}

\author[1]{Georg Heinze}
\author[2]{Jan-Frederik Pietschmann}
\author[3]{André Schlichting}
\affil[1]{{
\small
Research Group ``Partial Differential Equations'', Weierstrass Institute,
Mohrenstrasse 39, 10117 Berlin, Germany. 
{georg.heinze@wias-berlin.de}
}}
\affil[2]{{
\small
Universit\"{a}t Augsburg, Institut f\"ur Mathematik, Universit\"{a}tsstra\ss e 12a, 86159 Augsburg, Germany 
\newline
and
Centre for Advanced Analytics and Predictive Sciences (CAAPS), University of Augsburg,
Universit\"{a}tsstr. 12a, 86159 Augsburg, Germany. 
{jan-f.pietschmann@uni-a.de}
}}
\affil[3]{{
\small
Ulm University, Institute for Applied Analysis, Ulm, Germany. 
andre.schlichting@uni-ulm.de
}}

\date{\today}

\begin{document}

\maketitle

\begin{abstract}
We study evolution equations on metric graphs with reservoirs, that is graphs where a one-dimensional interval is associated to each edge and, in addition, the vertices are able to store and exchange mass with these intervals. 
Focusing on the case where the dynamics are driven by an entropy functional defined both on the metric edges and vertices, we provide a rigorous understanding of such systems of coupled ordinary and partial differential equations as (generalized) gradient flows in continuity equation format. 
Approximating the edges by a sequence of vertices, which yields a fully discrete system, we are able to establish existence of solutions in this formalism. 
Furthermore, we study several scaling limits using the recently developed framework of EDP convergence with embeddings to rigorously show convergence to gradient flows on reduced metric and combinatorial graphs.  
Finally, numerical studies confirm our theoretical findings and provide additional insights into the dynamics under rescaling.
\end{abstract}

\tableofcontents

\section{Introduction}

Our goal is the study of gradient flow dynamics on metric graphs with reservoirs, i.e., graphs, where each edge is associated to a mass-carrying interval and where the storage of mass on each vertex is also possible. Here, we are particularly interested in the (non-metric) gradient structure underlying these dynamics and we aim to employ a unified approach to study different multiscale limits both on the level of the gradient structures and the dynamics.  

As our starting point, we consider an undirected irreducible graph, that is a (finite) node set $\nodes$ with edge set $\edges \subset \nodes\times \nodes$, such that if $\sfv\sfw \in \edges$ then $\sfw\sfv\not\in\edges$. We also forbid self-loops, i.e., for all $\sfv\in \nodes$, $\sfv\sfv\not\in\edges$.  Based on the combinatorial graph $(\nodes,\edges)$, we construct a metric graph $\Mgraph$ by fixing for each $e\in\edges$ an orientation and then associating this oriented edge with a finite line segment $[0,\ell^e]$ of $\R$ for some $\ell^e>0$.
To fix notations, given $e = \sfv\sfw$, we will call $\sfv$ \emph{tail} and $\sfw$ \emph{head} and give those the orientation
\begin{equation}\label{eq:def:normalfield}
	\normal : \nodes \times \edges \to \set*{-1,0,1} \quad\text{with}\quad \normal_{\sfv}^e = \begin{cases}
		-1 , & \sfv \text{ tail} ;  \\
		0 , & \sfv \not\in e  ;  \\
		1 , & \sfv \text{ head}. 
	\end{cases}
\end{equation}
Then, we associate for the embedded metric graph an edge length to each $e\in \edges$ via the function
\begin{equation}\label{e:def:metric_edges}
	\ell : \edges\to (0,\infty) \qquad\text{giving rise to the set}\qquad  \Medges = \bigsqcup_{e\in \edges} [0, \ell^e] \,. 
\end{equation}
Here and in the following $\bigsqcup_{e\in \edges} [0,\ell^e]$ denotes the disjoint union, also called coproduct, understood as set of ordered pairs $\bigcup_{e\in \edges} (e,[0,\ell^e])$.
The metric graph is now the triplet $\Mgraph=(\nodes,\edges,\Medges)$, see Figure~\ref{fig:sketch_graph} for an example.
\begin{figure}[ht]
    \centering
    \begin{tikzpicture}[scale=1, transform shape]
    \pgfdeclarelayer{background}
    \pgfsetlayers{background,main}
    \begin{scope}[every node/.style={draw,circle}]
            \node (v1) at (0,0) {$\sfv_1$};
            \node (v2) at (2,-1) {$\sfv_2$};
            \node (v3) at (2,1) {$\sfv_3$};
    \end{scope}
    \begin{scope} %
            \node[below] (v4) at (5,0) {};
            \node[below] (v5) at (9,0) {};
            \node[below] (v6) at (5.1,-.3) {$0$};
            \node[below] (v7) at (8.9,-.3) {$\ell^e$};
            \node[below] (vl) at (4.25,0) {};
            \node[below] (vr) at (9.75,0) {};
            \node (a1) at (3,0) {};
            \node (a2) at (4,1) {};
            \node (a3) at (5,0) {};

            \draw[<-|] (vl) edge node[above] {{\small $\normal_{\sfv_2}^{e_2}=-1$}} (v4);
            \draw[|->] (v5) edge node[above] {{\small $\normal_{\sfv_3}^{e_2}=+1$}} (vr);
    \end{scope}
    \begin{scope}[every edge/.style={draw=bluegray,line width=1pt}]
            \path [-] (v1) edge node[bluegray,below] {$e_1$} (v2);
            \path [-] (v2) edge node[bluegray,below left] {$e_2$} (v3);
            \path [-] (v3) edge node[bluegray,above left] {$e_3$} (v1);
            \path [-] (v4) edge node[bluegray,below] {$e_2$} (v5);
             \draw[->] [black,thick,dashed] plot [smooth,tension=1] coordinates { (2.5,.4) (4.25,1.3) (6,.4)};
    \end{scope}
    \end{tikzpicture}
    \caption{Example of a complete three-state metric graph $\Mgraph_3\coloneqq(\nodes_3,\edges_3,\Medges_3)$ defined in terms of the nodes $\nodes_3\coloneqq\set*{\sfv_1,\sfv_2,\sfv_3}$, edges $\edges_3\coloneqq\set*{e_1,e_2,e_3}$ with $e_1\coloneqq\sfv_1\sfv_2$, $e_2\coloneqq\sfv_2\sfv_3$, $e_3\coloneqq\sfv_3\sfv_1$, edge lengths $\ell_3:\edges_3\to (0,\infty)$ and metric edges $\Medges_3$ as in~\eqref{e:def:metric_edges}.}
    \label{fig:sketch_graph}    
\end{figure}
To describe the amount of mass present on each edge and each vertex, respectively, we define the space of probability measures on $\Mgraph$ as tuples $\mu = (\gamma,\rho) \in \calP(\Mgraph) \subseteq \calM_{\geq 0}(\nodes)\times \calM_{\geq 0}(\Medges)$ such that $\gamma(\nodes)+ \rho(\Medges)=1$.

We are now in the position to introduce a system of PDEs on the metric edges that are coupled to the vertex reservoirs in a suitable way. 
Generally speaking, gradient flow methods are applicable to a variety of equations including nonlinear or nonlocal equations. However, since the main scope of this work is the introduction of appropriate gradient structures as well as the study of multiscale limits, we focus on a family of linear drift-diffusion equations on $\Medges$ that are linearly coupled with the vertex reservoirs. For a discussion of nonlocal interaction terms, we refer the reader to Section~\ref{sec:conclusion}.

To introduce the linear system, we fix a positive reference probability measure $(\omega,\pi)\in \calP_+(\Mgraph)$, where we assume $\pi$ to have a density given by $\pi(\dx x)= e^{-P(x)} \dx{x}$ for $P: \Medges\to \R$ such that $P^e : [0,\ell^e]\to \R$ is Lipschitz for all $e\in \edges$. With this, we consider for a.e. $t \in [0,T]$, the system
\begin{subequations}\label{eq:system_intro_linear}
\begin{align}
    \label{eq:edge_intro_linear}\partial_t \rho^e &= \diffedge^e\partial_x\cdot\bra*{\partial_x \rho^e +\rho^e \partial_x P^e },  &&\text{on } [0,\ell^e] \text{ for }  e\in \edges, \\
    \label{eq:Kirchhoff_intro_linear}
    - \diffedge^e \left.\bra*{\partial_x \rho^e + \rho^e \partial_x P^e}\right|_{\sfv}\cdot \sfn_{\sfv}^e &=r(e,\sfv) \rho^e|_{\sfv} - r(\sfv,e) \gamma_\sfv ,  &&\forall e\in \edges, \sfv\in \nodes,\\
    \label{eq:vertex_intro_linear}\partial_t\gamma_\sfv &= \sum_{e\in \edges(\sfv)}\bra*{ r(e,\sfv) \rho^e|_{\sfv} - r(\sfv,e) \gamma_\sfv },  &&\forall \sfv\in \nodes \,,
\end{align}
\end{subequations}
where the set of adjacent edges $\edges(\sfv)$ is defined in~\eqref{eq:def:adjacent:edges} below.
Hereby, we denote by $\rho^e|_{\sfv}= \rho^e(0)$ if $e=\sfv\sfw$ and $\rho^e|_\sfv = \rho^e(\ell^e)$ if $e=\sfw\sfv$ for some $\sfw\in \nodes$.
Moreover, we introduced the parameters $\diffedge^e>0$ as the diffusion constant on the metric edge associated to an edge $e$, $r(e,\sfv)$ as the jump rate from the endpoint of the metric edge associated to $e$ to an adjacent vertex $\sfv$, and $r(\sfv,e)$ for the reverse jump rate from the vertex $\sfv$ to the adjacent edge $e$.
We note that the boundary conditions \eqref{eq:Kirchhoff_intro_linear} ensure that the solutions of \eqref{eq:system_intro_linear} conserve their mass in time, i.e., it holds
\begin{align}
	\label{eq:mass-conservation_intro}\frac{\dx }{\dx t}\bra*{\sum_{e\in \edges}\rho^e([0,\ell^e])+\sum_{\sfv\in \nodes}\gamma_\sfv} = 0 \,.
\end{align}
Furthermore, the system~\eqref{eq:system_intro_linear} admits a gradient flow formulation provided that the jump rates satisfy the \emph{detailed balance condition} with respect to the reference measure $(\omega,\pi)\in \calP_+(\Mgraph)$, that is it holds
\begin{equation}\label{eq:DBC:r}
	\scrk^e_{\sfv}\coloneqq r(e,\sfv) 
    \sqrt{\frac{\pi^e\smash[t]{|_\sfv}}{\omega_\sfv}}
    = r(\sfv,e) 
    \sqrt{\frac{\omega_\sfv}{\pi^e|_\sfv}}
    \qquad \forall e\in \edges, \sfv\in \nodes \,.
\end{equation}
Indeed, the above conditions ensure that there exists a gradient system in continuity equation format which allows us to understand \eqref{eq:system_intro_linear} as a gradient flow with respect to the free energy consisting of edge and vertex relative entropies
\begin{align}\label{eq:def:free_energy}
	\calE(\mu) &\coloneqq \calE_{\Medges}(\rho) + \calE_{\nodes}(\gamma) \coloneqq \sum_{e\in\edges}\calH(\rho^e| \pi^e)+
	\sum_{\sfv\in\nodes} 
	\calH(\gamma_\sfv |\omega_\sfv) \,,
\end{align}
where the relative entropy $\calH$ is defined as
\begin{equation}\label{eq:def:RelEnt}
	\calH(\mu | \nu ) \coloneqq \begin{cases}
		\int \eta\bra[\big]{\pderiv{\mu}{\nu}} \dx{\nu} ,& \text{if } \mu \ll \nu \,; \\
		+\infty , &\text{else,}
	\end{cases}
	\qquad\text{where}\qquad %
	\eta(r)\coloneqq r\log r -r +1 \quad\text{ for } r\geq 0 \,.
\end{equation}
Here and in the following, we use the convention that $0\cdot \log 0 =0$.

In particular, a formal calculation shows that the free energy~\eqref{eq:def:free_energy} is a Lyapunov functional for the evolution~\eqref{eq:system_intro_linear}:
\begin{align}\label{eq:free_energy:dissipation}
	\pderiv{}{t} \mathcal{E}(\mu)
	&= -  \sum_{e\in \edges} \diffedge^e\!\!\int_0^{\ell^e} \!\abs[\bigg]{\partial_x \log\frac{\rho^e}{\pi^e} }^2 \!\dx{\rho^e} - 
	\sum_{\sfv \in \nodes} \!\sum_{e\in \edges(\sfv)}\!\!\mathscr{k}_\sfv^e\sqrt{\pi^e|_\sfv \omega_\sfv} \pra*{\log\frac{\rho^e|_\sfv}{\pi^e|_\sfv} - \log\frac{\gamma_\sfv}{\omega_\sfv}}   \pra*{\frac{\rho^e|_\sfv}{\pi^e|_\sfv} - \frac{\gamma_\sfv}{\omega_\sfv}} \leq 0 \,,
\end{align}
where the sign is thanks to the fact that the logarithm is monotone increasing.

Here, we observe that the dissipation consists of the sum two distinct mechanisms: a continuous part on each metric edge and a discrete jump part accommodating for the mass exchange between vertices and edges.

In the following subsections, we give an overview over the aspects of this work.

\subsection{Abstract gradient systems and EDP convergence with embedding}

We study the system~\eqref{eq:system_intro_linear} and its various limits from the perspective of gradient flows in continuity equation format, introduced in~\cite{PeletierSchlichting2022}. In Section~\ref{sec:EDP_general} we discuss this perspective on an abstract level, establishing in Definition~\ref{def:EDP-sol} the abstract notion of \emph{EDP solutions} to gradient flow equations as curves satisfying the \emph{energy dissipation inequality} \eqref{eqdef:EDP}.
In particular, for gradient systems satisfying a \emph{chain rule inequality} (Property~\ref{property:chain-rule}), we can understand solutions as minimizers of an \emph{energy dissipation functional}, introduced in \eqref{eq:def:EDfunctional}.

In the limits, which we will study, the base space of the dynamic will change, e.g.~from a combinatorial graph to a metric graph or vice versa. 
In order to establish these limits, we prove multiple $\Gamma$-type convergence results for the respective energy dissipation functionals $\calL^\eps$ towards a limit $\calL$ as $\eps\to 0$. 
A major step to achieve this goal is the introduction of an appropriate embedding to accommodate for the changing base spaces of the $\eps$-problems. In this way, we can obtain compactness in a common space.
Since the involved functionals are maps to $\R$, the lower limit inequalities can be achieved without the embedding. 
To summarize this strategy, we formulate the notion of \emph{EDP convergence with embedding} in Principle~\ref{def:EDP-convergence}.

\subsection{Gradient flow interpretation and energy dissipation principle}

Our first main result can be found in Section~\ref{sec:MetricGraphGF}. It is the rigorous interpretation of \eqref{eq:system_intro_linear} as gradient flow with respect to the energy \eqref{eq:def:free_energy}. 
The graph structure with reservoirs makes it necessary to encode two dissipation mechanisms acting either inside the edges or at the coupling of edge and vertex dynamics, which can be observed from the free energy dissipation identity~\eqref{eq:free_energy:dissipation}.
To this end, we tentatively define (see Definition~\ref{def:bCE:abstract} for a rigorous definition) the following gradient and divergence operations for $\Phi\coloneqq (\phi,\varphi)\in C^1(\nodes\times \Medges)$ and $\rmj\coloneqq(\bar\jmath,j)\in \calM((\nodes\times \edges)\times \Medges)$ by
\begin{alignat*}{2}
    \nabla \varphi(e,x) &\coloneqq \partial_x\varphi^e(x), &\qquad (\div j)(e,x) &\coloneqq \partial_x j^e(x),\\
    \gnabla(\phi,\varphi)(\sfv,e) &\coloneqq \phi_\sfv -\varphi^e|_\sfv, &\qquad(\gdiv\bar\jmath)_\sfv &\coloneqq -\sum_{e\in\edges(\sfv)}\bar\jmath^e_\sfv,
\end{alignat*}
as well as the dual dissipation potentials given for $\xi^e\in C^1([0,\ell^e])$ for $e\in\edges$ and $\zeta\in\R^{\nodes\times \edges}$ by 
\begin{align*}
    \calR_\Medges^\ast(\rho,\xi) &\coloneqq \sum_{e\in\edges}\frac{1}{2}\diffedge^e\int_0^{\ell^e}\abs{\xi_e}^2\dx\rho^e, \\
    \calR_{\nodes,\edges}^\ast(\mu,\zeta) &\coloneqq \sum_{e\in\edges}\sum_{\sfv\in\nodes(e)}\scrk^e_\sfv\sqrt{\rho^e|_\sfv \gamma_\sfv} \sfC^*(\zeta_\sfv^e) \qquad\text{with } \sfC^*(r)\coloneqq 4\bra[\big]{\cosh(r/2)-1} \,.
\end{align*}
We observe that on the edges, we have the classical quadratic Otto-Wasserstein dissipation, while for the coupling we consider a (non-metric) dissipation of $\cosh$-type. With these notions, we can formally rewrite \eqref{eq:system_intro_linear} as a system consisting of the \emph{continuity equations}
\begin{subequations}\label{eq:GF:strong}
\begin{equation}\label{eq:CE:strong}
    \partial_t \rho^e + \div j^e = 0  \,, \qquad
    \partial_t\gamma_\sfv + (\gdiv \bar\jmath)_\sfv = 0, 
    \,
    \qquad 
    j^e|_{\sfv} \sfn_{\sfv}^e = \bar\jmath^e_{\sfv}\,;
\end{equation}
and \emph{constitutive relations}
\begin{equation}\label{eq:CR:strong}
     j^e \coloneqq \rmD_2\calR_\Medges^\ast\bra*{\rho,\nabla(-\operatorname{D}\calE_\Medges(\rho))}|_e \quad\text{ and } \quad \bar\jmath^e_{\sfv} \coloneqq \rmD_2\calR_{\nodes,\edges}^\ast(\mu,\gnabla(-\operatorname{D}\calE(\mu)))|_{e,\sfv} \,; 
\end{equation}
\end{subequations}
revealing that it is a formal gradient flow with respect to \eqref{eq:def:free_energy}. 
System~\eqref{eq:CE:strong} has to be understood in a suitable weak form, which we provide later. On the other hand, the relations~\eqref{eq:CR:strong} are characterized variationally by the (again formally) defined \emph{energy dissipation functional}
\begin{align}\label{eq:intro:defL}
		\calL(\mu,\rmj)&\coloneqq \mathcal{E}(\mu(T))- \mathcal{E}(\mu(0)) + \calD(\mu,\rmj),
\end{align}
which contains the \emph{dissipation functional}
\begin{align*}
        \calD(\mu,\rmj) &\coloneqq \int_0^T \pra*{ \calR_\Medges(\rho_t,j_t) + \calR_{\nodes,\edges}(\mu_t,\bar\jmath_t) + \calR_\Medges^\ast(\rho_t, -\nabla\calE_\Medges'(\rho_t)) + \calR^\ast_{\nodes,\edges}(\mu_t, -\gnabla\calE'(\mu_t)) } \dx t \,. \nonumber
\end{align*}
Here, $\calR_{\nodes,\edges}$ and $\calR_\Medges$ denote the \emph{primal dissipation potentials} dual to $\calR_{\nodes,\edges}^\ast$ and $\calR_\Medges^\ast$, respectively. We refer to Definition~\ref{def:Dpot} and Definition~\ref{def:Mgraph:EDfunctional} for precise definitions of the above objects. 
By means of an appropriate chain rule, Proposition~\ref{prop:chain-rule}, we are then able to show the chain-rule inequality $\calL\geq0$.
This then allows us to characterize
solutions to \eqref{eq:system_intro_linear} as elements of the null-levelset of $\calL$. 
We emphasize that the functional $\calR_\Medges$ and its dual $\calR_{\Medges}^*$ only depend on edge-based quantities $\rho$ and $j$, while the functional $\calR_{\nodes,\edges}$ and its dual $\calR_{\nodes,\edges}^*$ depend on both edge- and vertex-based quantities $\mu=(\gamma,\rho)$ and $\rmj = (\bar\jmath,j)$. 
The former describe the diffusion along the metric edges, whereas the latter model the mass exchange between the edges and their adjacent vertices.

\subsection{Existence via microscopic derivation}
Having established a variational characterization of the system~\eqref{eq:system_intro_linear} via
a suitable energy dissipation functional~$\calL$ in~\eqref{eq:intro:defL}, allows us to show existence of solutions by an approximation strategy.

In Section~\ref{sec:microscopic}, we discretise each edge into a set of $n$ internal vertices, replacing the diffusion by jumps between nearest neighbours. Similar to the jumps at the boundary of each edge, these interior jumps are characterized by cosh-type dual dissipation potentials, the main difference being a scaling by $h^e_n \coloneqq \ell^e/n$. 
The jumps are linked to the diffusion by a suitably constructed embedding, allowing us to obtain compactness for the family of discrete solutions $(\upgamma_n,\rmf_n)_{n\in\bbN}$ satisfying $\calL_n(\upgamma_n,\rmf_n) = 0$ for all $n\in\bbN$. Those discrete solutions satisfy a birth-death type dynamic given 
internally on each $e\in \edges$ for $k=2,\ldots, n-1$ by
\begin{subequations}\label{eq:system_intro_micro}
\begin{align}
   \frac{\dx}{\dx t}\tilde\gamma^e_k 
       &= (h^e_n)^{-2} \tilde r_n(e,{k+1},k)\tilde\gamma^e_{k+1} 
   - (h^e_n)^{-2} \tilde r_n(e,k,{k+1})\tilde\gamma^e_k\\
   &+ (h^e_n)^{-2} \tilde r_n(e,{k-1},k)\tilde\gamma^e_{k-1} 
   - (h^e_n)^{-2} \tilde r_n(e,k,{k-1})\tilde\gamma^e_k
\end{align}
and at the boundary points of each edge $e=\sfv\sfw\in \edges$ by
\begin{align}
    \frac{\dx}{\dx t}\tilde\gamma^e_1 
    &= (h^e_n)^{-2} \tilde r_n(e,2,1)\tilde\gamma^e_2 
    - (h^e_n)^{-2} \tilde r_n(e,1,2)\tilde\gamma^e_1 
    +  \bar r_n(\sfv,e,1)\bar\gamma_\sfv
    - \bar r_n(e,1,\sfv)\tilde\gamma^e_1,\\
    \frac{\dx}{\dx t}\tilde\gamma^e_n 
    &= \bar r_n(\sfv,e,n) \bar\gamma_\sfv 
    - \bar r_n(e,n,\sfv)\tilde\gamma^e_n 
    + (h^e_n)^{-2} \tilde r_n(e,n-1,n)\tilde\gamma^e_{n-1} 
    - (h^e_n)^{-2} \tilde r_n(e,n,n-1)\tilde\gamma^e_n,    
\end{align}
for the rates $\tilde r_n(e,k,l) \coloneqq \diffedge^e\sqrt{\frac{\tilde\omega^e_{n,l}}{\tilde\omega^e_{n,k}}}$, $\bar r_n(e,k,\sfv) \coloneqq  \scrk^e_\sfv\sqrt{\frac{\tilde\omega^e_{n,k}}{\omega_\sfv}}$, and $\bar r_n(\sfv,e,k) \coloneqq  \scrk^e_\sfv\sqrt{\frac{\omega_\sfv}{\tilde\omega^e_{n,k}}}$, where we set $\tilde\omega^e_{n,k} \coloneqq \pi^e([(k-1)h^e_n,k h^e_n])$.

The vertex evolution is characterized at every $\sfv\in\nodes$ by
\begin{align}
    \partial_t\gamma_\sfv &= \sum_{e=\sfv\sfw\in \edges(\sfv)}\pra*{\bar r_n(e,1,\sfv)\tilde\gamma^e_1 - \bar r_n(\sfv,e,1)\bar\gamma_\sfv} 
    + \sum_{e=\sfw\sfv\in \edges(\sfv)}\pra*{\bar r_n(e,n,\sfv)\tilde\gamma^e_n - \bar r_n(\sfv,e,n) \bar\gamma_\sfv}.
\end{align}
\end{subequations}
An illustration of the  discretization can be found in Figure~\ref{fig:discretization}.

After showing the chain-rule inequality $\calL\ge 0$, the argument is completed by proving the lower limit inequality (without embedding) $\liminf_{n\to \infty}\calL_n(\upgamma_n,\rmf_n) \ge \calL(\mu,\rmj)$. This is done via continuity and duality arguments utilizing the fact that the cosh grows quadratically near zero.

\subsection{Multiscale limits}\label{ssec:MS_intro}
By rescaling the reference measures, the diffusion coefficients, the reaction rates or combinations thereof, we obtain $\eps$-dependent versions of system~\eqref{eq:system_intro_linear}. 
Depending on the particular choice of rescaling, we derive in Section~\ref{sec:multiscale} different limit systems, which we summarize below.

\paragraph*{Kirchhoff limit }%

In Section~\ref{ssec:Kirchhoff} we introduce a rescaling, which leads to a vanishing of the vertex reservoirs to obtain a limit equation with Kirchhoff-type coupling conditions at the vertices. 
To this end, we introduce the rescaled reference measures
\begin{subequations}\label{eq:KirchhoffScaling}
\begin{equation}\label{eq:KirchhoffScaling:omega}
	\omega^\eps_\sfv \coloneqq \frac{1}{Z^\eps}\eps \omega_\sfv\quad\text{and}\quad  \pi^{\eps,e} \coloneqq \frac{1}{Z^\eps}\pi^e\quad\text{where}\quad Z^\eps \coloneqq\eps\sum_{\sfv\in\nodes}\omega_\sfv + \sum_{e\in\edges}\pi^e([0,\ell^e])
\end{equation}
ensures that $(\omega^\eps,\pi^\eps) \in \calP(\Mgraph)$. 
To obtain a non-trivial limit, the detailed balance condition~\eqref{eq:DBC:r} dictates the scaling 
\begin{equation}\label{eq:KirchhoffScaling:r}
	r^\eps(\sfv,e) \coloneqq \eps^{-1} r(\sfv,e) \qquad\text{and}\qquad r^\eps(e,\sfv) \coloneqq \frac{r^\eps(\sfv,e) \omega^{\eps}_\sfv }{\pi^{\eps,e}|_{\sfv}} = \frac{r(\sfv,e) \omega_\sfv }{\pi^{\eps,e}|_{\sfv}} \qquad\forall \sfv\in \nodes , e\in\edges \,.
\end{equation}
\end{subequations}
From this, we derive that the solutions $\mu^\eps=(\gamma^\eps,\rho^\eps)$ converge (after an appropriate embedding, cf. Definition~\ref{def:tildeCE}) towards a limit $\mu^0=(0,\rho)$ satisfying the reduced system 
\begin{subequations}\label{eq:system_intro_Kirchhoff}
	\begin{align}
		\label{eq:edge_intro}\partial_t \rho^e &= \diffedge^e\partial_x\cdot\bra*{\partial_x \rho^e + \rho^e \partial_x P^e },  &&\text{on } [0,\ell^e] \ \forall e\in \edges, \\
		\label{eq:Kirchhoff_intro_limit}
		0&= \sum_{e\in \edges(\sfv)} \diffedge^e \bra*{\partial_x \rho^e + \rho^e \partial_x P^e}\cdot \sfn_{\sfv}^e,  &&\forall \sfv\in \nodes \,.
	\end{align}
\end{subequations}
The system~\eqref{eq:system_intro_Kirchhoff} was studied in~\cite{ErbarForkertMaasMugnolo2022}, where the authors showed in particular that it can be understood as an Otto-Wasserstein-type gradient flow of the free energy $\calE(\mu^0) \coloneqq  \sum_{e\in\edges}\calH(\rho^e| \pi^e)$.

\paragraph*{Fast edge diffusion limit }

In Section~\ref{ssec:edgepoints} we consider the rescaling of the diffusion constants
\begin{equation}
	\diffedge^{\eps,e} \coloneqq \eps^{-1} \diffedge^e\qquad\text{for all } e\in \edges \,.
\end{equation}
This rescaling has the effect that the dynamic on each metric edge collapes to the quasistationary evolution 
\begin{equation}
	\rho^{\eps,e}(x,t) \to \zeta^e(t) \pi^e(x) \qquad \text{as } \eps \to 0 \qquad \forall x\in [0,\ell^e]\  \forall e\in \edges \,.
\end{equation}
In particular, prelimit solutions $\mu^\eps=(\gamma^\eps,\rho^\eps)$ converge to a limit $(\gamma,\zeta \pi)$ as $\eps \to 0$, which is a solution of the system of ordinary differential equations
\begin{subequations}\label{eq:system_intro_EdgePoints}
	\begin{align}
		\label{eq:vertex_intro_EdgeI}\partial_t\gamma_\sfv(t) &= \sum_{e\in \edges(\sfv)}\bra*{ r(e,\sfv) \zeta^e(t) \pi^e|_{\sfv} - r(\sfv,e) \gamma_\sfv(t)},  &&\forall \sfv\in \nodes,\\
		\label{eq:vertex_intro_EdgeII}\partial_t\zeta^e(t)\pi^e[0,\ell^e] &= \sum_{\sfv\in \nodes(e)}\bra*{ r(\sfv,e) \gamma_\sfv(t)-r(e,\sfv) \zeta^e(t) \pi^e|_{\sfv} },  &&\forall e\in \edges \,.
	\end{align}
\end{subequations}
The limit system~\eqref{eq:system_intro_EdgePoints} is a (non-metric) gradient flow on the extended combinatorial graph $(\hat\nodes,\hat\edges)$ with
\begin{equation}\label{eq:def:ExtendedGraph}
\text{ extended node set } \hat\nodes\coloneqq\nodes \cup \edges \text{ and extended edge set } \hat\edges \coloneqq \set*{e\sfv : e\in \edges , \sfv\in \nodes(e)}\,,
\end{equation}
where $\nodes(e)$ denotes the set of two vertices attached to $e$.
We illustrate in Figure~\ref{fig:MarkovTerminal} (left side) the extended graph $(\hat \nodes_3,\hat\edges_3)$ obtained from the metric graph $\Mgraph_3=(\nodes_3,\edges_3,\Medges_3)$ from Figure~\ref{fig:sketch_graph}.

\paragraph*{Combinatorial graph limit }

Finally, in Section~\ref{ssec:terminal} we carry out an additional rescaling to arrive at a limit, which is only supported on the combinatorial graph $(\nodes,\edges)$. More precisely, we perform a two-terminal limit on each of the old metric edges following~\cite[§7]{PeletierSchlichting2022} (see also~\cite[§3.3]{LieroMielkePeletierRenger2017} for a similar setting). 
The scaling is opposite to the Kirchhoff limit described before by reversing the roles of nodes and edges, that is (up to a renormalization to unit mass, similar to the Kirchhoff prelimit)
\begin{equation}\label{eq:scaling:TerminalLimit}
	\pi^{\eps,e} = \eps \pi^{e} \qquad\text{and}\qquad r^\eps(e,\sfv) = \eps^{-1} r(e,\sfv) \qquad\forall e\in \edges, \sfv\in\nodes \,.
\end{equation}
	\begin{figure}[ht]
	\centering
	\begin{tikzpicture}[scale=1.1]
		\tikzstyle{every node}=[draw,shape=circle]
		\node (v1) at (2,0) {$\quad\sfv_1\quad$};
		\node (e1) at (4,0) {$e_1$};
		\node (v2) at (6,0) {$\quad\sfv_2\quad$};
		\node (e2) at (5,1.73) {$e_2$};
		\node (v3) at (4,3.46) {$\quad\sfv_3\quad$};
		\node (e3) at (3,1.73) {$e_3$};
		\node (v10) at (10,0) {$\quad \sfv_1\quad$};
		\node (v20) at (14,0) {$\quad \sfv_2\quad$};
		\node (v30) at (12,3.4) {$\quad \sfv_3 \quad$};
		\tikzstyle{every node}=[]
		
		\draw[-Latex, thick] (v1) to [out=-5, in=-170] node[below] {\footnotesize$r(\sfv_1,e_1)$} (e1);
		\draw[-Latex, very thick] (e1) to [out=170, in=5] node[above] {\footnotesize$\ \ \frac{r(e_1,\sfv_1)}{\eps}$} (v1);
		\draw[-Latex, very  thick] (e1) to [out=-10, in=-175] node[below] {\footnotesize$\frac{r(e_1,\sfv_2)}{\eps}$} (v2);
		\draw[-Latex, thick] (v2) to [out=175, in=10] node[above] {\footnotesize$r(\sfv_2,e_1)$} (e1);
		\draw[-Latex, thick] (v2) to [out=115, in=-50] node[right,above,rotate=-60] {\footnotesize$r(\sfv_2,e_2)$} (e2);
		\draw[-Latex, very  thick] (e2) to [out=-70, in=125] node[left,below,rotate=-60] {\footnotesize$\frac{r(e_2,\sfv_2)}{\eps}\quad$} (v2);
		\draw[-Latex, thick] (v3) to [out=-65, in=130] node[left,below,rotate=-60] {\footnotesize$r(\sfv_3,e_2)$} (e2);
		\draw[-Latex, very thick] (e2) to [out=110, in=-55] node[right,above,rotate=-60] {\footnotesize$\frac{r(e_2,\sfv_3)}{\eps}$} (v3);
		\draw[-Latex, thick] (v3) to [out=-125, in=70] node[right,above,rotate=60] {\footnotesize$r(\sfv_3,e_3)$} (e3);
		\draw[-Latex, very  thick] (e3) to [out=50, in=-115] node[left,below,rotate=60] {\footnotesize$\frac{r(e_2,\sfv_2)}{\eps}\quad$} (v3);
		\draw[-Latex, very thick] (e3) to [out=-130, in=65] node[left,above,rotate=60] {\footnotesize$\frac{r(e_3,\sfv_1)}{\eps}$} (v1);
		\draw[-Latex, thick] (v1) to [out=55, in=-110] node[right,below,rotate=60] {\footnotesize$r(\sfv_1,e_3)$} (e3);
		
		\draw[shorten >=0.7cm, shorten <=0.7cm,-Latex, double, thick] (6,1.73) to [out=0, in=180] node[above]{\Large$\eps\to 0$} (10,1.73);

        \draw[-Latex, thick] (v10) to [out=5, in=175] node[above] {\footnotesize$\hat r(\sfv_1,\sfv_2)$} (v20);
        \draw[-Latex, thick] (v20) to [out=-175, in=-5] node[below] {\footnotesize$\hat r(\sfv_2,\sfv_1)$} (v10);
        \draw[-Latex, thick] (v20) to [out=125, in=-65] node[left,below,rotate=-60] {\footnotesize$\hat r(\sfv_2,\sfv_3)$} (v30);
        \draw[-Latex, thick] (v30) to [out=-55, in=115] node[right,above,rotate=-60] {\footnotesize$\hat r(\sfv_3,\sfv_2)$} (v20);
        \draw[-Latex, thick] (v30) to [out=-115, in=55] node[right,below,rotate=60] {\footnotesize$\hat r(\sfv_3,\sfv_1)$} (v10);
        \draw[-Latex, thick] (v10) to [out=65, in=-125] node[left,above,rotate=60] {\footnotesize$\hat r(\sfv_1,\sfv_3)$} (v30);
	\end{tikzpicture}
	\caption{\emph{Left:} The extended graph $(\hat\nodes,\hat \edges)$ constructed from the complete three-state metric graph illustrated in Figure~\ref{fig:sketch_graph} with the rescaled rates from~\eqref{eq:scaling:TerminalLimit}, that is a  high rate of leaving the contracted metric edges $e\in \edges$. \\
		\emph{Right:} The limit for $\eps\to 0$, which leads to reduced dynamics on the three-state combinatorial graph $(\nodes_3,\edges_3)$ with harmonic averaged rates given in~\eqref{eq:intro:TerminalLimit}.}
	\label{fig:MarkovTerminal}
\end{figure}%
Under this rescaling, we arrive at a limit system supported on the combinatorial graph $(\nodes,\edges)$ described by the single probability distribution $\hat \gamma(t) \in \calP(\nodes)$ satisfying
\begin{equation}\label{eq:intro:TerminalLimit}
	\partial_t \hat \gamma_\sfv = \ \ \sum_{\mathclap{\sfw\in\nodes:\sfv\sfw\in \edges}}\ \  \hat r(\sfw,\sfv) \hat \gamma_\sfw-\ \ \sum_{\mathclap{\sfw\in\nodes:\sfw\sfv\in \edges}}\ \  \hat r(\sfv,\sfw) \hat \gamma_\sfv  \,,\quad{\text{with}}\quad  \hat r(\sfv,\sfw) \coloneqq  \frac{\mathsf{Harm}\bra*{\pi^{\sfv\sfw}|_\sfv r(\sfv,\sfv\sfw), \pi^{\sfv\sfw}|_\sfw r(\sfw,\sfv\sfw)}}{2\omega(\nodes) \omega_\sfv} \,,
\end{equation}
where $\mathsf{Harm}(a,b)\coloneqq 2/ (\frac{1}{a}+\frac{1}{b})$ for $a,b>0$ denotes the harmonic mean and $\omega(\nodes)\coloneqq\sum_{\sfv\in\nodes} \omega_\sfv$.
We provide a rigorous EDP convergence result towards the limit system, which then is a gradient flow on the combinatorial graph $(\nodes,\edges)$ by applying results from~\cite{LieroMielkePeletierRenger2017,PeletierSchlichting2022} in Section~\ref{ssec:terminal}.

\subsection{Numerical simulations}
In Section~\ref{sec:Numerics}, we will present numerical simulations based on discrete scheme introduced in Section~\ref{sec:microscopic}. 
We put a particular focus on the short-time behaviour in the case of initial data that are not well-prepared, thereby going beyond the analytic results. 
Furthermore, for well-prepared initial data, we study the convergence not only of the relative entropies, but also the curves themselves, comparing  Hellinger-type distances between prelimit and limit curves. 
Here, the findings are in agreement with the analytic results.
Finally, we investigate the joint fast edge diffusion and combinatorial graph limit and discuss the influence of the spatial discretization on this limit.

\subsection{Related work}\label{sec:related_work}

\paragraph*{Gradient flows and transport distances on metric graphs }
Gradient flows on metric graphs have previously been studied in \cite{ErbarForkertMaasMugnolo2022}, yet without any explicit dynamics on the vertices. In this case, the authors of \cite{ErbarForkertMaasMugnolo2022} were able to recover the classical characterization of absolutely continuous curves on the metric graph as solutions to a continuity equation with bounded velocity field. 
Based on this, they were able to rigorously show existence of solutions in the sense of an energy dissipation inequality (EDI) for the gradient flow of an energy containing an entropic part, an external potential and a nonlocal contribution.
In the complementary work \cite{BurgerHumpertPietschmann2023}, the authors study the dynamical formulation of a transport distance on a metric graph with mass reservoirs on the vertices. They use techniques from convex duality to show that this distance is well-defined (see also \cite{Fazeny2024} for the non-quadratic version of this case and related gradient flows).

Using a finite volume approximation, the authors of \cite{CancesEtAl2023} study the existence of solutions to a two-species system of coupled evolution PDEs on an interval, where the fluxes in the interior are coupled to chemical potentials on the boundary. Similar to the present setting, this coupling is realized by a linear relation%
, which they obtain from a force-flux relationship involving $\sinh$-type function.

\paragraph*{Gradient flows and transport distances on combinatorial graphs }

The theory of gradient flows on combinatorial graphs started with three independent works \cite{Maas2011,Mielke2011,ChowHuangLiZhou12}, where Markov chains were posed as gradient flows of the entropy. Later, the theory was generalized in the setting of discrete state space or nonlocal evolution equations in~\cite{Erb14,ErbarFathiLaschosSchlichting2016,Erbar2023,EspositoPatacchiniSchlichtingSlepcev2021,PeletierRossiSavareTse2022,PeletierSchlottke2022},

\paragraph*{Various definitions of EDP convergence } The \emph{$\Gamma$-convergence of gradient flows} goes back in the Hilbert space setting to~\cite{SandierSerfaty04,Serfaty11}, which implies the \emph{EDP-convergence} after~\cite{Mielke2016a,Mielke2016,LieroMielkePeletierRenger2017,MielkeMontefuscoPeletier21,PeletierSchlichting2022}.
Our variational approach is similar to the ones in~\cite{Schlichting2019} for a discrete to continuum limit of a growth model, in \cite{DisserLiero2015,HraivoronskaTse2023,HraivoronskaSchlichtingTse2024} to study the limit behaviour of random walks on tessellations in the diffusive limit, in~\cite{EspositoHeinzeSchlichting2023} for certain local limits on random graphs and in~\cite{LamSchlichting2024} for thermodynamic limits of stochastic particle system. In the context of multiscale limits, we also want to mention~\cite{MielkeStephan2020,Stephan21,PeletierRenger21,FrenzelLiero21}. 
The microscopic derivation of our model in Theorem~\ref{thm:EDP_discrete} is of $\Gamma$-convergence type, whereas the multiscale limits contained in Theorems~\ref{thm:EDP_Kirchhoff}, \ref{thm:EDP_limit_edgefast} and \ref{thm:EDP_terminal} are of EDP-convergence type.

\paragraph*{Other approaches to evolution equations on metric graphs }

The study of diffusion processes on metric graphs goes back to~\cite{FreidlinWentzell1993}, where asymptotic properties of diffusions in narrow tubes are investigated and where metric graph evolutions are obtained by coarse graining suitable perturbed Hamiltonian systems. The framework used is based on large deviations and developed further in~\cite{FreidlinSheu2000}.

Diffusion processes on metric graphs also go under the name of quantum graphs~\cite{QuantumGraphsApplications2006}. In this setting, similar multiscale limits as investigated in this work are considered, as for instance shrinking edges~\cite{BerkolaikoLatushkinSukhtaiev2019}.
There is by now a very classical approach towards evolution equations on networks, including metric graphs, based on semigroup theory, which is summarized in the book~\cite{Mugnolo2014}. 
In this context, we highlight the work~\cite{mugnolo2007dynamic}, where the authors study the asymptotic behaviour and regularity of solutions for a drift-diffusion system on the edges of a metric graph that are coupled to vertex dynamics. This model is indeed similar to the prelimit dynamics studied in the present work, though the applied techniques differ.

For a different recent review of the semigroup approach see also~\cite{KramarFijav-Puchalska2020}. 
The theory does not only apply to diffusion type processes, but also transport processes, as for instance in~\cite{BuddeKramar2024}.

\paragraph*{Stationary measures of diffusions on metric graphs }

There are different approaches to study stationary measures on metric graphs.
A functional analytic setting is used in~\cite{Carlson2008,Carlson2022} to characterize stationary states as the kernel for a metric graph Laplacian with different boundary conditions defined via a suitable Dirichlet form.
An approach inspired from Markov processes and statistical mechanics is used in~\cite{AleandriColangeliGabrielli2020}, leading to a combinatorial representation of the stationary states.

\subsection{Notation}

\paragraph*{Notational conventions }

To distinguish quantities related to the combinatorial graph $(\nodes,\edges)$ from the metric edges $\Medges$, we use the following notational convention:
\begin{itemize}
	\item Combinatorial graph quantities: Sans serif (e.g. $\sfv\in \nodes$) with subscripts.
	\item Metric graph quantities: Roman fonts (e.g. $e\in \edges$ or $x\in [0,\ell^e]\subset [0,\infty)$) with superscripts.
\end{itemize}
For a given node $\sfv\in \nodes$, we define the adjacent edges by
\begin{equation}\label{eq:def:adjacent:edges}
	\edges(\sfv)\coloneqq \set*{e\in \edges: e =\sfv\sfw \text{ or } e=\sfw\sfv \text{ for some } \sfw \in \nodes } \,.
\end{equation}
Analogously, for a given edge $e\in \edges$, it is convenient to define the set of adjacent nodes by
\begin{equation}\label{eq:def:adjacent:nodes}
	\nodes(e)\coloneqq \set*{\sfv,\sfw} \quad\text{ for } e=\sfv\sfw\in \edges \,.
\end{equation}
Based on these two definitions, we have for any combinatorial graph function $\varphi:\nodes \times \edges \to \R$ the identity
\begin{equation}\label{eq:resum}
	\sum_{\sfv\in \nodes }\sum_{e\in\edges(\sfv)} \varphi_\sfv^e = \sum_{e\in \edges}\sum_{\sfv \in \nodes(e)} \varphi_\sfv^e \,.
\end{equation}

\paragraph*{Spaces of functions and measures }

We briefly write $\varphi\in C^k(\Medges)$ for a family of functions $\varphi = \set*{\varphi^e \in C^k([0,\ell^e])}_{e\in \edges}$. Note however, that no continuity or differentiability is assumed at common nodes. 
Similarly, we have that the measure $j\in \calM(\Medges)$ denotes the family of measures $j = \set*{ j^e\in \calM([0,\ell^e])}_{e\in \edges}$.

\paragraph*{Evaluation of metric functions on combinatorial nodes } By a slight abuse of notation, we evaluate functions $\varphi\in C(\sfL)$ for a given edge $e=\sfv\sfw\in \edges$ on the nodes $\sfv,\sfw\in \nodes$ connected by the edge, by denoting 
\begin{equation}\label{eq:EdgeEvaluation}
	\varphi^e|_{\sfv} = \varphi^e(0) \qquad\text{and}\qquad \varphi^e|_{\sfw} = \varphi^e(\ell^e) \,.
\end{equation}
For $\sfv\notin \nodes(e)$, we set $\varphi^e|_\sfv\coloneqq0$.
Moreover, for $\varphi\in C(\nodes\times \Medges)$, we have continuity along the metric edge $e=\sfv\sfw\in \edges$, that is
\begin{equation*}
	\lim_{x\to 0} \varphi^e(x) = \varphi^e(0) = \varphi^e|_{\sfv} \qquad\text{and}\qquad \lim_{x\to \ell^e} \varphi^e(x) = \varphi^e(\ell^e) = \varphi^e|_{\sfw} \,.
\end{equation*}
However, there is no continuity implied between the endpoint evaluations $\varphi^e|_{\sfv}$ and the vertex values $\varphi_\sfv$. 

\paragraph*{Integration by parts formula } For two functions $\varphi,\psi\in C^1(\sfL)$ we have the integration by parts formula
\begin{equation}\label{eq:L:IntByParts}
	 \sum_{e\in \edges} \int_0^{\ell^e} \varphi^e(x) \partial_x \psi^e(x) \dx{x} = \sum_{e\in \edges} \bra*{\varphi^e(\ell^e) \psi^e(\ell^e)- \varphi^e(0) \psi^e(0) } - \sum_{e\in \edges} \int_0^{\ell^e} \partial_x \varphi^e(x) \psi^e(x) \dx{x} \,.
\end{equation}
Note that with the notation~\eqref{eq:EdgeEvaluation}, boundary terms can be rewritten as
\begin{equation}\label{eq:resum:bdry}
 	\sum_{e\in \edges} \bra*{\varphi^e(\ell^e) - \varphi^e(0)} = \sum_{e=\sfv\sfw\in \edges} \bra[\big]{ \varphi^e|_\sfw - \varphi^e|_{\sfv} } \,.
\end{equation}

\paragraph*{Variational derivatives of functionals } 
A functional $\calE: X\to \R$ for $X$ an open subset of some Banach space is said to have functional derivative $\calE'(x)\in X^*$ in $x$, provided that this is the unique element such that
\begin{equation}\label{eq:def:VarDeriv}
	\skp{\calE'(x),y}_{X^*\times X} = \left.\pderiv{\calE(x+h y)}{h}\right|_{h=0} \,,\qquad \text{ for all } y\in X\,.
\end{equation}

\paragraph*{Convention about the use of extended reals }

We work with extended real numbers $a\in [-\infty,\infty]$ and adopt the convention
\begin{equation}\label{convention:infty}
	\abs*{\pm \infty} = +\infty \quad a \cdot (+\infty) \coloneqq \begin{cases}
		+\infty  & \text{ if } a >0 ,\\
		0  & \text{ if } a = 0 ,\\
		-\infty & \text{ if } a < 0
	\end{cases}
	\quad\text{and}\quad
	a \cdot (-\infty) = - a \cdot (+\infty) \,.
\end{equation}
Also $\pm \infty + a = \pm\infty$ for any $a\in(-\infty,\infty)$.

\section{Gradient systems and EDP convergence with embeddings}\label{sec:EDP_general}

We follow the notion of gradient systems in continuity equation format, which was formalized in~\cite{PeletierSchlichting2022}, but used implicitly in various previous works as discussed in~\ref{sec:related_work}.
Since we deal with different spaces in the following, we consider two \emph{base spaces} $\sfX,\sfY$, which are two compact topological spaces. Moreover, we assume that there exists an abstract gradient $\bnabla: C^1(\sfX)\to C(\sfY)$, which is the basis for the weak formulation of the continuity equation. We define an abstract divergence operator $\bdiv: \calM(\sfY) \to \calM(\sfX)$ by duality as negative adjoint of the abstract gradient via
\begin{equation}\label{eq:def:bdiv}
	\skp{\psi,\bdiv j}_{\sfX} = - \skp{ \bnabla \psi, j}_{\sfY} \qquad \text{for all } \psi \in C^1(\sfX),\, j\in \calM(\sfY) \,,
\end{equation}
where $\skp{\cdot,\cdot}_{\sfX}$ and $\skp{\cdot,\cdot}_{\sfY}$ are suitable dual pairings on $\sfX$ and $\sfY$, respectively. In our setting, $j$ and $\bdiv j$ are typically measures and hence $\psi$ and $\bnabla\psi$ need to be continuous bounded functions.
Since in our setting $\sfX$ and $\sfY$ are compact, the narrow and wide topologies coincide.
\begin{definition}[Continuity equation]\label{def:continuity equation}\label{def:CE-top}
	For $T>0$ a pair~$(\rho(t,\cdot),j(t,\cdot))_{t\in [0,T]}$ satisfies the continuity equation  $\partial_t\rho + \bdiv j = 0$ if:
	\begin{enumerate}
		\item \label{def:CE-part1-continuity}
		For each~$t\in[0,T]$, $\rho(t,\cdot)\in\calM_{\geq0}(\sfX)$, and 
		the map~$t\mapsto \rho(t,\cdot)$ is continuous on~$\calM_{\geq0}(\sfX)$
		\item \label{def:CE-part2-j}
		For each~$t\in[0,T]$, $j(t,\cdot)\in\calM(\sfY)$, 
		the map~$t\mapsto j(t,\cdot)$ is measurable in~$\calM(\sfY)$
		and the joint measure $|j|\in \calM([0,T]\times \sfY)$ is locally finite on $[0,T]\times \sfY$. 
		\item \label{def:CE-part3-weak-eq}
		The pair solves $\partial_t\rho + \bdiv j = 0$ in the sense that  for any $\varphi\in C_{\mathrm c}^1((0,T)\times \sfX)$,
		\begin{equation}
			\label{eq:weak-form-CE}
			\int_0^T\!\!\int_{\sfX} \partial_t \varphi(t,x) \, \rho(t,\dx x) \dx t 
			+\int_0^T\!\!\int_{\sfY} \bnabla \varphi(t,y) \, j(t,\dx y)\dx t = 0.
		\end{equation}
	\end{enumerate}
	The set of all pairs~$(\rho,j)$ satisfying the continuity equation is denoted by $\CE$.
	
	In addition, a family $(\rho_n,j_n)_{n\in\bbN}\subset \CE$ is said to converge to $(\rho,j)\in \CE$ if 
	\begin{enumerate}[label=(\roman*)]
		\item $\rho_n(t)\to \rho(t)$ in $\calM_{\ge 0}(\sfX)$ for all $t\in[0,T]$.
		\item $j_n\to j$ in $\calM([0,T]\times \sfY)$.
	\end{enumerate}
\end{definition}
We refer to the continuity equation structure above as $(\sfX,\sfY,\bnabla)$ with $\bdiv$ defined as negative adjoint through~\eqref{eq:def:bdiv}. 
Next, we introduce the additional ingredients for specifying a gradient system.
\begin{definition}[Gradient system in continuity equation format]\label{def:GradSystCE}
	A gradient system in continuity equation format is a quintuple $(\sfX,\sfY,\bnabla,\calE,\calR)$ with the following properties:
	\begin{enumerate}
		\item $\sfX,\sfY$ are a topological spaces;%
		\item $\bnabla$ is a linear map from  $C^1(\sfX)$ to $C(\sfY)$,  with negative dual $\bdiv$ defined through~\eqref{eq:def:bdiv};
		\item $\calE:\calM_{\geq0}(\sfX)\to\R$ is an energy functional;
		\item \label{def:GradSystCE:DP}
		$\calR^*$ is a \emph{dual dissipation potential}, which means that for each $\rho\in \calM_{\geq0}(\sfX)$, $\Xi \mapsto \calR^*(\rho,\Xi)$ is a convex lower semicontinuous functional on $C(\sfY)$ satisfying $\min \calR^*(\rho,\cdot) = \calR^*(\rho,0) = 0$. 
	\end{enumerate}
\end{definition}
The dual dissipation potential induces a \emph{primal dissipation potential} by
\begin{equation}\label{eq:def:R*}
	\calR(\rho,j) \coloneqq \sup_{\Xi \in C(\sfY)} \set*{ \skp{j ,\Xi}_{\sfY} - \calR^*(\rho,\Xi)}  \,.
\end{equation}
In our setting, the dissipation potentials will be sufficiently regular with respect to the second variable, such that the subdifferentials $\partial_2 \calR$ and $\partial_2 \calR^*$ are single valued and we denote the single element with $\rmD_2 \calR$ and $\rmD_2\calR^*$, respectively. The dissipation potential defines a \emph{kinetic relation} through the equality case (contact set) in~\eqref{eq:def:R*}, which by basic convex analysis is characterized by any pair $(j,\Xi)\in\calM(\sfY)\times C(\sfY)$ satisfying
\begin{equation}\label{eq:KR}
	\skp{j,\Xi} = \calR(\rho,j)+\calR^*(\rho,\Xi) 
	\quad\Longleftrightarrow\quad
	j= \rmD_2 \calR^*(\rho,\Xi)
	\quad\Longleftrightarrow\quad
	\Xi = \rmD_2 \calR(\rho,j) \,.
\end{equation}
A gradient flow is now characterized by a solution of the continuity equation with flux given by the \emph{constitutive relation} 
\begin{equation}\label{eq:GF:abstract}
	\partial_t \rho + \bdiv j = 0 \qquad\text{and}\qquad j = \rmD_2 \calR^*(\rho,-\bnabla \calE'(\rho)) \,.
\end{equation}
The pair $(\rho,j)$ in~\eqref{eq:GF:abstract} can be characterized variationally as follows.
\begin{definition}[EDP solution to gradient system] \label{def:EDP-sol}
	Given a gradient system in continuity equation format $(\sfX,\sfY,\bnabla,\calE,\calR)$.
	For each $T>0$ and any $(\rho,j)\in \CE$ define the \emph{dissipation functional}
    \begin{equation}\label{eqdef:ED:dissipation}
		\calD(\rho,j) \coloneqq 
		\begin{cases}
			 \int_0^T \Bigl[ \calR(\rho,j)+\calR^*(\rho,-\bnabla \calE'(\rho))\Bigr]\dx t & \text{if }(\rho,j)\in \CE \,;\\
			+\infty &\text{otherwise}.
		\end{cases}
	\end{equation}
	Whenever $\calE(\rho(0))<\infty$, the \emph{energy dissipation functional} is defined by
	\begin{equation}\label{eq:def:EDfunctional}
		\calL(\rho,j)\coloneqq \calE(\rho(T))-\calE(\rho(0)) + \calD(\rho,j).
	\end{equation}
	A curve $(\rho,j)\in \CE$ is an \emph{EDP solution} of the gradient system $(\sfX,\sfY,\bnabla,\calE,\calR)$ provided that
	\begin{equation}\label{eqdef:EDP}	
		\calL(\rho,j) \leq 0 \,.
	\end{equation}	
\end{definition}

In general $-\bnabla \calE'$ is not in $C(Y)$, thus making it necessary to replace $\calR^*(\rho,-\bnabla \calE'(\rho))$ by a relaxed Fisher information term $\calI(\rho)$.
In light of this low regularity, a main challenge when linking gradient systems and gradient flows is establishing the chain rule property.
This property ensures in particular that inequality~\eqref{eqdef:EDP} characterizes EDP solutions as minimizers of the energy dissipation functional $\calL$.

\begin{property}[Chain rule]\label{property:chain-rule}
	A gradient system $(\sfX,\sfY,\bnabla, \calE,\calR)$ satisfies the chain rule inequality, if for $(\rho,j)\in\CE$ with $\sup_{t\in[0,T]}\calE(\rho(t))<\infty$ and $\calD(\rho,j)<\infty$ it holds
	\begin{align*}
		\calE(\rho(t)) - \calE(\rho(s)) = \int_s^t \skp{j(\tau),\bnabla \calE'(\rho(\tau))} \dx{\tau} \qquad\forall \ 0\leq s \leq t \leq T\,.
	\end{align*}
	In particular, $\calL$ defined in~\eqref{eq:def:EDfunctional} is non-negative, $\calL(\rho,j) \ge 0$.
\end{property}

The power of the variational characterization is that limits of gradient flows can be characterized variational very similar to $\Gamma$-convergence. 
We introduce a generalization to the one from the literature~\cite{Mielke2016a,Mielke2016}, which is on the one hand also incorporates the continuity equation structure and compared to~\cite[Definition 2.8]{PeletierSchlichting2022} uses an embedding.	
\begin{principle}[EDP convergence with embedding] \label{def:EDP-convergence}
The family $(\sfX^\eps,\sfY^\eps,\bnabla^\eps, \calE^\eps,\calR^\eps)_{\eps>0}$ is said to \emph{EDP converge} to $(\sfX^0,\sfY^0,\bnabla^0, \calE^0,\calR^0)$, in symbols
$(\sfX^\eps,\sfY^\eps,\bnabla^\eps, \calE^\eps,\calR^\eps) \xrightarrow{\EDP} (\sfX^0,\sfY^0,\bnabla^0, \calE^0,\calR^0)$, if the following conditions hold
\begin{enumerate}[label=(\roman*)]
    \item Compactness after embedding: There exists a family of embeddings $(\Pi^\eps)_{\eps>0}$, $\Pi^\eps: \CE^\eps \to \CE^0$ such that for every family $(\rho^\eps,j^\eps)_{\eps>0}$ with $(\rho^\eps,j^\eps)\in\CE^\eps$, $\sup_{\eps>0}\sup_{t\in[0,T]} \calE^\eps( \rho^\eps(t))<\infty$, and $\sup_{\eps>0} \calD^\eps( \rho^\eps,j^\eps)<\infty$ the embedded family $(\Pi^\eps(\rho^\eps, j^\eps))_{\eps>0}$ is relatively compact in~$\CE^0$.    
    \item Lower limit inequality: For each $(\rho^\eps,j^\eps)_{\eps>0}$ with $(\rho^\eps,j^\eps)\in\CE^\eps$ satisfying the uniform bounds $\sup_{\eps>0}\sup_{t\in[0,T]} \calE^\eps( \rho^\eps(t))<\infty$ and $\sup_{\eps>0} \calD^\eps( \rho^\eps,j^\eps)<\infty$, the convergence $\Pi(\rho^\eps,j^\eps) \to (\rho^0,j^0)$ in $\CE^0$, and the well-preparedness condition $\lim_{\eps\to0}\calE^\eps(\rho^\eps(0)) = \calE^0(\rho^0(0))$, it holds 
    \begin{align*}
    \liminf_{\eps \to 0} \calL^\eps(\rho^\eps,j^\eps) &\geq \calL^0(\rho^0,j^0). 
	\end{align*}
\end{enumerate}
\end{principle}
A direct consequence of the EDP convergence of gradient systems satisfying the chain rule property is the convergence of solutions:
\begin{corollary}\label{cor:meta_conv_of_sol}
    Let $(\Mgraph^\eps,\bnabla^\eps, \calE^\eps,\calR^\eps) \xrightarrow{\EDP} (\Mgraph^0,\bnabla^0, \calE^0,\calR^0)$. 
    Assume $(\Mgraph^0,\bnabla^0, \calE^0,\calR^0)$ satisfies the chain rule Property~\ref{property:chain-rule}.
    Consider a family of curves $(\rho^\eps,j^\eps)_{\eps>0}$ satisfying the uniform energy bound $\sup_{\eps>0}\sup_{t\in[0,T]} \calE^\eps( \rho^\eps(t))<\infty$, and for which every curve $(\rho^\eps,j^\eps)\in\CE^\eps$ lies in the zero level-set of its respective EDP functional, i.e., $\calL^\eps(\rho^\eps,j^\eps) = 0$. 
    Assume the initial data $(\rho^\eps_0)_{\eps\ge0}$ of the curves are well-prepared, that is
    \begin{align*}
        \lim_{\eps\to 0} \calE^\eps(\rho^\eps_0) =\calE^0(\rho^0_0) \,.
    \end{align*}
    Then, there exists an EDP solution $(\rho^0,j^0)\in\CE^0$ in the sense of Definition~\ref{def:EDP-sol}
    and (along a subsequence) we have $\Pi^\eps(\rho^\eps,j^\eps)\to (\rho^0,j^0)$ in $\CE^0$.
\end{corollary}
\begin{proof}
    The result follows from Principle~\ref{def:EDP-convergence} establishing the chain of inequalities
    \begin{equation*}
        0 = \liminf_{\eps\to 0} \calL^\eps(\rho^\eps,j^\eps) \ge \calL^0(\rho^0,j^0) \ge 0. \qedhere
    \end{equation*}
\end{proof}

\section{Gradient structure over metric graphs with reservoirs}\label{sec:MetricGraphGF}

The first step to apply the abstract strategy from Section~\ref{sec:EDP_general} to the system~\eqref{eq:system_intro_linear} is understanding it rigorously as a gradient flow in continuity equation format.

To this end, we introduce the standing assumption on the external potential.
\begin{assumption}[Lipschitz external potential]\label{ass:Lip_Potential}
The potential $P:\Mgraph\to \R$ is Lipschitz, that is for all $e\in\edges$ it holds $P^e\in \operatorname{Lip}([0,\ell^e])$.
\end{assumption}
We recall that $\pi(\dx x)= e^{-P(x)}\dx{x}$ on $\Mgraph$ is understood as $\pi^e(\dx x) = \exp(-P^e(x))\dx x$ for every $e\in \edges$. For later reference, we state:
\begin{remark}[Poincaré inequality]\label{rem:PI}
	The equilibrium $\pi=\exp(-P)$ on the metric edges $\Medges$ satisfies a uniform Poincaré inequality, that is there exists $C_{\mathrm{PI}}\in (0,\infty)$, such that
	\begin{equation}\label{eq:ass:PI}
		\forall f \in L^2(\pi^e),e\in\edges: \qquad \int_0^{\ell^e} \abs[\bigg]{f - \frac{1}{\pi^e([0,\ell^e])}\int_0^{\ell^e} f \dx\pi^e}^2 \dx \pi^e \leq C_{\mathrm{PI}} \int_0^{\ell^e} \abs*{\partial_x f}^2 \dx \pi^e .
	\end{equation}
    Indeed, this is a classic consequence of the fact that the uniform measure on $[0,\ell^e]$ satisfies a Poincaré equality in combination with the Holley-Stroock perturbation principle~\cite{HolleyStroock1987} thanks to the fact that $\pi$ has strictly positive and bounded density. 
\end{remark}

\subsection{Continuity equation structure}

To arrive at a continuity equation structure on $\Mgraph$, we start with the definition of the continuity equation on the metric edges $\Medges$, for which we specify a family of time-dependent signed measures $\rho: [0,T]\to \calM_{\geq 0}(\Medges)$ and fluxes $j: [0,T]\to  \calM(\Medges)$ such that 
\begin{equation}\label{e:def:CE_Medges}
\forall e\in \edges: \qquad\partial_t \rho^e + \partial_x j^e = 0 \qquad \text{in } (C^1)'\bra*{[0,\ell^e]}. 
\end{equation}
To obtain a mass preserving evolution, we complement the equation with a local Kirchhoff condition, which receives the mass leaving from the metric edges due to the net-balance of the normal fluxes of $\set{j^e}_{e\in \edges}$ evaluated in $\sfv \in \nodes$ and acts as a reservoir for each of the metric edges.
Since, in general, $j\in \calM(\Medges)$ does not posses a normal flux, we introduce a new set of variables $\bar\jmath^e_\sfv : [0,T]\to \R$ for any $(\sfv,e)\in \nodes\times \edges$, which are coupled in a weak sense by imposing the integration by parts formula~\eqref{eq:L:IntByParts} on the metric edges, that is for any $j^e\in \calM([0,\ell^e])$ with smooth density $j^e(x)\dx{x}$, we impose the identity
\begin{equation}\label{eq:flux:IntByParts}
	\int_0^{\ell^e} \varphi^e(x) \partial_x j^e(x)\dx{x} = \sum_{\sfv \in \nodes(e)} \varphi^e|_{\sfv} \, \bar\jmath^e_\sfv - \int_0^{\ell^e} \partial_x \varphi^e \dx j^e
\end{equation}
Then, the net-balance of the adjacent normal fluxes is stored in the reservoir $\gamma_\sfv$ for each $\sfv\in \nodes$ and the \emph{local Kirchhoff law} becomes
\begin{equation}\label{e:def:localKirchhoff}
 \forall \sfv \in \nodes  : \qquad \partial_t \gamma_\sfv(t) = \sum_{e\in\edges(\sfv)} \bar\jmath^e_\sfv(t) \eqqcolon - \bra*{\gdiv \bar\jmath(t)}_\sfv \qquad\text{ for a.e. } t \in [0,T] . 
\end{equation}
We summarize the system~\eqref{e:def:CE_Medges} and~\eqref{e:def:localKirchhoff} compactly as metric graph continuity equation with node reservoirs in the following definition.
\begin{definition}[Continuity equation on metric graphs with node reservoirs]\label{def:bCE}
	Let $T>0$. A family of curves $\mu=\bra*{\gamma,\rho} : [0,T] \to \calP(\Mgraph)$ with flux pair $\rmj= (\bar\jmath,j)  : [0,T] \to \calM((\nodes \times \edges)\times \Medges)$ satisfies the \emph{continuity equation on the metric graph with node reservoirs} provided that~\eqref{e:def:CE_Medges} and~\eqref{e:def:localKirchhoff} hold. 
	Such a curve is denoted with $(\mu,\rmj)\in \CE$ and solves the abstract continuity equation
	\begin{equation}\label{eq:def:bCE}
		\partial_t \mu_t + \bdiv \rmj_t = 0 \qquad\text{ in } (C^1)'(\sfV\times \Medges) \text{ for a.e. } t\in [0,T] \,,
	\end{equation}
	which in terms of duality for $\Phi=(\phi,\varphi)\in C^1(\nodes \times \Medges)$ is defined by
	\begin{equation}\label{e:def:bCE:weak}
		\pderiv{}{t} \pra[\bigg]{ \sum_{\sfv \in \nodes} \phi_\sfv \gamma_\sfv(t) + \sum_{e\in\edges} \int_0^{\ell^e} \mkern-8mu \varphi^e(x) \dx{\rho^e}(t)} = \sum_{\sfv \in \nodes} \sum_{e\in\edges(\sfv)} \bra*{\phi_\sfv - \varphi^e|_\sfv} \bar\jmath^e_\sfv(t) + \sum_{e\in\edges} \int_0^{\ell^e}\mkern-8mu \partial_x \varphi^e \dx j^e(t)  \,
	\end{equation}
	for a.e.~$t\in [0,T]$.
\end{definition}
Among the solutions to~\eqref{eq:def:bCE} with bounded flux, we can find an absolutely continuous representative.
\begin{lemma}[Well-posedness of $\CE$]\label{lem:AC:representative}
	Let $(\mu,\rmj)\in\CE$ with
	\begin{equation}\label{eq:flux:div:integrable}
		\int_0^T\pra[\bigg]{ \sum_{\sfv \in \nodes} \sum_{e\in \edges(\sfv)} \abs{\bar\jmath^e_{\sfv}}(t) +  \sum_{e\in \edges} \int_0^{\ell^e} \dx\abs{j^e}(t)} \dx t < \infty \,,
	\end{equation}
	then there is $(\tilde\mu,\tilde \rmjmath)\in\CE$ such that $\tilde \mu \in \AC(0,T;\calP(\Mgraph))$ equipped with the narrow topology of measures, that is 
    for all $\Phi=(\phi,\varphi) \in C(\nodes \times \Medges)$ the map
	\begin{equation*}
		t \mapsto \skp{\varphi,\tilde\mu(t)}_{\sfV\times \sfL} \coloneqq \sum_{\sfv\in\nodes} \phi_\sfv \tilde\gamma_\sfv(t) + \sum_{e\in \edges} \int_0^{\ell^e} \varphi^e \dx{\tilde\rho^e(t)}
	\end{equation*}
    is absolutely continuous.
\end{lemma}
\begin{proof}
	Analogous to \cite[Lemma 8.1.2]{AmbrosioGigliSavare2008} for the metric edge parts and \cite[Lemma 3.1]{Erb14} for the edge-vertex-transition parts. 
\end{proof}

We show that the continuity equation on the metric graph with node reservoirs fits into the abstract framework from Section~\ref{sec:EDP_general}.
\begin{definition}[Abstract formulation of $\CE$]\label{def:bCE:abstract}
The continuity equation as defined in Definition~\ref{def:bCE} is of the form $(\sfX,\sfY,\bnabla)$ from Definition~\ref{def:continuity equation} in Section~\ref{sec:EDP_general} by defining the spaces $\sfX\coloneqq \nodes \times \Medges$ and $\sfY\coloneqq(\nodes\times\edges) \times \Medges$, which are given the disjoint union topologies of discrete and the standard one on~$\R$, respectively. The divergence $\bdiv$ in~\eqref{eq:def:bCE} is obtained as negative adjoint of the gradient operator $\bnabla: C^1(\sfX)\to C^0(\sfY)$, which is read off from~\eqref{e:def:bCE:weak} as follows: For a function $\Phi=(\phi,\varphi) \in C^1(\nodes \times \Medges)$, it is defined by the two cases
\begin{subequations}\label{eq:def:abstract-gradient}
\begin{align}
	\forall (\sfv,e)\in\nodes \times \edges: \qquad \bnabla \Phi(\sfv,e) &\coloneqq (\gnabla \Phi)(\sfv,e) \coloneqq\phi_\sfv - \varphi^e|_\sfv \,; 
	\label{eq:def:abstract-gradient:Medges} \\
	\text{and}\quad \forall (e,x) \in \Medges : \qquad \bnabla \Phi(e,x) &\coloneqq (\nabla \varphi)(e,x)\coloneqq\partial_x \varphi^e(x) \,.
	\label{eq:def:abstract-gradient:nodes}
\end{align}
\end{subequations}
It is readily checked, that for any sufficiently smooth $\rmj= (\bar \jmath,j)$ and $\Phi=(\phi,\varphi)\in C^1(\nodes \times \Medges)$, the integration by parts formula~\eqref{eq:L:IntByParts} gives the duality
\begin{equation}\label{e:grad-div:duality}
	 \skp{ \phi , \gdiv \bar\jmath}_{\nodes} + \skp{ \varphi, \div j}_{\Medges} \eqqcolon	\skp{\Phi, \bdiv  \rmj}_{\sfX}  = - \skp{\bnabla \Phi,  \rmj}_{\sfY} \coloneqq -\skp{\gnabla \Phi, \bar\jmath}_{\nodes\times\edges} - \skp{\nabla \varphi, j}_{\Medges} ,
\end{equation}
where
\begin{equation}\label{eq:def:products}
\begin{alignedat}{2}
  \skp{ \phi , \gdiv \bar\jmath}_{\nodes}&\coloneqq \sum_{\sfv \in \nodes} \phi_\sfv \bra[\bigg]{-\sum_{e\in\edges(\sfv)} \bar\jmath^e_\sfv} \,, &\qquad 
  \skp{\gnabla \Phi, \bar\jmath}_{\nodes\times\edges}&\coloneqq\sum_{\sfv \in \nodes} \sum_{e\in\edges(\sfv)} \bra*{\phi_\sfv - \varphi^e|_\sfv} \bar\jmath^e_\sfv \,, \\
  \skp{ \varphi, \div j}_{\Medges} &\coloneqq \sum_{e\in \edges} \int_0^{\ell^e} \varphi^e(x) \partial_x j^e(x) \dx{x} \,, &
  \skp{\nabla \varphi, j}_{\Medges} &\coloneqq \sum_{e\in\edges} \int_0^{\ell^e} \partial_x \varphi^e(x) j^e(x) \dx{x} \,.
\end{alignedat}
\end{equation}
With those definition $(\sfX,\sfY,\bnabla)$ defines a continuity equation after Definition~\ref{def:continuity equation}, which weak form is~\eqref{eq:def:bCE} and with the introduced notation can be rewritten in the abstract form~\eqref{eq:def:bCE} as
\begin{equation}
	\pderiv{}{t}\skp{\Phi,\mu_t}_{\sfX} =    \pderiv{}{t} \bra[\big]{\skp{\phi, \gamma_t}_{\nodes} +  \skp{\varphi, \rho_t}_{\Medges}  } =\skp{\gnabla \Phi, \bar\jmath_t}_{\nodes\times\edges} + \skp{\nabla \varphi, j_t}_{\Medges} 
	=\skp{\bnabla \varphi,\rmj_t}_\sfY 
	\qquad\forall \varphi \in C^1(\sfX) \,.
\end{equation}
Note that due to the presence of the boundary terms in~\eqref{eq:L:IntByParts}, neither is there an individual integration by parts formula between $\skp{ \phi , \gdiv \bar\jmath}_{\nodes}$ and $\skp{\gnabla \Phi, \bar\jmath}_{\nodes\times\edges}$ nor between $\skp{ \varphi, \div j}_{\Medges}$ and $\skp{\nabla \varphi, j}_{\Medges}$. 
\end{definition}

\subsection{Dissipation potentials}

We recall the constitutive relations~\eqref{eq:CR:strong} for the flux, which take the form
\begin{subequations}\label{eq:fluxes}
\begin{align}
	j^e(t) &= \diffedge^e \rho^e(t) \partial_x \log \frac{\rho^e(t)}{\pi^e}  &&\forall e\in \edges \text{ on } [0,\ell^e] \,; \label{eq:fluxes:je} \\	
	\bar\jmath^e_\sfv(t) &=  \mathscr{k}_\sfv^e\sqrt{\pi^e|_\sfv \omega_\sfv} \pra*{\frac{\rho^e(t)|_\sfv}{\pi^e|_\sfv} - \frac{\gamma_\sfv(t)}{\omega_\sfv}} &&\forall e\in \edges \ \forall \sfv\in \nodes \,. \label{eq:fluxes:hatj}
\end{align}
\end{subequations}
The equations~\eqref{eq:fluxes} are encoded through a kinetic relation~\eqref{eq:KR} between the fluxes and the variational derivatives $\mathcal{E}$ of the free energy~\eqref{eq:def:free_energy}, which we identify, in the sense of~\eqref{eq:def:VarDeriv}, for any $\mu\in \calP_+(\Medges)$ with the two components given by 
\begin{subequations}\label{eq:free_energy:variations}
	\begin{align}
		\calE_{\Medges}'(\rho)(e,x) &= \log \pderiv{{\rho^e}}{{\pi^e}}(x)  &&\forall e\in \edges \text{ and a.e. }  x\in [0,\ell^e] \,; \label{eq:free_energy:variation:Medges} \\	
		\calE_{\nodes}'(\gamma)(\sfv) &= \log \frac{\gamma_\sfv}{\omega_\sfv}  && \forall \sfv\in \nodes \,. \label{eq:free_energy:variation:nodes}
	\end{align}
\end{subequations}
The kinetic relation~\eqref{eq:KR} is encoded with the help of suitable dissipation potentials defined as follows. We use for the relation~\eqref{eq:fluxes:hatj} dissipation potentials of $\cosh$-type defined in term of the convex Legendre-Fenchel pair
\begin{equation}\label{eq:def:C-C*}
        \sfC(r) \coloneqq 2r\log\bra[\bigg]{\frac{r+\sqrt{r^2+4}}{2}}-2\sqrt{r^2+4}+4 \qquad\text{and}\qquad 
        \sfC^\ast(s) \coloneqq 4 \bra*{\cosh(s/2)- 1} \,.
\end{equation}
We note that $\sfC'(r) = 2\arsinh(r/2)$ and $(\sfC^*)'(s)= 2\sinh(s/2)$.

\begin{definition}[Dissipation potentials]\label{def:Dpot}
Let $\mu=(\gamma,\rho)\in \calP(\Mgraph)$ and $\rmj= (\bar\jmath,j)\in \calM((\nodes \times \edges)\times \Medges)$. 
The primal dissipation potential is defined by
	\begin{subequations}\label{eq:primal-dissipation}
	\begin{equation}\label{eq:primal-dissipation:decompose}
		\calR(\mu,\rmj) \coloneqq \calR_{\nodes,\edges}(\mu,\bar\jmath) + \calR_{\Medges}(\rho, j) \coloneqq \sum_{\sfv\in\nodes}\sum_{e\in\edges(\sfv)} \sfR_\sfv^e(\gamma_\sfv,\rho^e|_\sfv, \bar\jmath_\sfv^e) +\sum_{e\in \edges} \sfR^e(\rho^e,j^e), 
	\end{equation}
	where
	\begin{align}
			\sfR_\sfv^e(\gamma_\sfv,\rho^e|_\sfv, \bar\jmath_\sfv^e) &\coloneqq \sigma_{\sfv}^e(\gamma_\sfv,\rho^e|_\sfv) \sfC\bra*{\frac{\bar\jmath_\sfv^e}{\sigma_{\sfv}^e(\gamma_\sfv,\rho^e|_\sfv)}} , \qquad\text{with}\qquad      \sigma_{\sfv}^e(a,b) \coloneqq 
            \scrk_\sfv^e \sqrt{ a b} \,.
			\label{eq:primal-dissipation:nodes} \\
		 \sfR^{e}(\rho^e,j^e) &\coloneqq \frac{1}{2 \diffedge^e}\int_0^{\ell^e}\abs*{\pderiv{{j^e}}{{\rho^e}}}^2\dx\rho^e.
		 \label{eq:primal-dissipation:Medges}
	\end{align}
	\end{subequations}
    For $\mu=(\gamma,\rho)\in \calP(\Mgraph)$ and $\upxi=(\bar\xi,\xi)\in C((\nodes\times \edges)\times\Medges)$, the dual dissipation potential is defined by
   	\begin{subequations}\label{eq:dual-dissipation}
    \begin{align}\label{eq:dual-dissipation:decompose}
		\calR^*(\mu,\upxi) 
        &\coloneqq \calR_{\nodes,\edges}^*(\mu,\bar\xi) + \calR_{\Medges}^*(\rho,\xi) 
        \coloneqq \sum_{\sfv\in\nodes}\sum_{e\in\edges(\sfv)} (\sfR_\sfv^e)^\ast\bigr(\gamma_\sfv,\rho^e|_\sfv,\bar\xi^e_\sfv \bigl) +\sum_{e\in \edges} (\sfR^e)^\ast\bigr(\rho^e,\xi^e\bigl),
	\end{align}
	where
	\begin{align}
        (\sfR_\sfv^e)^\ast(\gamma_\sfv,\rho^e|_\sfv,\bar\xi_\sfv^e) &\coloneqq \sigma_\sfv^e(\gamma_\sfv,\rho^e|_\sfv) \sfC^\ast(\bar\xi_\sfv^e), \\
		(\sfR^e)^\ast(\rho,\xi^e) &\coloneqq \frac{\diffedge^e}{2}\int_{\ell^e}\abs{\xi^e}^2\dx\rho^e. \label{eq:Mgraph:Re*}
	\end{align}
    \end{subequations}
    The relaxed dual dissipation potentials are defined by
	\begin{subequations}\label{eq:relaxed-dual-dissipation}
		\begin{equation}\label{eq:relaxed-dual-dissipation:decompose}
			\calI(\mu) \coloneqq \calI_{\nodes,\edges}(\mu) + \calI_{\Medges}(\rho) \coloneqq \sum_{\sfv\in\nodes}\sum_{e\in\edges(\sfv)} \sfI_\sfv^e(\gamma_\sfv,\rho^e|_\sfv) +\sum_{e\in \edges} \sfI^e(\rho^e),
		\end{equation}
		where
		\begin{align}\label{eq:relaxed-dual-dissipations} 
			\sfI_\sfv^e(\gamma_\sfv,\rho^e|_\sfv) &\coloneqq 
			2 \sigma_\sfv^e(\gamma_\sfv,\rho^e|_\sfv) \abs[\bigg]{ \sqrt{\frac{\rho^e}{\pi^e}}\bigg|_\sfv - \sqrt{\frac{\gamma_\sfv}{\omega_\sfv}}}^2
            \quad\text{and}\quad
            \sfI^e(\rho^e) \coloneqq 2\diffedge^e\int_0^{\ell^e}\abs[\bigg]{\partial_x\sqrt{\frac{\rho^e}{\pi^e}}}^2\dx\pi^e.
		\end{align}
	\end{subequations}
\end{definition}

\begin{remark}[Other dissipation potentials]
    We note that there exist other choices for the pair $(\sigma_\sfv^e,\sfC)$ which lead to a different gradient system for the same equations \eqref{eq:system_intro_linear}. One such choice is (the continuous continuation of) the weighted logarithmic mean
\begin{equation}\label{eq:quadratic-sigma}
	\sigma_\sfv^e(\gamma_\sfv,\rho^e|_\sfv) = 
    \scrk^e_\sfv\frac{\frac{\rho^e}{\pi^e}\big|_\sfv-\frac{\gamma_\sfv}{\omega_\sfv}}{\log \frac{\rho^e}{\pi^e}\big|_\sfv -\log \frac{\gamma_\sfv}{\omega_\sfv}} 
\end{equation}
in conjunction with the quadratic function
\begin{align*}
    \sfC^\ast(\zeta) = \frac{1}{2} \abs{\zeta}^2.
\end{align*}
This quadratic choice leads to a metric structure and hence turns \eqref{eq:system_intro_linear} into a metric gradient flow in the spirit of \cite{AmbrosioGigliSavare2008}. However, as we see in \eqref{eq:quadratic-sigma}, this choice explicitly depends on the reference measure defining the energy $\calE$, i.e., the dissipative structure explicitly depends on the energy driving the dynamics. In contrast, the dissipative structure introduced in Definition~\ref{def:Dpot} is independent of the energy, it is \emph{tilt-independent}. This property of tilt-independence is shared, e.g., by the classic Wasserstein gradient structure and is generally viewed as a desirable property. We refer the reader to \cite{PeletierSchlichting2022} for a broader picture on the topic of tilting gradient flows. 
\end{remark}

Before we proceed, we collect important properties of the Fisher information functional, which in particular characterizes the edge terms of the relaxed slope.
\begin{lemma}[Properties of the Fisher information functional] \label{lem:I:convex}
    Let $\ell, \diffedge >0$ and for $P\in\mathrm{Lip}([0,\ell])$ let $\dx\pi=e^{-P}\dx x$. 
    The functional $\sfI: \calM_{\geq 0}([0,\ell]) \to [0,\infty]$ defined by
    \begin{equation}
    	\sfI(\rho) = 2\diffedge\int_0^{\ell}\abs[\big]{\partial_x\sqrt{u}}^2\dx\pi \quad\text{with}\quad u = \frac{\dx\rho}{\dx \pi} \,.\label{eq:relaxed-dual-dissipation:Mgraph}
    \end{equation}
   coincides for $\sqrt{u}\in H^1(0,\ell)$ with the \emph{Fisher information} 
    \begin{equation}\label{eq:def:Fisher:edge}
        \sfI(\rho) \coloneqq
        \begin{cases}
            \displaystyle\frac{\diffedge}{2}\int_0^{\ell} \abs*{\frac{\partial_x u}{u}}^2 \dx \rho , & \text{ if } \rho\ll \pi \text{ with } \dx \rho = u \dx \pi \,;\\
            +\infty , & \text{ else.}
        \end{cases}
    \end{equation}
    In particular, it is convex and lower semicontinuous with respect to the narrow topology on its domain. 
    
\end{lemma}
\begin{proof}
The representation~\eqref{eq:def:Fisher:edge} for $\sqrt{u}\in H^1(0,\ell)$ is well-known, see e.g. in~\cite[Lemma 2.2]{GianazzaSavareToscani2009}. For dimension one, it e.g. follows from the fact that $H^1(0,\ell) \subset \AC((0,\ell))$ and the Leibniz rule for absolutely continuous curves.

Convexity and lower semicontinuity is best observed from the dual representation of \eqref{eq:def:Fisher:edge}, which is derived by the convex dual of the square-function and suitable integration by parts.
From $\sfI(\rho)<\infty$, we get that $u\in W^{1,1}(\pi)$ by Hölder's inequality:
\[
    \int_0^{\ell} \abs*{\partial_x u} \dx{\pi}=  \int_0^{\ell} \abs*{\frac{\partial_x u}{u}} u \dx{\pi}
    =  \int_0^{\ell} \abs*{\frac{\partial_x u}{u}} \dx{\rho}\leq \sqrt{\frac{2}{\diffedge} \sfI(\rho)} \sqrt{\rho([0,\ell]) } < \infty \,.
\]
With this and noting that $\pi$ has a Lipschitz density $e^{-P}$ by Assumption~\ref{ass:Lip_Potential}, we have for any $\varphi\in C_c^1((0,\ell))$ and $u\in W^{1,1}(\pi)$ the integration by parts identity
\[
    \int_0^{\ell} \partial_x u \varphi \dx{\pi} = \int_0^{\ell} \bra*{\partial_x P \varphi - \partial_x \varphi} u \dx{\pi} \,.
\]
Hence, we derive the dual formulation as
	\begin{align}
		\sfI(\rho) &= \sup_{\varphi\in C_c^1((0,\ell))} \diffedge \int_0^{\ell}  \bra*{ \frac{\partial_x u}{u} \varphi - \frac{\varphi^2}{2}} \dx \rho \nonumber \\
        &= \sup_{\varphi\in C_c^1((0,\ell))} \set[\bigg]{ \diffedge \int_0^{\ell}  \partial_x u \varphi \dx\pi  -\int_0^{\ell} \frac{\varphi^2}{2} \dx \rho } \nonumber \\
        &= \sup_{\varphi\in C_c^1((0,\ell))} \diffedge \int_0^{\ell}  \bra[\bigg]{ \partial_x P \varphi - \partial_x \varphi - \frac{\varphi^2}{2}} \dx \rho \,, \label{eq:dual:Fisher}
	\end{align}
    where we used that $P$ is differentiable $\rho$-a.e. since $\rho\ll \scrL$.
    
	The dual formulation~\eqref{eq:dual:Fisher} shows that 
    $\sfI$ is a supremum of linear functionals on $\calM_{\ge 0}([0,\ell])$. Therefore, $\sfI$ is convex and lower semicontinuous with respect to narrow convergence by~\cite[Theorem 3.4.1.]{Buttazzo1989}. 

\end{proof}

The functionals $\calR$ and $\calR^\ast$ are dual to each other if restricted to positive densities and we have the following representation.
\begin{lemma}[Dual estimate]\label{lem:duality:dissipations}
     For any $\mu \in \calP_+(\Mgraph)$ the dual dissipation potential $\calR^*$ defined in~\eqref{eq:dual-dissipation} is indeed the dual of $\calR$ from~\eqref{eq:primal-dissipation}, that is any $\upxi=(\bar\xi,\xi)\in C^0((\nodes\times \edges)\times\Medges)$ satisfies
     \begin{align}\label{eq:DualDissipation:duality}
     	 \calR(\mu, \rmj) = \sup_{\upxi\in C((\nodes\times\edges)\times\Medges)} \set*{ \skp{\upxi,\rmj}_{\Mgraph} - \calR^*(\mu, \upxi)} \,.
     \end{align}
     For such $\mu \in \calP_+(\Mgraph)$ it agrees with the relaxed dual dissipation potential $\calI$ 
     defined in~\eqref{eq:relaxed-dual-dissipation}, that is 
     \begin{equation}\label{eq:Identitfy:DualDissipations}
     	\calI(\mu) = \calR^*(\mu,-\bnabla \calE'(\mu)) \,.
     \end{equation}
    Furthermore, $\calI$ extends to $\calP(\Mgraph)$ as the sequentially lower semicontinuous envelope of the map $\calP_+(\Mgraph) \ni \mu \mapsto \calR^*(\mu,-\bnabla \calE'(\mu))$, i.e., for every $\mu\in \calP(\Mgraph)$ it holds
    \begin{equation}\label{eq:I:envelope}
        \calI(\mu) = \inf \set*{ \liminf_{n\to \infty} \calR^*(\mu_n,-\bnabla \calE'(\mu_n)) : \mu_n \xrightharpoonup{*} \mu \text{ as } n\to \infty \text{ with } \mu_n \in \calP_+(\Mgraph) \text{ for all } n\in \bbN } \,.
    \end{equation}
\end{lemma}
\begin{proof}
	To check the consistency between the duality relation~\eqref{eq:DualDissipation:duality} and the definition~\eqref{eq:primal-dissipation:decompose} of~$\calR$, we make use of the decomposition of $\calR^\ast$ from~\eqref{eq:dual-dissipation:decompose} and the products in~\eqref{eq:def:products} to obtain
	\begin{align*}
		\calR(\mu,\rmj) 
		&= \calR_{\nodes,\edges}(\mu,\bar\jmath) 
        + \calR_{\Medges}(\rho,j) 
		= \sup_{\upxi\in C((\nodes\times\edges)\times\Medges)} \set*{ \skp{\bar \xi,\bar\jmath}_{\nodes \times\edges} + \skp{\nabla \xi,j}_{\Medges} 
        - \calR_{\nodes,\edges}^*(\mu,\bar\xi) 
        - \calR_{\Medges}^*(\rho,\xi)}\\
		&= \sup_{\bar \xi \in \R^{\nodes \times \edges}} \set*{ \skp{\bar \xi,\bar\jmath}_{\nodes \times\edges} -  \calR_{\nodes,\edges}^*(\mu,\bar\xi)}
		+ \sup_{\xi\in C(\Medges)} \set*{ \skp{\nabla \xi,j}_{\Medges}- \calR_{\Medges}^*(\rho,\xi)} \,,
	\end{align*}
	where we use the fact that $C(\nodes\times\edges)\times C(\Medges)\subset C((\nodes\times\edges)\times\Medges)$ is dense with respect to weak-$^*$ convergence and that the map $\calM(\Medges)\ni j \mapsto \skp{\nabla \xi,j}_{\Medges}- \calR_{\Medges}(\rho, j)$ is upper semicontinuous for any fixed $\xi \in C^1(\Medges)$, $\rho\in \calM_+(\Medges)$, thanks to the lower semicontinuity of $\calM(\Medges)\ni j \mapsto \calR_{\Medges}(\rho, j)$ as a convex integral functional. Hence, the duality~\eqref{eq:DualDissipation:duality} is consequence of the individual dualities $\calR_{\nodes,\edges}(\mu,\bar\jmath) = \sup_{\bar \xi \in \R^{\nodes \times \edges}} \set[\big]{ \skp{\bar \xi,\bar\jmath}_{\nodes \times\edges} -  \calR^\ast_{\nodes,\edges}(\mu,\bar\xi)}$ and $\calR_{\Medges}(\rho,j) = \sup_{\xi\in C(\Medges)} \set*{ \skp{\nabla \xi,j}_{\Medges}- \calR^\ast_{\Medges}(\rho, \xi)}$, which are immediate from their definitions~\eqref{eq:primal-dissipation} and~\eqref{eq:dual-dissipation} taking also the duality between $\sfC$ and $\sfC^*$ from~\eqref{eq:def:C-C*} into account.
	
	The identity~\eqref{eq:Identitfy:DualDissipations} follows from basic algebraic manipulations, once observed after recalling the definition of $\sfC^*$ in~\eqref{eq:def:C-C*} that
	\begin{equation*}
		\sqrt{ab} \; \sfC^*\bra*{ \log \frac{a}{b}} = 2\bra*{\sqrt{a}-\sqrt{b}}^2 \qquad\text{ for all } a,b>0 \,.
	\end{equation*}
	Since the function $(a,b) \mapsto 2\bra[\big]{\sqrt{a}-\sqrt{b}}^2$ is jointly one-homogeneous, convex and continuous on $[0,\infty)\times [0,\infty)$ and using Lemma~\ref{lem:I:convex} for the part $\calI_{\Medges}$, we observe that $\calI$ is indeed the sequentially lower semicontinous envelope satisfying~\eqref{eq:I:envelope}.
\end{proof}

\subsection{Energy dissipation principle}
Having specified the continuity equation structure $(\sfX,\sfY,\bnabla)$ in Definition~\ref{def:bCE:abstract}, the dissipation potential $\calR$ in Definition~\ref{def:Dpot} and the energy $\calE$ in~\eqref{eq:def:free_energy}, we have the ingredients for a gradient system in continuity equation format $(\sfX,\sfY,\bnabla,\calE,\calR)$. The energy dissipation functional for a curve $(\mu,\rmj)\in \CE$ is abstractly defined according to Definition~\ref{def:EDP-sol}. However, since $-\bnabla \calE'(\mu)$ is not necessarily in the domain of $\calR^*(\mu,\cdot)$, we replace it with the envelope defined in~\eqref{eq:I:envelope}. 
\begin{definition}[Energy dissipation functional]\label{def:Mgraph:EDfunctional}
    For a given curve $(\mu,\rmj)\in \CE$, the \emph{dissipation functional} is defined by
    \begin{equation}\label{eq:def:Mgraph:Dfunctional}
        \calD(\mu,\rmj)\coloneqq \int_0^T \pra[\big]{\calR(\mu(t),\rmj(t)) + \calI(\mu(t))} \dx{t} 
    \end{equation}
    and the \emph{energy dissipation functional} is defined by
	\begin{equation}\label{eq:def:Mgraph:EDfunctional}
		\calL(\mu,\rmj)\coloneqq \mathcal{E}(\mu(T))- \mathcal{E}(\mu(0)) + \calD(\mu,\rmj).
	\end{equation}
\end{definition}
For later reference, we note the convexity of the involved functionals with respect to affine interpolations.
\begin{lemma}[Strict convexity with respect affine interpolation]\label{lem:ED:convex:affine}
	The map $(\mu,\rmj)\mapsto \calL(\mu,\rmj)$ is strictly convex for curves with fixed starting point within the domain of $\calL$, that is for $\mu^0,\mu^1 \in \AC(0,T;\calP(\Mgraph))$ with $\calL(\mu^0,\rmj^0),\calL(\mu^1,\rmj^1)< \infty$ starting from $\bar \mu$ with $(\mu^0,\rmj^0),(\mu^1,\rmj^1)\in \CE$, the curve defined for $\lambda \in (0,1)$ with $(\mu^\lambda,\rmj^\lambda)\in \CE$ by $\mu^\lambda \coloneqq (1-\lambda)\mu^0 + \lambda \mu^1$ and likewise $\rmj^\lambda\coloneqq (1-\lambda)\rmj^0 + \lambda \rmj^1$ satisfies
	\begin{equation}\label{eq:ED:convex}
		\calL(\mu^\lambda,\rmj^\lambda) < (1-\lambda) \calL(\mu^0,\rmj^0) + \lambda \calL(\mu^1,\rmj^1) \qquad\text{for any } \lambda \in (0,1)\,.
	\end{equation}
\end{lemma}
\begin{proof}
	The result is an immediate consequence of the strict convexity of $\calE(\mu)$ as defined in~\eqref{eq:def:free_energy}, because of the strict convexity of the relative entropy defined in~\eqref{eq:def:RelEnt}. Moreover, we have the joint convexity of $\calR(\mu,\rmj)$ thanks to its definition in~\eqref{eq:primal-dissipation:decompose} as the sum of two jointly convex functional. Hereby, $\sfR_\sfv^e$ defined in~\eqref{eq:primal-dissipation:nodes} is convex as the perspective function of the convex function $\sfC$ defined in~\eqref{eq:def:C-C*} and the concavity of the geometric mean. Likewise $\sfR^e$ as defined in~\eqref{eq:primal-dissipation:Medges} is convex as the perspective function of the square, which is convex. 
	Similarly, the relaxed dual dissipation functional $\calI$ defined in~\eqref{eq:relaxed-dual-dissipation:decompose} is the sum of two convex functions. Again $\sfI_\sfv^e$ defined in~\eqref{eq:relaxed-dual-dissipation} is convex, since $(a,b)\mapsto \abs[\big]{\sqrt{a}-\sqrt{b}}^2$ is jointly convex. Finally, $\sfI^e$ defined in~\eqref{eq:relaxed-dual-dissipation:Mgraph} is convex by Lemma~\ref{lem:I:convex}.
\end{proof}
Our next result concerns the differentiability of $\mathcal{E}$ along a curve $\mu \in \AC(0,T;\calP(\Mgraph))$, which we characterize as a solution to the continuity equation $(\mu,\rmj)\in\CE$ by Lemma~\ref{lem:AC:representative}. 
\begin{proposition}[Chain rule]\label{prop:chain-rule}
	Along any curve $(\mu,\rmj)\in\CE$ with $\calE(\mu(0))<\infty$ and
 $\calL(\mu,\rmj) < \infty$, the free energy $[0,T]\ni t \mapsto \calE(\mu(t))$ is absolutely continuous satisfying the chain rule identity
	\begin{equation}\label{eq:chain-rule}
		  \calE(\mu(t)) - \calE(\mu(s)) 
          = \int_s^t \pra*{ \skp{\gnabla \calE'(\mu(\tau)), \bar \jmath(\tau)}_{\nodes,\edges}  + \skp{\nabla \calE'(\mu(\tau)), j(\tau)}_{\Medges} } \dx{\tau}. 
	\end{equation}
	Here, the variation of the energy is given in~\eqref{eq:free_energy:variations}. Since $\calE'$ takes values in $[-\infty,+\infty)$, the definition of the according gradients is extended in comparison to~\eqref{eq:def:abstract-gradient} using the convention~\eqref{convention:infty} as
	\begin{subequations}\label{eq:force_from_energy}
	\begin{align}
		\forall (\sfv,e)\in \nodes\times \edges:\quad \gnabla \calE'(\mu)(\sfv,e) &\coloneqq 
		\begin{cases}
			0 \,, &\text{if } \calE_\nodes'(\gamma)(\sfv) =  -\infty =\calE_\Medges'(\rho)(e,\cdot)|_{\sfv} \,, \\
		\calE_\nodes'(\gamma)(\sfv) - \calE_\Medges'(\rho)(e,\cdot)|_{\sfv} \,, &\text{else,}
		\end{cases} \\
		 \forall (e,x)\in \Medges: \quad \nabla \calE'(\mu)(e,x) &= \partial_x \calE'_\Medges(\rho)(e,x) \,.
	\end{align}
\end{subequations}
\end{proposition}
\begin{proof}%
	To show the statement, we combine the proof of \cite[Proposition 5.8]{ErbarForkertMaasMugnolo2022} for the metric edges with the one in~\cite[Theorem 4.16]{PeletierRossiSavareTse2022} for the reservoirs. First, we adapt the regularization procedure \cite[Section 3.3]{ErbarForkertMaasMugnolo2022} to our framework. 
    
    To do so, we note that each test function $\Phi=(\phi,\varphi) \in C^1(\nodes \times \Medges)$ consists of the edge components $\varphi^e \in C^1([0,\ell^e])$, $e\in\edges$ and vertex components $\phi_\sfv\in\R$, $\sfv\in\nodes$, which are independent of each other. In particular, we do not require $\varphi$ to be continuous when passing through a vertex from one edge to another. Therefore, for each edge $e=\sfv\sfw\in\edges$ we introduce two auxiliary edges $e_\sfv$ and $e_\sfw$, connected to $\sfv$ and $\sfw$, respectively. Using these auxiliary edges, we follow the strategy of \cite[Section 3.3]{ErbarForkertMaasMugnolo2022} to construct space-regularized test functions and by duality also space regularized measures. To apply the regularization to $\varphi\in C_c^1(\Medges)$, we need to extend it to the auxiliary edges. 
    In all cases we choose the constant extension, which ensures that the associated regularized test function $\varphi^\eps$ satisfies $\gnabla\Phi \equiv \gnabla\Phi^\eps$ with $\Phi^\eps\coloneqq(\phi,\varphi^\eps)$. With this, following the proof of \cite[Proposition 5.8]{ErbarForkertMaasMugnolo2022}, we obtain the $\Medges$-component in~\eqref{eq:chain-rule}. 
    
    Likewise, the discrete jump parts follows along the lines of~\cite[Theorem 4.16]{PeletierRossiSavareTse2022} by regularization of the free energies in~\eqref{eq:def:free_energy}. That is, we define for $k>0$ the function $\log_k:(0,\infty)\to [-k,k]$ defined by
	\begin{equation*}
		\log_k(x) \coloneqq \begin{cases}
			-k, & \log x < -k\\
			\log x, & -k\leq \log x \leq k \\
			k, & \log x \geq k 
		\end{cases}\,.
	\end{equation*}
	Let $\eta_k$ be the primitive of $\log_k$ normalized such that $\eta_k(1)=0$ for any $k>0$ and $\calH_k(\mu|\nu)$ the according regularized entropy as defined in~\eqref{eq:def:RelEnt} with $\eta$ replaced by $\eta_k$. Then the chain rule is classic for the regularized energies $\calE_{\Medges,k}(\rho)\coloneqq\sum_{e\in \edges} \calH_k(\rho^e|\pi^e)$ and $\calE_{\nodes,k}\coloneqq \sum_{\sfv\in \nodes} \calH_k(\gamma_\sfv|\omega_\sfv)$. Since $\eta_k\leq \eta$ pointwise, we also get $\calH_k(\mu|\nu)\leq \calH(\mu|\nu)$ and similar bounds for $\calE_{\Medges,k}$ and $\calE_{\nodes,k}$, which allows for a monotone convergence argument along the same lines as in~\cite[Theorem 4.16]{PeletierRossiSavareTse2022}
\end{proof}
The chain-rule and the duality structure from Lemma~\ref{lem:duality:dissipations} allow to conclude that the energy dissipation functional~\eqref{eq:def:Mgraph:EDfunctional} from Definition~\ref{def:EDP-sol} is actually non-negative on its domain. 
\begin{corollary}[Nonnegativity of energy dissipation functional]\label{cor:EDfunctional:nonneg}
	For any $(\mu,\rmj)\in \CE$ with $\calE(\mu(0))<\infty$ it holds $\calL(\mu,\rmj)\geq 0$.
\end{corollary}
\begin{proof}
If $\calL(\mu,\rmj)=\infty$, there is nothing to show. Assume $\calL(\mu,\rmj)<\infty$.
Thanks to the chain rule from Proposition~\ref{prop:chain-rule}, we get for $(\mu,\rmj)\in \CE$ solving~\eqref{eq:def:bCE} the identity
\begin{equation}\label{eq:EDfunctional:nonneg}
\begin{aligned}
  \pderiv{}{t} \calE(\mu) &%
  = -\skp{-\gnabla \calE'(\mu),\bar \jmath}_{\nodes,\edges} + -\skp{-\nabla \calE_\Medges'(\rho), j}_{\Medges}\\
  &\ge -\calI_{\nodes,\edges}(\mu)-\calR_{\nodes,\edges}(\mu, \bar\jmath)-\calI_{\Medges}(\rho) -\calR_{\Medges}(\rho, j)\\
  &= - \sum_{\sfv\in\nodes}\sum_{e\in\edges(\sfv)} \sfI_\sfv^e\bra{\gamma_\sfv,\rho^e|_\sfv} - \sum_{\sfv\in\nodes}\sum_{e\in\edges(\sfv)} \sfR_\sfv^e(\gamma_\sfv,\rho^e|_\sfv, j_\sfv^e)
  -\sum_{e\in\edges}\sfI^e(\rho^e)- \sum_{e\in\edges} \sfR^e(\rho^e,j^e).
\end{aligned}
\end{equation}
An integration in time finishes the proof.
\end{proof}
Next, we show that the zero locus set of $\calL$ indeed corresponds to weak solutions to the system~\eqref{eq:system_intro_linear}.

\begin{corollary}[Energy dissipation balance]\label{cor:EDB}
    Let $(\mu,\rmj)\in\CE$ 
    with $\calE(\mu(0))<\infty$. 
    It holds $\calL(\mu,\rmj) = 0$ if and only if $\mu$ satisfies \eqref{eq:system_intro_linear} in the sense of distributions and $\rmj$ is given by \eqref{eq:fluxes}.
\end{corollary}
\begin{proof}
    Let $\calL(\mu,\rmj) = 0$. Then, we have that all inequalities in \eqref{eq:EDfunctional:nonneg} are in fact equalities. This implies that \eqref{eq:fluxes} must hold, which in conjunction with the continuity equation yields~\eqref{eq:system_intro_linear}.

    Conversely, let us assume that \eqref{eq:system_intro_linear} holds true. In particular, we have $(\mu,\rmj)\in\CE$ and \eqref{eq:fluxes} is satisfied. This, yields equality in \eqref{eq:EDfunctional:nonneg} and hence $\calL(\mu,\rmj) = 0$.
\end{proof}
Corollary~\ref{cor:EDB} provides only a characterization of solutions and we will ensure existence by a finite dimensional approximation result via EDP convergence in the sense of Definition~\ref{def:EDP-convergence}. Nevertheless, we can already conclude uniqueness of any global minimizer of the energy dissipation functional $\calL$.
\begin{theorem}[Uniqueness]\label{thm:GF:well-posed}
	For any $T>0$ and any $\bar \mu\in\calP(\Mgraph)$ with $\calE(\bar\mu)<\infty$, there exists at most one gradient flow $\mu\in \AC(0,T;\calP(\Mgraph))$ in the sense of Definition~\ref{def:EDP-sol} starting from $\mu(0)=\bar \mu$.
\end{theorem}
\begin{proof}
	We use Lemma~\ref{lem:ED:convex:affine} along $\mu^0,\mu^1 \in \AC(0,T;\calP(\Mgraph))$ two solutions with $\calL(\mu^0,\rmj^0)=0=\calL(\mu^1,\rmj^1)$ starting from $\bar \mu$ with $(\mu^0,\rmj^0),(\mu^1,\rmj^1)\in \CE$. 
	Then, for the convex combination $(\mu^\lambda,\rmj^\lambda)\in \CE$, we get from~\eqref{eq:ED:convex} for any $\lambda\in (0,1)$ the estimate
	\begin{equation}
		\calL(\mu^\lambda,\rmj^\lambda) < (1-\lambda)\calL(\mu^0,\rmj^0) + \lambda \calL(\mu^1,\rmj^1) = 0 \,
	\end{equation}
	which is a contradiction to the non-negativity of $\calL$ shown in Corollary~\ref{cor:EDfunctional:nonneg}.
\end{proof}

\section{Microscopic derivation from pure jump process}\label{sec:microscopic}

In this section we introduce a microscopic model, which provides a finite dimensional Markov chain approximation to the evolution \eqref{eq:edge_intro_linear} on the metric graph. 
We show EDP-convergence of the associated gradient flows structure in the sense of Definition~\ref{def:EDP-convergence} from Section~\ref{sec:EDP_general}, which provides existence for the gradient structure of the metric-graph-system defined in Section~\ref{sec:MetricGraphGF}.

\subsection{Microscopic model}

We discretize the metric graph as follows.
For $e\in\edges$, $n\in\bbN$ we define the grid length $h^e_n\coloneqq \ell^e/n$, sets of internal vertices $\tilde\nodes_n\coloneqq\set{\sfv^e_k\coloneqq h^e_n(k-1/2): e\in\edges,k=1,\ldots,n}$, and sets of all vertices $\nodes_n\coloneqq\nodes\cup\tilde\nodes_n$.

    The set of internal edges is given as $\tilde\edges_n\coloneqq\set{\sfv^e_k\sfv^e_{k+1}:\sfv^e_k,\sfv^e_{k+1}\in\tilde\nodes_n}$. We further define $\bar\edges_n\coloneqq\set{\sfv\sfv^e_1:\sfv^e_1\in\tilde\nodes_n,e=\sfv\sfw}\cup\set{\sfv^{e}_n\sfw:\sfv^e_n\in\tilde\nodes_n,e=\sfv\sfw}$ and we denote $\edges_n\coloneqq \tilde\edges_n\cup\bar\edges_n$.
Next, for $e\in\edges$, $n\in\bbN$, and $\alpha\in[1,n)$ we define intervals
\begin{align*}
    I^e_{n,\alpha}&\coloneqq [(\alpha-1) h^e_n,\alpha h^e_n)
\end{align*}
and the internal discrete vertex measures $\tilde\omega_n\coloneqq (\tilde\omega^e_n)_{e\in\edges}$ as 
\begin{align}\label{eq:def_internal_vertex_measures}
\tilde\omega^e_n\coloneqq (\tilde\omega^e_{n,k})_{k=1}^n\text{ with }\tilde\omega^e_{n,k}\coloneqq \pi^e(I^e_{n,k}).
\end{align}

\begin{figure}[!ht]
\centering
\begin{tikzpicture}
  \draw (0,0) -- (5,0);
  \foreach \x in {0,.5,1.5,4.5,5}
    \draw (\x,0.05) -- (\x,-0.05);
  \foreach \x in {0.5,1.5,4.5} {
    \draw[bluegray,thick] (\x - 0.5,0.15) -- (\x + 0.5,0.15);
    }
  \foreach \x in {0,1,2,4,5} {
    \draw[blue] (\x,0.2) -- (\x,0.1);
    }  

  \node at (0, -0.3) {$0$};
  \node at (.5, -0.3) {$\sfv^e_1$};
  \node at (1.5, -0.3) {$\sfv^e_2$};
  \node at (4.5, -0.3) {$\sfv^e_n$};
  \node at (5, -0.3) {$\ell^e$};

  \node[above] at (.5,0.1) {\textcolor{bluegray}{$I^e_{n,1}$}};
  \node[above] at (1.5,0.1) {\textcolor{bluegray}{$I^e_{n,2}$}};
  \node[above] at (4.5,0.1) {\textcolor{bluegray}{$I^e_{n,n}$}};

  \draw[loosely dotted] (2.075, -0.3) -- (4., -0.3);
  \draw[blue, loosely dotted] (2.075, 0.15) -- (4., 0.15);
\end{tikzpicture}
\caption{Sketch of the discretization of the edges into new vertices. Blue: Intervals $I^e_{n,k}$.}
\label{fig:discretization}
\end{figure}

We combine these with the vertex measures of the limit model and introduce $\upomega_n\coloneqq \omega+\tilde\omega_n\in\calP(\nodes_n)$. For any probability measure $\upgamma\in\calP(\nodes_n)$, we set $\bar\gamma\coloneqq \upgamma|_\nodes$ and $\tilde\gamma\coloneqq\upgamma|_{\tilde\nodes_n}$. 

We continue by introducing appropriate free energy and dissipation functionals on these quantities. For every $n\in\bbN$ we define the discrete energies $\calE_n:\calP(\nodes_n)\to\R$ and dual dissipation potentials $\calR^\ast_n:\calP(\nodes_n)\times \R^{\edges_n}\to[0,\infty)$ for $\upxi=((\tilde \xi^e_k)_{e\in\edges,k=1,\ldots,n-1},(\bar \xi^e_\sfv)_{e\in\edges,\sfv\in\nodes})$ by
\begin{align*}
    \calE_n(\upgamma) &\coloneqq \calH(\upgamma|\upomega_n),\\
    \calR^\ast_n (\upgamma,\upxi) &\coloneqq \sum_{e\in\edges}\sum_{k=1}^{n-1} (\sfR^e_{n,k})^\ast(\tilde\gamma^e_{k},\tilde\gamma^e_{k+1},\tilde\xi^e_{k}) + \sum_{e=\sfv\sfw\in\edges}\Big[(\sfR_\sfv^e)^\ast(\bar\gamma_\sfv,\tilde\gamma^e_1,\bar\xi_\sfv^e)+(\sfR_\sfw^e)^\ast(\bar\gamma_\sfw,\tilde\gamma^e_n,\bar\xi_\sfw^e)\Big]\\
    (\sfR^e_{n,k})^\ast(\tilde\gamma,\tilde\xi)&\coloneqq \sigma^e_n\bra[\big]{\tilde\gamma^e_{k}, \tilde\gamma^e_{k+1}}\sfC^\ast(\tilde\xi^e_{k}), \qquad\text{where } \sigma^e_n(a,b)\coloneqq \tfrac{\diffedge^e}{\bra{h^e_n}^2}\sqrt{a b} \,.
\end{align*}
Similar to Definition~\ref{def:Dpot} before, we introduce the Fisher information
\begin{equation*}
	\calI_n(\upgamma) \coloneqq %
    \sum_{e\in\edges}\sum_{k=1}^{n-1} \sfI^e_{n,k}(\tilde\gamma^e_{k},\tilde\gamma^e_{k+1}) + \sum_{e=\sfv\sfw\in\edges}\Big[\sfI_\sfv^e(\bar\gamma_\sfv,\tilde\gamma^e_1)+\sfI_\sfw^e(\bar\gamma_\sfw,\tilde\gamma^e_n)\Big],
\end{equation*}
with
\begin{align*}
    \sfI^e_{n,k}(\tilde\gamma^e_{k},\tilde\gamma^e_{k+1}) &\coloneqq 
	2\sigma_n^e(\tilde\omega^e_{n,k},\tilde\omega^e_{n,k+1}) \abs[\Bigg]{ \sqrt{\frac{\tilde\gamma^e_{k}}{\tilde\omega_{n,k}^e}} - \sqrt{\frac{\tilde\gamma^e_{k+1}}{\tilde\omega_{n,k+1}^e}}}^2,\\
	\sfI_\sfv^e(\gamma_\sfv,\tilde\gamma^e_1) &\coloneqq 
	2\sigma_\sfv^e(\omega_\sfv,\tilde\omega^e_{n,1}) \abs[\bigg]{ \sqrt{\frac{\tilde\gamma^e_1}{\tilde\omega_{n,1}^e}} - \sqrt{\frac{\bar\gamma_\sfv}{\omega_\sfv}}}^2,\\
    \sfI_\sfw^e(\gamma_\sfw,\tilde\gamma^e_n) &\coloneqq 
	2\sigma_\sfw^e(\omega_\sfw,\tilde\omega^e_{n,n}) \abs[\bigg]{ \sqrt{\frac{\tilde\gamma^e_n}{\tilde\omega_{n,n}^e}} - \sqrt{\frac{\bar\gamma_\sfw}{\omega_\sfw}}}^2,
\end{align*}
for $e =\sfv\sfw$ and $\sigma_\sfv^e(a,b) = \scrk^e_\sfv\sqrt{ab}$.

Analogous to $\upxi$, we use the discrete flux measures $\rmf = ((\tilde f^e_k)_{e\in\edges,k=1,\ldots,n-1},(\bar f^e_\sfv)_{e\in\edges,\sfv\in\nodes}) \in \R^{\edges_n}$ and introduce the primal dissipation potentials $\calR_n:\calP(\nodes_n)\times \R^{\edges_n}\to[0,\infty)$ by
\begin{align*}
    \calR_n (\upgamma,\rmf) &\coloneqq \sum_{e\in\edges}\sum_{k=1}^{n-1} \sfR^e_{n,k}(\tilde\gamma^e_{k},\tilde\gamma^e_{k+1},\tilde f^e_k) + \sum_{e=\sfv\sfw\in\edges}\Big[\sfR_\sfv^e(\bar\gamma_\sfv,\tilde\gamma^e_1,\bar f_\sfv^e)+\sfR_\sfw^e(\bar\gamma_\sfw,\tilde\gamma^e_n,\bar f_\sfw^e)\Big],\\
    \sfR^e_{n,k}(\tilde\gamma,\tilde f)&\coloneqq \sigma^e_n\bra[\big]{\tilde\gamma^e_{k}, \tilde\gamma^e_{k+1}}\sfC\bra*{\frac{\tilde f^e_k}{\sigma^e_n\bra[\big]{\tilde\gamma^e_{k}, \tilde\gamma^e_{k+1}}}}.
\end{align*}
Finally, we define the graph gradient operators $\gnabla:\R^{\nodes_n} \to \R^{\edges_n}$ and the graph divergence operators $\gdiv:\R^{\edges_n} \to \R^{\nodes_n}$ by
\begin{align*}
    (\gnabla\upvarphi)_{\sfv\sfw} \coloneqq \upvarphi_\sfw-\upvarphi_\sfv
    \qquad\text{and}\qquad 
    (\gdiv \rmf)_\sfv \coloneqq \sum_{\sfw:\sfv\sfw\in\edges_n} \rmf_{\sfv\sfw} - \sum_{\sfw:\sfw\sfv\in\edges_n} \rmf_{\sfw\sfv}.
\end{align*}
They give rise to the continuity equation 
\begin{align}\label{eq:CE_discrete}
    \dot \upgamma + \gdiv \rmf = 0.
\end{align}
We denote by $\overline\CE_n$ the set of pairs $(\upgamma,\rmf)\in\calP(\nodes_n)\times\R^{\edges_n}$ satisfying \eqref{eq:CE_discrete} in the sense of distributions. Similar to Lemma~\ref{lem:AC:representative}, if $\rmf$ is integrable in time, then $\upgamma\in\AC(0,T;\calP(\nodes_n))$ and since $\nodes_n$ is finite, we can understand \eqref{eq:CE_discrete} in the strong sense.

For each $n\in\bbN$ and each pair $(\upgamma,\rmf)\in\overline\CE_n$, the discrete analogue of the dissipation functional defined in \eqref{eq:def:Mgraph:EDfunctional} then reads as 
\begin{align}\label{eq:def:D_n}
    \calD_n(\upgamma,\rmf) \coloneqq \int_0^T \pra*{\calR_n (\upgamma,\rmf)+ \calI_n(\upgamma)} \dx t \,.
\end{align}
The associated energy dissipation functional of the curve $(\upgamma,\rmf)\in \overline\CE_n$ 
\begin{align}\label{eq:def:L_n}
    \calL_n(\upgamma,\rmf) \coloneqq \calE_n(\upgamma(T))-\calE_n(\upgamma(0)) + \calD_n(\upgamma,\rmf).
\end{align}
For any fixed $n\in\bbN$ this model falls in the class considered in \cite{PeletierRossiSavareTse2022}. 
In particular, from \cite[Section 5]{PeletierRossiSavareTse2022} the following two results follow.
\begin{proposition}[Chain rule inequality]
    For all $n\in\bbN$ and all $(\upgamma,\rmf)\in\overline\CE_n$ such that $\calE_n(\upgamma(0))<\infty$ and $\calD_n(\upgamma,\rmf)<\infty$, it holds
    \begin{align*}
        \calL_n(\upgamma,\rmf) \ge 0.
    \end{align*}
\end{proposition}
\begin{theorem}[Existence of solutions]\label{thm:discrete_existence}
    For each $n\in\bbN$ let $\upgamma_0 \in \calP(\nodes_n)$ such that $\calE_n(\upgamma_0)<\infty$. Then, there exists an ODE solution $\upgamma\in\AC(0,T;\calP(\nodes_n))$ of \eqref{eq:system_intro_micro} with $\upgamma(0) = \upgamma_0$.
    Furthermore, there exists a flux $\rmf:[0,T]\to \R^{\edges_n}$ such that $(\upgamma,\rmf)\in\overline\CE_n$ and $\calL_n(\upgamma,\rmf) = 0$, i.e., $(\upgamma,\rmf)$ is an EDP solution in the sense of Definition~\ref{def:EDP-sol}.
\end{theorem}

\subsection{Statement of EDP convergence and existence of solutions}

So far we have constructed a discrete gradient system satisfying the chain rule inequality and with~\eqref{eq:system_intro_micro} as its gradient flow. Our next goal is to show that families of these discrete gradient systems indeed converge to the gradient system introduced in Section~\ref{sec:MetricGraphGF}, which will directly imply the existence of gradient flow solutions for~\eqref{eq:system_intro_linear}. We begin by constructing appropriate embeddings to connect $\overline\CE_n$ and $\CE$.
\begin{definition}[Embedding]
    For $\upgamma\in\calP(\nodes_n)$, the embedded measures are defined by
    \begin{equation}\label{eq:embedding_discr_to_cont}
    \begin{aligned}
        \rho^e_n(\dx x) &\coloneqq \iota_n \tilde\gamma^e(\dx x) \coloneqq \frac{1}{h^e_n}\sum_{k=1}^n \tilde\gamma^e_k \one_{I^e_{n,k}}(x)\dx x\,,\\
        \gamma_{n,\sfv} &\coloneqq \iota_n \bar\gamma_\sfv \coloneqq \bar\gamma_\sfv \,,
    \end{aligned}
    \end{equation}
    and the embedded fluxes by
    \begin{subequations} \label{eq:micro:embed-flux}
    \begin{align}
        \iota_n \tilde f^e(\dx x) &\coloneqq 
        \int_0^1 \sum_{k=1}^{n-1} \tilde f^e_k\one_{I^e_{n,k+s}}(x) \dx s \dx x
        \label{eq:micro:embed-flux1}\\
        \iota_n \bar f^e_\sfv &\coloneqq \bar f^e_\sfv.
        \label{eq:micro:embed-flux2}
    \end{align}
    \end{subequations}
    Furthermore, on $\tilde\nodes_n$, the dual coarse graining operator $\iota_n^\ast$ for $\varphi\in C(\Medges)$ is defined by
    \begin{align*}
        (\iota_n^\ast \varphi)^e_k \coloneqq \frac{1}{h^e_n}\int_{I^e_{n,k}}\varphi\dx x
    \end{align*}
    Finally, we introduce for $\alpha \le n$ the shift operator $S^\alpha_n$ given for $\varphi^e\in C([0,\ell^e])$ by
    \begin{align}\label{eq:def:shift-operator}
    S^\alpha_n\varphi^e(x)\coloneqq 
    \begin{cases}
        \varphi^e(x+\alpha h^e_n), &x\in [0,h^e_n(n-\alpha)]\\
        \varphi^e(\ell^e), &x\in(h^e_n(n-\alpha),\ell^e]
    \end{cases}
    \end{align}
\end{definition}

\begin{remark}\label{rem:RefMeasureStrong}
    We note that due to Assumption~\ref{ass:Lip_Potential} it follows from standard arguments that for all $k\in\bbN_0$ we have
    \begin{align*}
        S^k_n\iota_n\tilde\omega^e_n \to \pi^e
		\quad\text{ strongly in } L^\infty(0,\ell^e) \text{ for each } e\in \edges \,.
    \end{align*}
\end{remark}

Next, we would like to state that the embedding maps solutions to the discrete continuity equation \eqref{eq:CE_discrete} to those satisfying the continuum continuity equation in the sense of Definition~\ref{def:bCE}.
However, this is generally not true, since for most edge test functions there is mismatch between the boundary value and the average over a short interval near the boundary. 
Therefore, we restrict the statement to edge test functions that are flat near the boundary, thus ensuring that the values coincide.
When passing to the limit $n\to\infty$, we will remove this restriction via a diagonal argument.
\begin{lemma}\label{lem:discrete_CE}
    Let $\delta>0$. We denote
    \begin{equation}\label{eq:def_C1flat}
        C^1_{\textnormal{flat},\delta}(\Mgraph)\coloneqq \set{\Phi\in C^1(\Mgraph): \varphi^e|_{(0,\delta)} \equiv \varphi^e(0), \varphi^e|_{(\ell^e-\delta,\ell^e)} \equiv \varphi^e(\ell^e)\; \forall e\in\edges}
    \end{equation}
    the class of differentiable functions whose metric edge parts are flat near the boundary (and we analogously denote differentiable and/or time dependent functions that are flat near the boundary).
    
    If $n \ge 1/\delta$ and $(\upgamma,\rmf)\in\overline\CE_n$, then $\iota_n(\upgamma,\rmf)$ satisfies \eqref{e:def:bCE:weak} for all $\Phi\in C^1_{\textnormal{flat},\delta}(\Mgraph)$. 
    
    Conversely, if $\iota_n(\upgamma,\rmf)$ satisfies \eqref{e:def:bCE:weak} for all $\Phi\in C^1_{\textnormal{flat},1/n}(\Mgraph)$, then $(\upgamma,\rmf)\in\overline\CE_n$.
\end{lemma}
\begin{proof}
    We recall the shift map introduced in \eqref{eq:def:shift-operator} and observe that due to the differentiability of $\varphi^e$ it holds
    \begin{align}\label{eq:HSDI}
        S^1_n\varphi^e(x)-\varphi^e(x) 
        &= \int_0^1 \partial_x\varphi^e\biggl(x+\frac{s\ell^e}{n}\biggr)h^e_n \dx s 
        = h^e_n\int_0^1 S^s_n\partial_x\varphi^e(x) \dx s
    \end{align}
    With this, we can carry out the following calculation:
    \begin{align*}
        \langle \tilde f,\overline\nabla\iota_n^\ast \varphi \rangle 
        &= \sum_{e\in\edges}\sum_{k=1}^{n-1} \tilde f^e_k \pra[\bigg]{\frac{1}{h^e_n}\int_{I^e_{n,k+1}}\varphi^e\dx x-\frac{1}{h^e_n}\int_{I^e_{n,k}}\varphi\dx x}\\
        &= \sum_{e\in\edges}\sum_{k=1}^{n-1} \tilde f^e_k \int_{I^e_{n,k}} \int_0^1 S^s_n\partial_x\varphi^e(x) \dx s\dx x\\
        & =  \sum_{e\in\edges}\sum_{k=1}^{n-1} \tilde f^e_k \int_0^{\ell^e}\int_0^1\one_{I^e_{n,k}}(x) S^s_n\partial_x\varphi^e(x) \dx s\dx x\\
        &= \sum_{e\in\edges} \int_0^{\ell^e}\partial_x\varphi^e(x) \int_0^1 \sum_{k=1}^{n-1} \tilde f^e_k \one_{I^e_{n,k+s}}(x) \dx s\dx x \,.
    \end{align*}
    Thanks to~\eqref{eq:micro:embed-flux1}, this shows $\langle \tilde f,\overline\nabla\iota_n^\ast \varphi \rangle  = \langle \iota_n \tilde f,\nabla \varphi \rangle_{\sfL}$ where we used the notation~\eqref{eq:def:products}. 
    This yields the equivalence of the formulations in the interior of the metric edges. 
    
    For the boundary and vertex parts we use that if $n\ge 1/\delta$, then $\frac{1}{h^e_n}\int_{I^e_{n,1}}\varphi\dx x = \varphi^e(0)$ (and the same holds at the opposite boundary), which directly yields the first implication.
    
    The converse implication follows by choosing test functions that are constant on $I^e_{n,1}$, but differ between $I^e_{n,1}$ and $I^e_{n,2}$.
\end{proof}

We state the EDP convergence result of this section in the sense Definition~\ref{def:EDP-convergence}.
\begin{theorem}[EDP convergence]\label{thm:EDP_discrete}
    Let $(\upgamma_n,\rmf_n)_{n\in\bbN}$ such that for all $n\in\bbN$ it holds that $(\upgamma_n,\rmf_n)\in\overline\CE_n$ 
    satisfy the uniform bounds 
    \begin{equation}\label{eq:ass:discrete:bdd}
        \sup_{n\in\bbN}\sup_{t\in(0,T)}\calE_n(\upgamma_n(t))<\infty
    	\qquad\text{ and }\qquad
    	\sup_{n\in\bbN}\calD_n(\upgamma_n,\rmf_n)<\infty \,.
    \end{equation}
    Then, the embedding $\iota_n(\upgamma_n,\rmf_n)\in\CE$ is relatively compact (in the sense of Proposition~\ref{prop:discrete_comp}, below) and along every converging subsequence with limit $(\mu,j)\in\CE$%
    , the lower limit inequalities 
    \begin{subequations}
    \label{eq:discrete:lower-limit}
    \begin{align}
        \liminf_{n\to\infty} \calE_n(\upgamma_n)&\ge \calE(\mu) , 
        \qquad\text{for all } t\in[0,T]\,;
        \label{eq:discrete:lower-limit-energy} \\
        \liminf_{n\to\infty} \calD_n(\upgamma_n,\rmf_n) &\ge \calD(\mu,\rmj)
        \label{eq:discrete:lower-limit-dissipation}
    \end{align}
    \end{subequations}
    hold for the functionals defined in \eqref{eq:def:D_n} and \eqref{eq:def:Mgraph:Dfunctional}.

    If in addition the initial data are well-prepared, i.e., if $\calE_n(\upgamma_n(0)) \to \calE(\mu(0))$ as $n\to\infty$, it follows
    \begin{align*}
        \liminf_{n\to\infty} \calL_n(\upgamma_n,\rmf_n) &\ge \calL(\mu,\rmj).
    \end{align*}
\end{theorem}
\begin{proof}
    The relative compactness of the embedded curves is content of Proposition~\ref{prop:discrete_comp}.
    The lower limit of energies and dissipation functionals is shown in Proposition~\ref{prop:lsc_dissipation_discrete}.
\end{proof}

\begin{corollary}[Existence of solutions]
    Let $\mu^0\in\calP(\Mgraph)$ be such that $\calE(\mu^0)<\infty$. Then, there exists an EDP solution (in the sense of Definition~\ref{def:EDP-sol}) to the system \eqref{eq:system_intro_linear} with initial datum $\mu^0$, where the corresponding gradient system in continuity equation format (in the sense of Definition~\ref{def:GradSystCE}) is $(\nodes\times\Medges,(\nodes\times\edges)\times\Medges,\bnabla,\calE,\calR^\ast)$ as introduced in Section~\ref{sec:MetricGraphGF}.
\end{corollary}
\begin{proof}
    The result follows directly from Theorem~\ref{thm:discrete_existence} and Theorem~\ref{thm:EDP_discrete} upon constructing well-prepared initial data.

    For this, we use the simple construction $(\tilde\gamma_n^0)^e_k \coloneqq \rho^0(I^e_{n,k})$ and $(\bar\gamma_n^0)^e_\sfv \coloneqq (\gamma^0)^e_\sfv$.
    The constructed family of initial data indeed satisfies $\limsup_{n\to\infty}\calE_n(\upgamma_n^0)\leq \calE(\mu^0)$. 
    For the terms containing $(\bar\gamma_n^0)^e_\sfv$, this is immediate. 
    For the edge terms we apply Jensen's inequality to the convex function $\lambda_B(r) = r\log r-r+1$ to obtain for $e\in\edges$ and $k=1,\ldots,n$
    \begin{align*}
        \lambda_B\bra[\bigg]{\frac{(\tilde\gamma^0_n)^e_k}{\tilde\omega^e_{n,k}}}\tilde\omega^e_{n,k}
        &= \lambda_B\bra[\bigg]{\frac{1}{\pi(I^e_{n,k})}\int_{I^e_{n,k}}\frac{\dx\rho^0}{\dx\pi}\dx\pi}\pi(I^e_{n,k})
        \le \int_{I^e_{n,k}}\!\! \lambda_B\bra[\Big]{\frac{\dx\rho^0}{\dx\pi}}\dx \pi.
    \end{align*}
    A summation of these terms yields the limsup-estimate. The full convergence of initial data then follows from \eqref{eq:discrete:lower-limit-energy}.
\end{proof}

\subsection{Proof of the EDP limit}

The compactness result of this section is as follows.
\begin{proposition}[Compactness after embedding]\label{prop:discrete_comp}
    Consider $(\upgamma_n,\rmf_n)_{n\in\bbN}$ such that for all $n\in\bbN$ it holds $(\upgamma_n,\rmf_n)\in\overline\CE_n$ and the uniform bounds~\eqref{eq:ass:discrete:bdd} 
    are satisfied.
    Then, there exist $(\mu,\rmj)\in\CE$ such that (abusing notation by using the same symbols for measures and densities) we have 
    \begin{enumerate}[label=(\roman*)]
        \item\label{prop:discrete_comp:1} $\iota_n \tilde f_n \rightharpoonup j$ weakly in $L^1([0,T]\times\Medges)$;
        \item\label{prop:discrete_comp:2} $\iota_n \bar f_n \rightharpoonup \bar\jmath$ weakly in $L^1([0,T]\times\nodes)$;
        \item\label{prop:discrete_comp:3newnew} $\iota_n\upgamma_n(t) \rightharpoonup \mu(t)$ narrowly in $\calP(\Mgraph)$ for all $t\in[0,T]$;
        \item\label{prop:discrete_comp:3} $\iota_n\upgamma_n \to \mu$ strongly in $L^1([0,T]\times\Mgraph)$;
        \item\label{prop:discrete_comp:4new} For all $e=\sfv\sfw\in\edges$ it holds $\frac{\tilde\gamma^e_{n,1}}{\tilde\omega^e_{n,1}}\rightharpoonup\frac{\rho^e}{\pi^e}\big|_\sfv$ and $\frac{\tilde\gamma^e_{n,n}}{\tilde\omega^e_{n,n}}\rightharpoonup\frac{\rho^e}{\pi^e}\big|_\sfw$ weakly in $L^1(0,T)$;
        \item\label{prop:discrete_comp:4} $\nabla_n \iota_n\sqrt{\tilde\gamma_n/\tilde\omega_n}\rightharpoonup \nabla \sqrt{\rho/\pi}$ weakly in $L^2([0,T]\times\Medges)$;
    \end{enumerate}
    where $\nabla_n:PC(\Medges)\to PC(\Medges)$ is defined by $\nabla_n \varphi^e(x) \coloneqq \one_{[0,(n-1)h^e_n]}\frac{\varphi^e(x+h^e_n)-\varphi^e(x)}{h^e_n}$ for $PC(\Medges)$ denoting the piecewise continuous functions on $\Medges$.
\end{proposition}
\begin{proof}
    The proof is split into Lemma~\ref{lem:finite_flux_comp}, Lemma~\ref{lem:strong_compactness}, and Lemma~\ref{lem:comp_fin_diff}, below.
\end{proof}

As an intermediate step, we observe the following spatial regularity and improved spatio-temporal integrability. 
The proof is essentially a simplification of \cite[Proposition~4.6]{HraivoronskaTse2023}.

\begin{lemma}[Spatial regularity and higher integrability]\label{lem:spat_reg}
    For $n\in\bbN$ denote $\pi_n = \iota_n \tilde\omega_n$ and $u_n = \frac{\dx\rho_n}{\dx\pi_n}$.
    The sequence $(u_n)_n$ is uniformly bounded in $L^1([0,T];\rmB\rmV(\Medges))$ and in $L^{2}([0,T]\times\Medges)$.
\end{lemma}
\begin{proof}
    We begin our considerations with the following observation: Fix $e\in\edges$, $\delta>0$, and constantly extend $u^e_n$ to $[-\delta,\ell^e+\delta)$. For $x\in[0,\ell^e]$, $y\in[-\delta,\delta]$ define the set of indices
    \begin{align*}
        \mathrm{Int}^e(x,y) \coloneqq \set{k\in\set{1,\ldots,n{-}1}: k h^e_n\in [x-y,x]}.
    \end{align*}
    From the opposite perspective, for $y\in[0,\delta)$ and $k\in\set{1,\ldots,n{-}1}$ define the set of points $x\in[0,\ell^e]$ such that the interval between $x{-}y$ and $x$ contains $k h^e_n$, the boundary between $I^e_{n,k}$ and $I^e_{n,k+1}$,
    \begin{align*}
        \widetilde{\operatorname{Int}}^e(y,k) \coloneqq \set[\big]{x\in [0,\ell^e]: k h^e_n\in [x-y,x]}.
    \end{align*} 
    Then, recalling that $u^e_{n,k} = \tilde\gamma^e_{n,k}/\tilde\omega^e_{n,k}$,
    it holds for all $x\in[0,\ell^e]$, $y\in[-\delta,\delta]$
    \begin{equation}
    \label{eq:cylinder_estimate}
    \begin{aligned}
        \abs[\big]{u^e_n(x-y)-u^e_n(x)} 
        &\le \sum_{k=1}^{n-1}\abs[\bigg]{\frac{\tilde\gamma^e_{n,k+1}}{\tilde\omega^e_{n,k+1}}-\frac{\tilde\gamma^e_{n,k}}{\tilde\omega^e_{n,k}}}\one_{\mathrm{Int}^e(x,y)}(k)
        = \sum_{k=1}^{n-1}\abs[\bigg]{\frac{\tilde\gamma^e_{n,k+1}}{\tilde\omega^e_{n,k+1}}-\frac{\tilde\gamma^e_{n,k}}{\tilde\omega^e_{n,k}}}\one_{\widetilde{\mathrm{Int}}^e(y,k)}(x).
    \end{aligned}
    \end{equation}
    Note that $\widetilde{\mathrm{Int}}^e(y,k)$ is an interval of length $\abs{y}$ and hence integrating the above inequality w.r.t. $x,t$, we obtain
    \begin{align*}
        \MoveEqLeft \int_0^T\int_0^{\ell^e}\abs[\big]{u^e_n(x-y)-u^e_n(x)}\dx x\dx t
        \le \abs{y}\int_0^T \sum_{k=1}^{n-1}\abs[\bigg]{\frac{\tilde\gamma^e_{n,k+1}}{\tilde\omega^e_{n,k+1}}-\frac{\tilde\gamma^e_{n,k}}{\tilde\omega^e_{n,k}}}\dx t\\
        &\le C_\pi\abs{y}\int_0^T \sum_{k=1}^{n-1}n\abs[\bigg]{\frac{\tilde\gamma^e_{n,k+1}}{\tilde\omega^e_{n,k+1}}-\frac{\tilde\gamma^e_{n,k}}{\tilde\omega^e_{n,k}}} \sqrt{\tilde\omega^e_{n,k}\tilde\omega^e_{n,k+1}}\dx t\\
        &\le \widetilde C_\pi\abs{y}\sqrt{T\norm{u^e_n}_{L^\infty([0,T];L^1([0,\ell^e])}}\bra[\bigg]{\int_0^T\sum_{k=1}^{n-1}\abs[\bigg]{\sqrt{\frac{\tilde\gamma^e_{n,k+1}}{\tilde\omega^e_{n,k+1}}}-\sqrt{\frac{\tilde\gamma^e_{n,k}}{\tilde\omega^e_{n,k}}}}^2 \sigma^e_n\bra{\tilde\omega^e_{n,k}\tilde\omega^e_{n,k+1}}\dx t}^{\frac{1}{2}}\\
        &\le C\sqrt{T}\abs{y},
    \end{align*}
    where we used that $\pi^e$ is comparable to the Lebesgue measure on $[0,\ell^e]$ and the uniform bounds on energy and the discrete Fisher information.
    In particular, for every $\varphi\in C^1([0,T]\times[0,\ell^e])$ we can approximate $\partial_x\varphi$ by difference quotients to obtain
    \begin{align*}
        \int_0^T\int_0^{\ell^e} u^e_n \partial_x\varphi \dx x\dx t 
        \le  C\sqrt{T}\norm{\varphi}_{L^\infty},
    \end{align*}
    i.e., $u_n$ is uniformly bounded in $L^1\bigl([0,T];\rmB\rmV(\Medges)\bigr)$ and hence also in $L^1\bigl([0,T];L^{\infty}(\Medges)\bigr)$.
    Since the energy control also yields uniform boundedness in $L^\infty\bigl([0,T];L^1(\Medges)\bigr)$, an interpolation also leads the uniform bound in $L^{2}([0,T]\times\Medges)$.
\end{proof}
Next, we prove statements~\ref{prop:discrete_comp:1} and~\ref{prop:discrete_comp:2} from Proposition~\ref{prop:discrete_comp}.
\begin{lemma}[Compactness of fluxes]\label{lem:finite_flux_comp}
    Consider $(\upgamma_n,\rmf_n)_{n\in\bbN}$ such that for all $n\in\bbN$ it holds $(\upgamma_n,\rmf_n)\in\overline\CE_n$ and the uniform bounds~\eqref{eq:ass:discrete:bdd} 
    are satisfied.
    Then, there exists $\rmj =(\bar\jmath,j) \in L^1([0,T]\times \Medges)\times L^1([0,T]\times \nodes)$ such that along a (not renamed) subsequence $\iota_n \tilde f_n \rightharpoonup j$ weakly in $L^1([0,T]\times\Medges)$ and $\iota_n \bar f_n \rightharpoonup \bar\jmath$ weakly in $L^1([0,T]\times\nodes)$.
\end{lemma}
\begin{proof}
    We argue as in the proof of \cite[Proposition~5.5]{heinze2025discrete}.
    In particular, we rely on the fact that for all $s\in\bbR$, $r\ge 0$, and $q>1$ it holds
    \begin{align}\label{eq:magical}
        \sfC(s) \le \frac{q}{q{-}1} \sfC(s|r) + \frac{4\,r^q}{q{-}1}
    \end{align}
    (cf. e.g. \cite[Eq. (2.7)]{fischer2022global} or \cite[Appendix A]{heinze2025discrete}).
    With this, we obtain for $e\in\edges$
    \begin{align*}
        \frac{1}{h^e_n}\sum_{k=1}^{n-1}\int_0^T \sfC\bra[\Big]{ &\tilde f^e_{n,k}}\dx t
        \le
        \frac{1}{h^e_n}\sum_{k=1}^{n-1}\int_0^T \bra[\Bigg]{\frac{q}{ q{-}1}  \sfC\bra[\bigg]{\tilde f^e_{n,k} \bigg| \frac{\sqrt{\tilde \gamma^e_{n,k}\tilde \gamma^e_{n,k+1}}}{h^e_n}} 
        + \frac{4}{q {-}1 } \bra[\bigg]{\frac{\tilde \gamma^e_{n,k}}{h^e_n}\frac{\tilde \gamma^e_{n,k+1}}{h^e_n}}^{q/2}} \dx t\\
        & \leq  C_{q,\pi,\Medges} \int_0^T \bra[\bigg]{ \sum_{k=1}^{n-1}\sfC\big( \tilde f^e_{n,k} \big| (h^e_n)^{-2}(\tilde \gamma^e_{n,k}\tilde \gamma^e_{n,k+1})^{q/2} \big) + \int_0^{\ell^e-h^e_n}(u^e_n S^1_nu^e_n)^{q/2}\dx x} \dx t,
    \end{align*} 
    where for the second inequality we used that $\pi^e$ is comparable to the Lebesgue measure on $[0,\ell^e]$ and that for $r \ge 0$ the map $(0,\infty)\ni\lambda\mapsto \lambda^2 \sfC(r/\lambda)$ is increasing.
    The first term on the right-hand side is bounded by the dissipation, while the second term is bounded for $q\in(1,2]$ due to Lemma~\ref{lem:spat_reg}.
    
    Upon realizing that $\sum_{k=1}^{n-1}\int_0^1 \one_{I^e_{n,k+s}}(x)\dx s =(n-1)h^e_n\approx\ell^e$, the above estimate and an application of Jensen's inequality also yields the $n$-uniform bound
    \begin{align*}
        \sup_{n\in \bbN}\int_0^T\int_0^{\ell^e} \sfC\bra[\bigg]{\frac{\dx(\iota_n\tilde f^e)}{\dx\scrL^2}}\dx x\dx t <\infty.
    \end{align*}
    Since $\sfC$ is superlinear, this implies that the densities $\frac{\dx(\iota_n\tilde f^e)}{\dx \scrL^2}$ are uniformly integrable and hence (along a not renamed subsequence) have a weak limit in $L^1([0,T]\times[0,\ell^e])$.

    The weak compactness for $\iota_n\bar f_n$ follows similarly, keeping in mind that $\nodes$ is a finite space.
\end{proof}

\begin{lemma}[Compactness of curves and continuity equation]\label{lem:strong_compactness}
    Let $(\upgamma_n,\rmf_n)_{n\in\bbN}$ such that for all $n\in\bbN$ it holds $(\upgamma_n,\rmf_n)\in\overline\CE_n$ and the uniform bounds~\eqref{eq:ass:discrete:bdd}
    are satisfied.
    
    Then, there exists $\mu\in L^1([0,T]\times\Mgraph)$ s.t. along a (not renamed) subsequence $\iota_n\upgamma_n\to \mu$ strongly in $L^1([0,T]\times\Mgraph)$ and for a.e. $t\in[0,T]$ we have $\iota_n\upgamma_n(t) \rightharpoonup\mu(t)$ narrowly in $\calP(\Mgraph)$.

    Furthermore, denoting $\rmj= (\bar\jmath,j)$ the limit obtained in Lemma~\ref{lem:finite_flux_comp} (w.l.o.g. along the same subsequence), we have $(\mu,\rmj)\in\CE$.
\end{lemma}
\begin{proof}
    The pointwise in time narrow in space compactness follows from the bounds established in Lemma~\ref{lem:finite_flux_comp} (see also the proof of \cite[Lemma 4.5]{HraivoronskaTse2023}).
    
    The strong $L^1$-convergence follows as an application of the Aubin-Lions-type argument \cite[Proposition 1.10]{rossi2003tightness} by combining the time regularity obtained in Lemma~\ref{lem:finite_flux_comp} and the spatial regularity obtained in Lemma~\ref{lem:spat_reg} (see also the proof of \cite[Theorem 4.8]{HraivoronskaTse2023}).
    
    To show that the limit satisfies the continuity equation, let $\Phi: \Medges\to \R$ be Lipschitz and for $\delta>0$ sufficiently small define $\Phi_\delta\in C^1_{\textnormal{flat},\delta}(\Mgraph)$ by $\varphi^e_\delta(x) \coloneqq \varphi^e(x)\tilde\chi_\delta(x)$ for some $\tilde\chi_\delta\in C^\infty([0,\ell^e];[0,1])$ with $\tilde\chi_\delta|_{[0,\delta)\cup (\ell^e-\delta,\ell^e]} \equiv 0$, $\tilde\chi_\delta|_{[2\delta,\ell^e-2\delta]} \equiv 1$, and which is monotone on the between intervals.

    In particular, for every $e\in\edges$, $x\in\set{0,\ell^e}$, we have
    \begin{align*}
        \abs{\varphi^e(x) - \varphi^e_\delta(x)}\le C_\varphi \delta \overset{\delta\to 0}{\longrightarrow} 0.
    \end{align*}    
    Regarding the dual product in the interior, we have
    \begin{align*}
        \int_0^T\int_0^{\ell^e} \nabla(\varphi^e-\varphi^e_\delta)\cdot\dx j^e_n \le \norm{\nabla\varphi}_{L^\infty([0,T]\times\Medges)}\abs{j^e_n}([0,T]\times([0,2\delta]\cup[\ell^e-2\delta,\ell^e])) \overset{\delta\to 0}{\longrightarrow} 0,
    \end{align*}
    uniformly in $n$ due to Lemma~\ref{lem:finite_flux_comp}.
    With this, the claim follows from Lemma~\ref{lem:discrete_CE}.    
\end{proof}

\begin{lemma}[Compactness of finite differences and traces]\label{lem:comp_fin_diff}
    Let $(\upgamma_n,\rmf_n)_{n\in\bbN}$ such that for all $n\in\bbN$ it holds $(\upgamma_n,\rmf_n)\in\overline\CE_n$ and the uniform bounds~\eqref{eq:ass:discrete:bdd}
    are satisfied. Let $(\mu,j)\in\CE$ be the limit of $\iota_n(\upgamma_n,\rmf_n)$.
    Let the square root densities be given by
    \[ 
    	\nu_n \coloneqq \sum_{e\in\edges}\sum_{k=1}^n \sqrt{\frac{\tilde\gamma^e_{n,k}}{\tilde\omega^e_{n,k}}} \one_{I^e_{n,k}}
    	\quad\text{and}\quad
    	\nu \coloneqq \sum_{e\in\edges} \sqrt{\frac{\rho^e}{\pi^e}}
    \]
   	as well as the difference quotient operators mapping
   	\[
   		\nabla_n:PC(\Medges)\to PC(\Medges)
   		\quad\text{with}\quad
   		\nabla_n \varphi^e(x) \coloneqq \one_{[0,(n-1)h^e_n]}\frac{\varphi^e(x+h^e_n)-\varphi^e(x)}{h^e_n} \,,
   	\] 
    where $PC(\Medges)$ denotes the piecewise continuous functions on $\Medges$.
    Then, $\nu^e\in L^2(0,T;H^1(0,\ell^e))$ and (along a subsequence) $\nu^e_n \rightharpoonup \nu^e$ strongly in $L^2([0,T]\times[0,\ell^e])$ and $\nabla_n \nu^e_n\rightharpoonup \nabla \nu^e$ weakly in $L^2([0,T]\times[0,\ell^e])$.

    Additionally, for all $e=\sfv\sfw\in\edges$ it holds (along the same subsequence) that $\frac{\tilde\gamma^e_{n,1}}{\tilde\omega^e_{n,1}}\rightharpoonup\frac{\rho^e}{\pi^e}\big|_\sfv$ and $\frac{\tilde\gamma^e_{n,n}}{\tilde\omega^e_{n,n}}\rightharpoonup\frac{\rho^e}{\pi^e}\big|_\sfw$ weakly in $L^1(0,T)$.
\end{lemma}
\begin{proof}
Recall that by Assumption~\ref{ass:Lip_Potential} the family $(\tilde\omega^e_n)_n$ is uniformly bounded away from zero and that by Remark~\ref{rem:RefMeasureStrong} $\iota_n\tilde\omega^e_n\to \pi^e$ strongly $L^\infty(0,\ell^e)$ as $n\to \infty$. 
On the other hand, Lemma~\ref{lem:strong_compactness} in particular implies strong $L^2$-convergence of $\sqrt{\rho_n}$ to $\sqrt{\rho}$ along a not renamed subsequence.
Hence, along the same subsequence we have $\nu_n \to \sqrt{\frac{\rho}{\pi}} =\nu$ strongly in $L^2([0,T]\times\Medges)$.

Regarding the finite differences, observe that 
\begin{align*}
    \sum_{e\in\edges}\sum_{k=1}^{n-1} \bra[\bigg]{\frac{\sqrt{\mathstrut\smash{\tilde\gamma^e_{n,k}/\tilde\omega^e_{n,k}}} - \sqrt{\mathstrut\smash{\tilde\gamma^e_{n,k+1}/\tilde\omega^e_{n,k+1}}}}{h^e_n}}^{\!2}
    &= \sum_{e\in\edges} \int_0^{(n-1)h^e_n} \bra*{\frac{\nu^e_n(x)-\nu^e_n(x+h^e_n)}{h^e_n}}^2\dx x\\
    &= \sum_{e\in\edges} \int_0^{\ell^e} \abs[\big]{-\nabla_n \nu^e_n(x)}^{2} \!\dx x.
\end{align*}
We thus obtain from the a priori bound~\eqref{eq:ass:discrete:bdd} that $\int_0^T
\norm{\nabla_n \nu^e_n}_{L^2([0,T]\times[0,\ell^e])} 
\le C<\infty$ for all $e\in\edges$ and (along a not renamed subsequence of the previous subsequence) that $\nabla_n \nu^e_n \rightharpoonup v$ weakly in $L^2([0,T]\times[0,\ell^e])$ for some $v \in L^2([0,T]\times[0,\ell^e])$. 
This $v$ is the weak gradient of $\nu^e$. 
Indeed, for all $\varphi\in C_c^\infty([0,T]\times(0,\ell^e))$ we have $\one_{[h^e_n,\ell^e]}(x)\bra[\Big]{\frac{S^{-1}_n\varphi-\varphi}{h^e_n}} \to -\partial_x\varphi$ and $\one_{[0,\ell^e-h^e_n]}\varphi\to \varphi$ strongly in $L^2([0,T]\times[0,\ell^e])$ as $n\to\infty$. Therefore, it holds (along the above subsubsequence) for every $e\in\edges$
\begin{align*}
    \int_0^T\int_0^{\ell^e} v^e(t,x) \varphi(t,x)\dx x \dx t &= \lim_{n\to\infty} \int_0^T\int_0^{\ell^e-h^e_n}\bra*{\frac{\nu^e_n(t,x+h^e_n)-\nu^e_n(t,x)}{h^e_n}}\varphi(t,x)\dx x \dx t\\
    &= \lim_{n\to\infty} \int_0^T\int_{h^e_n}^{\ell^e}\nu^e_n(t,x)\bra*{\frac{\varphi(t,x-h^e_n)-\varphi(t,x)}{h^e_n}}\dx x \dx t\\
    &= -\int_0^T\int_0^{\ell^e}\nu^e(t,x)\partial_x\varphi(t,x)\dx x \dx t.
\end{align*}
For the convergence towards the traces, recall from Lemma~\ref{lem:strong_compactness} that $\rho^n\rightharpoonup\rho$ weakly in $L^1(0,T;W^{1,1}(\Medges))$. 
In particular, we have for each $e\in\edges$, $\sfv\in\nodes(e)$ that $\rho_n^e|_\sfv\rightharpoonup \rho^e|_\sfv$ weakly in $L^1(0,T)$. 
On the other hand, by definition it holds for each $e=\sfv\sfw\in\edges$, $n\in\bbN$ that $\rho_n^e|_\sfv = \tilde\gamma^e_{n,1}/h^e_n$ and $\rho_n^e|_\sfv = \tilde\gamma^e_{n,n}/h^e_n$.

Since the reference measures are continuous, we further have $\omega^e_{n,1}/h^e_n = \tfrac{1}{h^e_n}\int_{\smash[b]{I^e_{n,1}}}\pi^e\dx x \to \pi^e(0)$ as well as $\omega^e_{n,n}/h^e_n = \pi^e(\ell^e)$, and their strict positivity allows us to divide by them to obtain the claimed convergence of quotients.
\end{proof}

Having established compactness, next we prove the corresponding lower limit inequality, which is the second step in establish the EDP convergence result (see Definition~\ref{def:EDP-convergence}).

\begin{proposition}[Lower limit inequality]
\label{prop:lsc_dissipation_discrete}
    Let $(\upgamma_n,\rmf_n)_{n\in\bbN}$ such that for all $n\in\bbN$ it holds $(\upgamma_n,\rmf_n)\in\overline\CE_n$ and the uniform bounds~\eqref{eq:ass:discrete:bdd}
    are satisfied. 
    Let $(\mu,j)\in\CE$ be a limit of $\iota_n(\upgamma_n,\rmf_n)$ in the sense of Proposition~\ref{prop:discrete_comp}.
    Then the lower limit inequalities~\eqref{eq:discrete:lower-limit} hold.
\end{proposition}
\begin{proof}
    The lower limit of the energies follows by Fatou's Lemma. 
    Indeed, for all $t\in[0,T]$ we have
    \begin{equation*}
       \liminf_{n\to\infty} \calE_n(\upgamma_n(t)) 
       = 
       \liminf_{n\to\infty} \calH(\upgamma_n(t)|\upomega_n) 
       = \liminf_{n\to\infty} \calH(\iota_n\upgamma_n(t)|\iota_n\upomega_n)
       \ge  \calH(\mu(t)|(\omega,\pi))  \,.
    \end{equation*}
    For the dissipation we consider the relaxed slope and the rate individually, showing the lower limit inequality for the former in Lemma~\ref{lem:lsc_slope_discrete} and for the latter in Lemma~\ref{lem:lsc_rate_discrete}.
\end{proof}

\begin{lemma}\label{lem:lsc_slope_discrete}
    Let $(\upgamma_n,\rmf_n)_{n\in\bbN}$ such that for all $n\in\bbN$ it holds $(\upgamma_n,\rmf_n)\in\overline\CE_n$ and the uniform bounds $\sup_{n\in\bbN}\sup_{t\in(0,T)}\calE_n(\upgamma_n(t))<\infty$ and $\sup_{n\in\bbN}\calD_n(\upgamma_n,\rmf_n)<\infty$ are satisfied. 
    Let $(\mu,j)\in\CE$ be a limit of $\iota_n(\upgamma_n,\rmf_n)$ in the sense of Proposition~\ref{prop:discrete_comp}.
    Then, it holds
    \begin{align*}
        \liminf_{n\to\infty}\int_0^T \calI_n(\upgamma_n)\dx t
        &\ge \int_0^T\calI(\mu)\dx t.
    \end{align*}
\end{lemma}
\begin{proof}
    We consider first the internal edge parts. 
    Since for all $e\in\edges$ we have by Remark~\ref{rem:RefMeasureStrong} that $\bra[\big]{\iota_n\tilde\omega^e_n S^1_n\iota_n\tilde\omega^e_n}^{1/2} \to \pi^e$ strongly in $L^\infty(0,\ell^e)$ and by Lemma~\ref{lem:comp_fin_diff} we have $\nabla_n \nu^e_n\rightharpoonup \nabla \nu^e$ weakly in $L^2([0,T]\times[0,\ell^e])$, the lower semicontinuity statement \cite[Theorem 2.3.1]{Buttazzo1989} applies and yields for all $e\in\edges$
    \begin{align*}
        \liminf_{n\to\infty}\int_0^T \sum_{k=1}^{n-1} \sfI^e_{n,k}(\tilde\gamma^e_{n,k},\tilde\gamma^e_{n,k+1})\dx t
        &= \liminf_{n\to\infty} 2\diffedge^e \int_0^T\int_0^{\ell^e}\bra[\big]{\iota_n\tilde\omega^e_n S^1_n\iota_n\tilde\omega^e_n}^{1/2}(x)\abs*{\nabla_n\nu^e_n(x)}^2\dx x\dx t\\
        &\ge 2\diffedge^e \int_0^T\int_0^{\ell^e} \abs[\big]{\nabla \nu^e}^{2}\dx\pi^e\dx t = \int_0^T \sfI^e(\rho^e)\dx t.
    \end{align*}
    The lower limit for the edge-vertex transition slopes $\sfI^e_\sfv$ directly follows from \cite[Theorem 2.3.1]{Buttazzo1989} in light of Proposition~\ref{prop:discrete_comp}~\ref{prop:discrete_comp:4new}.
\end{proof}

For the lower limit of the rate, we follow the proof strategy of \cite[Theorem 6.2 (i)]{HraivoronskaTse2023}, first establishing a $\limsup$-estimate for the dual dissipation potentials with sufficiently smooth test functions.
\begin{lemma}\label{lem:limsup_Rast_discrete}
    Let $(\upgamma_n,\rmf_n)_{n\in\bbN}$ such that for all $n\in\bbN$ it holds $(\upgamma_n,\rmf_n)\in\overline\CE_n$ and the uniform bounds~\eqref{eq:ass:discrete:bdd}
    are satisfied. 
    Let $(\mu,j)\in\CE$ be a limit of $\iota_n(\upgamma_n,\rmf_n)$ in the sense of Proposition~\ref{prop:discrete_comp}.
    Then, for every $e\in\edges$ and $\varphi^e \in C^1([0,\ell^e])$ it holds
    \begin{align*}
        \limsup_{n\to\infty}\sum_{k=1}^{n-1}(\sfR^e_{n,k})^\ast(\tilde\gamma_n,\gnabla\iota_n^\ast\varphi) \le \frac{\diffedge^e}{2} \int_0^{\ell^e}\rho^e \abs[\big]{\partial_x\varphi^e}^2\dx x.
    \end{align*}
\end{lemma}
\begin{proof}
    Let $\varphi^e \in C^1([0,\ell^e])$ and denote by $\eta$ a modulus of continuity of $\nabla\varphi^e$. Then, for $k=1,\ldots,n-1$ we argue similar to \eqref{eq:HSDI} to obtain
    \begin{align*}
        \abs*{(\iota^*_n \varphi^e)_{k+1} - (\iota^*_n\varphi^e)_k} 
        &\leq \abs[\bigg]{ \int_{I^e_{n,k}}\partial_x\varphi(x)\dx x } + h^e_n\eta\bra{h^e_n}
        = h^e_n\bra[\big]{\abs*{\iota^*_n(\partial_x\varphi^e)_k} + \eta\bra{h^e_n}}.
    \end{align*}
    Furthermore, observe that for each non-negative $\psi^e \in C([0,\ell^e],[0,\infty))$ with modulus of continuity~$\eta$, it holds
    \begin{align*}
        \sum_{k=1}^{n-1} \tilde\gamma^e_{n,k+1} (\iota_n^\ast\psi)_k 
        &= \sum_{k=1}^{n-1} \tilde\gamma^e_{n,k+1} \frac{1}{h^e_n}\int_{I^e_{n,k}}\psi(x)\dx x
        \le \sum_{k=2}^n \tilde\gamma^e_{n,k} \frac{1}{h^e_n}\int_{I^e_{n,k}}\psi(x)\dx x + \eta(h^e_n)\\
        &\le \sum_{k=1}^n \tilde\gamma^e_{n,k} (\iota_n^\ast\psi)_k + \eta(h^e_n)
    \end{align*}
    Therefore, for $\iota_n\tilde\gamma^e_n \rightharpoonup^\ast \rho^e$ and $\varphi\in C^1$, Jensen's inequality together with a Taylor expansion of $\sfC^*(r)=\tfrac{r^2}{2} + o(r^2)$ around $0$ yields the upper estimate
    \begin{align*}
        \sum_{k=1}^{n-1}(\sfR^e_{n,k})^\ast(\tilde\gamma_n,\gnabla\iota_n^\ast\varphi) &=\sum_{k=1}^{n-1}\sigma^e_n\bra{\tilde\gamma^e_{n,k},\tilde\gamma^e_{n,k+1}}\sfC^\ast((\gnabla\iota_n^\ast\varphi^e)_k)\\
        &= \sum_{k=1}^{n-1}\sigma^e_n\bra{\tilde\gamma^e_{n,k},\tilde\gamma^e_{n,k+1}}\sfC^\ast\bra*{\abs*{(\iota_n^\ast \varphi^e)_{k+1}-(\iota_n^\ast \varphi^e)_k}}\\
        &\leq \sum_{k=1}^{n-1} \frac{\diffedge^e}{(h^e_n)^2} \sqrt{\tilde\gamma^e_{n,k}\tilde\gamma^e_{n,k+1}}\sfC^\ast\bra[\big]{h^e_n\bra{\abs*{\iota^*_n(\partial_x\varphi^e)_k} + \eta\bra{h^e_n}}}\\
        &= \frac{\diffedge^e}{2}\sum_{k=1}^{n-1}\sqrt{\tilde\gamma^e_{n,k}\tilde\gamma^e_{n,k+1}}\bra[\big]{\abs[\big]{\abs*{\iota^*_n(\partial_x\varphi^e)_k} + \eta\bra{h^e_n}}^2+ o(1)} \\ %
        &\le \frac{\diffedge^e}{2}\sum_{k=1}^{n-1}\frac{\tilde\gamma^e_{n,k}+\tilde\gamma^e_{n,k+1}}{2}\abs[\big]{\iota^*_n(\partial_x\varphi^e)_k }^2 +o(1)\\
        &\le \frac{\diffedge^e}{2}\sum_{k=1}^{n-1}\frac{\tilde\gamma^e_{n,k}+\tilde\gamma^e_{n,k+1}}{2}\iota^*_n\bra[\Big]{\abs[\big]{\partial_x\varphi^e}^2}_k +o(1)\\
        &\le \frac{\diffedge^e}{2}\sum_{k=1}^n\tilde\gamma^e_{n,k}\iota^*_n\bra[\Big]{\abs[\big]{\partial_x\varphi^e}^2}_k +o(1)\\
        &= \frac{\diffedge^e}{2} \int_0^{\ell^e}\iota_n\tilde\gamma^e \abs[\big]{\partial_x\varphi^e}^2\dx x + o(1)
    \end{align*}
    By letting $n\to \infty$, we get the upper limit bound by $(\sfR^e)^\ast(\rho,\nabla\varphi)=\frac{\diffedge^e}{2} \int_0^{\ell^e}\rho^e \abs[\big]{\partial_x\varphi^e}^2\dx x$ as defined in~\eqref{eq:Mgraph:Re*}.
\end{proof}
Now, we are able to prove the lower limit inequality for the rate as second part for the proof of Proposition~\ref{prop:lsc_dissipation_discrete}.
\begin{lemma}\label{lem:lsc_rate_discrete}
    Let $(\upgamma_n,\rmf_n)_{n\in\bbN}$ such that for all $n\in\bbN$ it holds $(\upgamma_n,\rmf_n)\in\overline\CE_n$ and the uniform bounds~\eqref{eq:ass:discrete:bdd}
    are satisfied. 
    Let $(\mu,j)\in\CE$ be a limit of $\iota_n(\upgamma_n,\rmf_n)$ in the sense of Proposition~\ref{prop:discrete_comp}.
    Then, it holds
    \begin{align*}
        \liminf_{n\to\infty}\int_0^T\calR_n (\upgamma_n,\rmf_n)\dx t \ge \int_0^T\calR(\mu,\rmj)\dx t. 
    \end{align*}
\end{lemma}
\begin{proof}
    By the duality of $\sfC$ and $\sfC^\ast$, we have
    \begin{align*}
        \sum_{k=1}^{n-1} (\gnabla\iota_n^\ast\varphi^e)_k \tilde f^e_{n,k}
        &\le \sum_{k=1}^{n-1} \pra[\big]{\sigma^e_n\bra{\tilde\gamma^e_{n,k},\tilde\gamma^e_{n,k+1}}\sfC^\ast((\gnabla\iota_n^\ast\varphi^e)_k) + \sfC\bra{\tilde f^e_{n,k}|\sigma^e_n\bra{\tilde\gamma^e_{n,k},\tilde\gamma^e_{n,k+1}}}}\\
        &= \sum_{k=1}^{n-1} \pra[\big]{(\sfR^e_{n,k})^\ast(\tilde\gamma_n,\gnabla\iota_n^\ast\varphi) + \sfR^e_{n,k}(\tilde\gamma_n,\tilde f_n)}
    \end{align*}
    Combining these estimates, the limits $j^e_n \rightharpoonup^\ast j^e$ and $\rho^e_n\rightharpoonup^\ast \rho^e$ yield for each $\varphi^e \in C^1([0,T]\times[0,\ell^e])$ 
    \begin{align*}
        &\int_0^T\pra[\bigg]{\int_0^{\ell^e} \partial_x\varphi \dx j^e  - (\sfR^e)^\ast(\rho^e,\partial_x\varphi)} \dx t \\
        &\qquad\le \lim_{n\to\infty} \int_0^T\int_0^{\ell^e}\partial_x\varphi \dx j^e_n \dx t - \limsup_{n\to\infty} \int_0^T \sum_{k=1}^{n-1}(\sfR^e_{n,k})^\ast(\tilde\gamma_n,\gnabla\iota_n^\ast\varphi) \dx t\\
        &\qquad\le \liminf_{n\to\infty} \int_0^T \sum_{k=1}^{n-1} \pra[\big]{\tilde f^e_{n,k}\overline\nabla\iota_n^\ast \varphi^e - (\sfR^e_{n,k})^\ast(\tilde\gamma_n,\gnabla\iota_n^\ast\varphi)} \dx t\\
        &\qquad\le \liminf_{n\to\infty} \int_0^T \sum_{k=1}^{n-1}\sfR^e_{n,k}(\tilde\gamma_n,\tilde f_n)\dx t.
    \end{align*} 
    With this, the lower semicontinuity for the internal edge contribution $\calR_\Medges$ follows from the Fenchel-Moreau duality theorem as in the remainder of the proof of \cite[Theorem 6.2 (i)]{HraivoronskaTse2023} with $\bbT = \operatorname{Id}$.

    It remains to treat the edge--vertex transition part $\calR_{\nodes,\edges}$, which in $\calR_n$ is given by the finitely many terms $  \sum_{e=\sfv\sfw\in\edges}\bigl[\sfR^e_\sfv(\tilde\gamma^e_{n,1},\bar\gamma_{n,\sfv},\bar f^e_{n,\sfv})+\sfR^e_\sfw(\bar\gamma_{n,\sfw},\tilde\gamma^e_{n,n},\bar f^e_{n,\sfw})\bigr]$. Since each $\sfR^e_\sfv$ is jointly convex and lower semicontinuous as the perspective function of $\sfC$, the convergence of the boundary traces from Proposition~\ref{prop:discrete_comp}~\ref{prop:discrete_comp:4new} together with the weak convergence $\iota_n\bar f_n\rightharpoonup\bar\jmath$ from Proposition~\ref{prop:discrete_comp}~\ref{prop:discrete_comp:2} directly yields $\liminf_{n\to\infty}\int_0^T\calR_{\nodes,\edges}(\upgamma_n,\bar f_n)\dx t\ge\int_0^T\calR_{\nodes,\edges}(\mu,\bar\jmath)\dx t$, e.g.\ by \cite[Theorem~2.34]{AmbrosioFuscoPallara00}. Adding the two contributions gives the claim.
\end{proof}

\section{Multiscale limits via EDP Convergence}\label{sec:multiscale}

Having established the gradient structure and well-posedness of solutions to~\eqref{eq:system_intro_linear}, we are now in the position to study different multiscale limits of the system.

\subsection{Kirchhoff limit}\label{ssec:Kirchhoff}

We start with the case in which the size of the reservoir vanishes and simultaneously the mass exchange between edges and vertices is accelerated. 
This scaling leads to the system studied in~\cite{ErbarForkertMaasMugnolo2022}. 
In particular, we use a similar continuity equation defined in duality with test functions being continuous across edges, which share a common vertex. The latter relation defines an equivalence relation $\sim$ and the according quotient space is denoted with $\sfG \coloneqq \sfL_{/\sim}$.
In this way, we arrive at a weak characterization of the Kirchhoff condition~\eqref{eq:Kirchhoff_intro_limit}.

\begin{definition}[Limit continuity equation and embedding]\label{def:tildeCE}
    The space of pairs $(\mu, j):[0,T]\to \calP(\Mgraph) \times\calM(\Medges)$ satisfying for all $\Phi\in C^1(\nodes\times\Medges)$
    with $\phi_\sfv = \varphi^e|_\sfv$ for all $\sfv\in\nodes$ and $e\in\edges(\sfv)$
    \begin{equation}\label{eq:def:tildeCE}%
		\pderiv{}{t} \pra[\bigg]{ \,\sum_{\sfv \in \nodes} \phi_\sfv \gamma_\sfv(t) + \sum_{e\in\edges} \int_0^{\ell^e} \mkern-8mu \varphi^e(x) \dx{\rho^e}(t)} = \sum_{e\in\edges} \int_0^{\ell^e}\mkern-8mu \partial_x \varphi^e \dx j^e(t)  \,
	\end{equation}
	for a.e.~$t\in [0,T]$
    is denoted by $\widetilde\CE$.

    Furthermore, the embedding $\Pi : \CE \to \widetilde\CE$ is defined by $\Pi(\mu,\rmj) = (\mu,j)$, i.e., the exchange fluxes with the reservoir $\bar\jmath$ are neglected.
\end{definition}
Indeed, choosing in \eqref{e:def:bCE:weak} test functions satisfying $\phi_\sfv=\varphi^e|_\sfv$ for all $\sfv \in \nodes$ and $e\in \edges(\sfv)$, we observe for all $(\mu,\rmj)\in\CE$ that $\Pi^\eps(\mu,\rmj) \in \widetilde\CE$.

The next ingredient is the rescaled dissipation functional and its tentative limit.
\begin{definition}[Rescaling and limit functionals]\label{def:D_epsD_0}
    Multiplying each of the constants $\scrk^e_\sfv$ by $\eps^{-1}$ and inserting the rescaled reference measures $\pi^\eps \coloneqq \frac{\pi}{\pi(\Medges)+\eps\omega(\nodes)}$ and $\omega^\eps \coloneqq \frac{\eps\omega}{\pi(\Medges)+\eps\omega(\nodes)}$ into \eqref{eq:def:Mgraph:EDfunctional}, the prelimit dissipation functionals read for $\mu=(\gamma,\rho)$ and $\rmj=(\bar\jmath,j)$ with $(\mu,\rmj)\in \CE$ as
\begin{align*}
    \calD_\eps(\mu,\rmj) = \int_0^T\Bigg(&\sum_{e\in\edges}\pra[\Bigg]{\frac{1}{2\diffedge^e}\int_0^{\ell^e} \frac{\abs{j^e}^2}{\rho^e}\dx x + 2\diffedge^e \int_0^{\ell^e}\abs[\Bigg]{\partial_x\sqrt{\frac{\rho^e}{\pi^{\eps,e}}}}^2\pi^{\eps,e}\dx x}\\
    &+\sum_{\sfv\in\nodes}\sum_{e\in\edges(\sfv)}\pra[\Bigg]{\sfC\bra[\big]{\bar\jmath^e_\sfv| \sigma^{\eps,e}_\sfv(\gamma_\sfv,\rho^e|_\sfv)}
    + 2 \sigma^{\eps,e}_\sfv\bra{\pi^{\eps,e}|_\sfv,\omega^\eps_\sfv} \abs[\Bigg]{ \sqrt{\frac{\rho^e|_\sfv}{\pi^{\eps,e}|_\sfv}} -  \sqrt{\frac{\gamma_\sfv}{\omega^\eps_\sfv}}}^2}\Bigg)\dx t,
\end{align*}
with $\sigma^{\eps,e}_\sfv(a,b) = \frac1\eps \scrk^e_\sfv\sqrt{ab}$.
The tentative limit dissipation functional is for $\mu=(\gamma,\rho)$ and $j$ with $(\mu,j)\in \widetilde\CE$ given by
\begin{align*}
    \calD_0(\mu,j) \coloneqq\int_0^T\Bigg(&\sum_{e\in\edges}\pra[\Bigg]{\frac{1}{2\diffedge^e}\int_0^{\ell^e} \frac{\abs{j^e}^2}{\rho^e}\dx x + 2\diffedge^e \int_0^{\ell^e}\abs[\Bigg]{\partial_x\sqrt{\frac{\rho^e}{\pi^e}}}^2\pi^e\dx x}\\
    &+\sum_{\sfv\in\nodes}\sum_{\;e,e'\in\edges(\sfv)}\chi_{\{0\}}\bra[\Bigg]{\frac{\rho^e|_\sfv}{\pi^e|_\sfv} - \frac{\rho^{e'}|_\sfv}{\pi^{e'}|_\sfv}}\Bigg)\dx t .
\end{align*}
Note that $\calD_0$ is independent of the component $\gamma$ in $\mu$.
\end{definition}

We are now in the position to state the main result of this section.
\begin{theorem}[EDP limit]\label{thm:EDP_Kirchhoff}
    Consider $(\mu^\eps,\rmj^\eps)_{\eps>0}\subset\CE$ satisfying $\sup_{\eps>0}\sup_{t\in[0,T]}\calE(\mu^\eps(t))<\infty$ and $\sup_{\eps>0}\calD_\eps(\mu^\eps,\rmj^\eps)<\infty$. Let $\Pi:\CE\to\widetilde\CE$ be defined by $\Pi(\mu,\rmj)\coloneqq (\mu,j)$.
    
    Then, there exists $(\mu,j)\in\widetilde\CE$ such that $\Pi(\mu^\eps,\rmj^\eps) \rightharpoonup (\mu,j)$ in the sense of Proposition~\ref{prop:compactness_Kirchhoff}, below, and it holds
    \begin{align*}
        \liminf_{\eps\to 0}\calE(\mu^\eps(t)) &\ge \calE(\mu(t))\quad\text{for all $t\in[0,T]$}\qquad\text{and}\qquad
        \liminf_{\eps\to 0} \calD_\eps(\mu^\eps,\rmj^\eps) \ge \calD_0(\mu,j).
    \end{align*}
    In particular, if in addition we have the well-preparedness $\calE(\mu^\eps(0)) \to \calE(\mu(0))$ as $\eps\to 0$, then 
    \begin{align*}
        \liminf_{\eps\to 0} \calL_\eps(\mu^\eps,\rmj^\eps) 
        \ge \calL_0(\mu,j).
    \end{align*}
\end{theorem}
\begin{proof}
    The compactness is proved in Proposition~\ref{prop:compactness_Kirchhoff}, after which the lower limit is established in Proposition~\ref{prop:lsc_Kirchhoff}, below. 
\end{proof}

\begin{remark}
     Since by Proposition~\ref{prop:compactness_Kirchhoff}, below, we have that $\gamma^\eps\to 0$ strongly in $L^1([0,T]\times\nodes)$, the remaining pair $(\rho,j)$ satisfies for all $\varphi\in  C(\sfG) \cap C^1(\Medges)$ the continuity equation 
    \begin{equation}\label{eq:def:tildeCE_no_gamma}
        \frac{\dx{}}{\dx t}\sum_{e\in\edges} \int_0^{\ell^e} \varphi^e \dx{\rho^e} 
        = \sum_{e\in\edges} \int_0^{\ell^e} \partial_x \varphi^e \dx j^e.
    \end{equation}
    Furthermore, it suffices to consider $\calD_0$ as a function of just $(\rho,j)$. 
    Hence, Theorem~\ref{thm:EDP_Kirchhoff} induces the gradient system in continuity equation format $(\Medges,\Medges,\nabla,\calE_0, \calD_0)$ in the sense of Definition~\ref{def:GradSystCE}, where $\nabla:C^1(\Medges)\to C^0(\Medges)$ is just the standard gradient on each metric edge (see~\eqref{eq:def:abstract-gradient:nodes}) and $\calE_0(\rho)\coloneqq \calH(\rho|\pi)$.
    The corresponding gradient flow equation is given in~\eqref{eq:system_intro_Kirchhoff}.
    In particular, we have recovered exactly the metric gradient structure of \cite{ErbarForkertMaasMugnolo2022}.

\end{remark}

We proceed by establishing the two steps of EDP convergence with embedding (cf.~Definition~\ref{def:EDP-convergence}), starting with the compactness after embedding.
\begin{proposition}[Compactness after embedding]\label{prop:compactness_Kirchhoff}
    Consider $(\mu^\eps,j^\eps) \in\CE$ such that 
    \[ 
    \sup_{\eps>0} \sup_{t\in[0,T]}\calE(\mu^\eps(t))<\infty
    \quad \text{and}\quad 
    \sup_{\eps>0}\calD_\eps(\mu^\eps,j^\eps) < \infty \,.
    \]
    Then, there exists a pair $(\mu,j)\in\widetilde\CE$ such that (along a subsequence) it holds
    \begin{enumerate}[label=(\roman*)]

    \item $\int_\cdot j^\eps(t)\dx t \rightharpoonup^\ast \int_\cdot j(t)\dx t$ in $\calM([0,T]\times\Medges)$;
    \item $\gamma^\eps \to 0$ strongly in $L^1([0,T]\times\nodes)$;
    \item $\mu^\eps(t)\rightharpoonup (0,\rho(t))$ narrowly in $\calP(\Mgraph)$ for a.e. $t\in[0,T]$;
    \item $\forall\;\sfv\in\nodes, e\in \edges(\sfv)$ there exists $u_\sfv \in L^1(0,T)$ such that 
    $\sqrt{\frac{\rho^{\eps,e}}{\pi^e}}\Big|_\sfv \rightharpoonup \sqrt{u_\sfv}$ weakly in $L^2(0,T)$;
    \item $\sqrt{\frac{{\rho^{\eps}}}{\pi}}\rightharpoonup\sqrt{\frac{\rho}{\pi}}$ weakly in $L^2(0,T;H^1(\pi))$.
    \end{enumerate}
\end{proposition}
\begin{proof}
    Applying from Remark~\ref{rem:PI} the Poincaré inequality, we obtain $\sqrt{\frac{\rho^\eps}{\pi}}\rightharpoonup\sqrt{\frac{\rho}{\pi}}$ weakly in $L^2(0,T;H^1(\pi))$. 
    Thus, due to continuity of the density of $\pi$, their traces converge weakly in $L^2(0,T;L^2(\nodes\times\edges))$ and it holds
\begin{align} \label{eq:Kichhoff_liminf_bound_1}  
    \int_0^T\sum_{\sfv\in\nodes}\sum_{e\in\edges(\sfv)}\abs[\bigg]{ \sqrt{\frac{\rho^{\eps,e}}{\pi^{\eps,e}}}\bigg|_\sfv}^2 \dx t
    \le C<\infty.
\end{align}
    Furthermore, the bound on the vertex part of the dissipation implies for all $\sfv\in\nodes$ and $e\in\edges(\sfv)$
\begin{align*} %
     \int_0^T\abs[\bigg]{ \sqrt{\frac{\rho^{\eps,e}}{\pi^{\eps,e}}}\bigg|_\sfv -  \sqrt{\frac{\gamma^\eps_\sfv}{\omega^\eps_\sfv}}}^2 \dx t
     \le C\sqrt{\eps}.
\end{align*}
Together with \eqref{eq:Kichhoff_liminf_bound_1}, this implies a uniform $L^2$ bound on $\sqrt{\frac{\gamma^\eps_\sfv}{\omega^\eps_\sfv}}$.
Hence, for all $\sfv\in\nodes$ there exists $u_\sfv$ such that $\sqrt{\frac{\gamma^\eps_\sfv}{\omega^\eps_\sfv}}\rightharpoonup \sqrt{u_\sfv}$ weakly in $L^2(0,T)$ and $\sqrt{\frac{\rho^{\eps,e}}{\pi^e}}\Big|_\sfv \rightharpoonup \sqrt{u_\sfv}$ weakly in $L^2(0,T)$ for all $e\in\edges(\sfv)$.

The uniform bounded of $\sqrt{\frac{\gamma^\eps}{\omega^\eps}}$ in $L^2([0,T]\times\nodes)$ and the vanishing $\omega^\eps \to 0$ by the scaling from Definition~\ref{def:D_epsD_0} also imply that $\gamma^\eps \to 0$ strongly in $L^1([0,T]\times\nodes)$ (and hence for a.e. time) as $\eps \to 0$ for all $\sfv\in\nodes$.

Finally, the uniform bound on $\int_0^T\int_0^{\ell^e} \frac{\abs{j^{\eps,e}}^2}{\rho^e}\dx x\dx t$ and the fact that $(\mu^\eps,j^\eps)\in\widetilde\CE$ imply by standard arguments (cf. e.g. \cite[Section 4]{DolbeaultNazaretSavare2009}) that there exists $\eta\in \AC([0,T];\calP(\nodes\times\Medges))$ such that $(\gamma^\eps\oplus\rho^\eps)(t) \rightharpoonup \eta(t)$ narrowly in $\calP(\nodes\times\Medges)$ and $j\in \calM([0,T]\times\Medges)$ such that $\int_\cdot j_t\dx t \rightharpoonup^\ast \int_\cdot j\dx t$ in $\calM([0,T]\times\Medges)$. Since we already know that $\gamma^\eps(t)$ vanishes as $\eps\to 0$ for a.e. $t\in[0,T]$, this implies the narrow convergence of $\mu^\eps(t)$  to $(0,\rho(t))$ for a.e. $t\in[0,T]$.
\end{proof}

The second step according to Principle~\ref{def:EDP-convergence}, the lower limit inequality, is shown next.
\begin{proposition}[Lower limit inequality]\label{prop:lsc_Kirchhoff}
Let $(\mu^\eps, \rmj^\eps)_\eps\subset\CE$ such that $\mu^\eps(t)\rightharpoonup\mu(t)$ narrowly in $\calP(\Mgraph)$ for all $t\in[0,T]$. 
Assume that for every $\sfv\in\nodes$ there exists $u_\sfv \in L^1(0,T)$ such that $\sqrt{\frac{\gamma^\eps_\sfv}{\omega^\eps_\sfv}} \rightharpoonup \sqrt{u_\sfv}$ weakly in $L^2(0,T)$ and $\sqrt{\frac{\rho^{\eps,e}}{\pi^{\eps,e}}}\Big|_\sfv \rightharpoonup \sqrt{u_\sfv}$ weakly in $L^2(0,T)$ for all $e\in\edges(\sfv)$. Furthermore, assume $\int_\cdot j_t^\eps\dx t \rightharpoonup^\ast \int_\cdot j_t\dx t$ weakly-$^\ast$ in $\calM([0,T]\times\Medges)$. Then, it holds
\begin{align*}
    \liminf_{\eps\to 0}\calE(\mu^\eps) &\ge \calE(\mu)\qquad\text{and}\qquad
    \liminf_{\eps\to 0} \calD_\eps(\mu^\eps,\rmj^\eps) \ge \calD_0(\mu,j).
\end{align*}
\end{proposition}
\begin{proof}
    The edge functionals of $\calD_\eps$ and the energy functional are independent of $\eps$ and weakly lower semicontinuous (see e.g. \cite[Theorem~2.34]{AmbrosioFuscoPallara00}). 
    Hence, it remains to consider the remaining edge-vertex terms. However, since $\frac{\rho^e}{\pi^e}\big|_\sfv = u_\sfv$, these terms can be dropped to obtain the result.
\end{proof}

\subsection{Contraction of metric edges by fast diffusion}\label{ssec:edgepoints}

By accelerating the diffusion  $d^e \to d^e/\eps$ on each metric edge $e\in \edges$, the discussion from Section~\ref{ssec:MS_intro} suggests that the densities $\rho^e$ on the metric edges become quasi-stationary in the sense that $\rho^e(t) = \zeta^e(t) \pi^e$ for some time-dependent edge function $\zeta: [0,T]\times \edges \to [0,\infty)$. Thus, we introduce the following tentative limit continuity equation.
\begin{definition}[Limit continuity equation and embedding]\label{def:CE:edgefast}
    Recall the notations $\hat\nodes = \nodes \cup \edges$ and $\hat\edges = \set*{e\sfv : e\in \edges , \sfv\in \nodes(e)}$ from~\eqref{eq:def:ExtendedGraph}.
    We denote by $\pCE$ the space of pairs $(\mu, \bar\jmath):[0,T]\to \calP(\Mgraph)\times\calM(\hat\edges)$ with $(\gamma,\rho(\Medges))\in\AC(0,T;\calP(\hat\nodes))$, which satisfy for all $\bar\Phi =(\phi,\bar\varphi) \in C(\hat\nodes)$ 
    \begin{equation}\label{eq:def:CE:edgefast}
        \pderiv{}{t} \pra*{ \sum_{\sfv \in \nodes} \phi_\sfv \gamma_\sfv(t) + \sum_{e\in\edges} \bar\varphi^e \rho(t,[0,\ell^e])}
        = \sum_{\sfv\in\nodes} \sum_{e\in\edges(\sfv)} 
            (\phi_\sfv-\bar \varphi^e) \bar\jmath^e_\sfv(t)
             \,.
    \end{equation}
    Furthermore, the embedding $\Pi:\CE\to \pCE$ is defined by $\Pi(\mu,\rmj)\coloneqq(\mu,\bar\jmath)$.
\end{definition}
Comparing~\eqref{eq:def:CE:edgefast} with the weak formulation of $\CE$ in~\eqref{e:def:bCE:weak}, we indeed observe that $\Pi$ embeds $\CE$ into $\pCE$. On the other hand, the weak formulation~\eqref{eq:def:CE:edgefast} is equivalent to the system of ODEs given by
\begin{equation}\label{eq:CE:edgefast:ODE}
    \left\{
    \begin{alignedat}{2}
        \dot \gamma_v(t) &= \sum_{e\in \edges(\sfv)} \bar \jmath_\sfv^e \quad &&\forall \sfv\in \nodes \,; \\
        \dot \rho(t,[0,\ell^e]) &= \sum_{\sfv\in\nodes(e)} \bar \jmath_\sfv^e =  \bar \jmath_\sfv^e + \bar \jmath_\sfw^e \quad &&\forall \sfv\sfw=e\in \edges \,.
    \end{alignedat}
    \right.
\end{equation}
We note that the system~\eqref{eq:CE:edgefast:ODE} is a standard (directed) graph continuity equation on the extended graph $(\hat\nodes,\hat\edges)$ as defined in~\eqref{eq:def:ExtendedGraph}.

\begin{definition}\label{def:CombGraph:D_epsD_0}
    Inserting the rescaled diffusivity constants into \eqref{eq:def:Mgraph:EDfunctional}, the prelimit dissipation functionals read 
\begin{align*}
    \calD_\eps(\mu,\rmj)
    \coloneqq \int_0^T\Bigg(&\sum_{e\in\edges}\pra[\Bigg]{\frac{\eps}{2\diffedge^e}\int_0^{\ell^e} \frac{\abs{j^e}^2}{\rho^e}\dx x + \frac{2\diffedge^e}{\eps} \int_0^{\ell^e}\abs[\Bigg]{\partial_x\sqrt{\frac{\rho^e}{\pi^e}}}^2\pi^e\dx x}\\
    &+\sum_{\sfv\in\nodes}\sum_{e\in\edges(\sfv)}\pra[\Bigg]{\sfC\bra[\big]{\bar\jmath^e_\sfv| \sigma^e_\sfv(\gamma_\sfv,\rho^e|_\sfv)}
    + 2 \sigma^e_\sfv\bra{\omega_\sfv,\pi^e|_\sfv} \abs[\Bigg]{ \sqrt{\frac{\rho^e}{\pi^e}}\bigg|_\sfv -  \sqrt{\frac{\gamma_\sfv}{\omega_\sfv}}}^2}\Bigg)\dx t,
\end{align*}
where we recall $\sigma^e_\sfv(a,b)= \scrk^e_\sfv\sqrt{ab}$.
The tentative limit dissipation functional is defined by
\begin{align*}
    \calD_0(\mu,\bar\jmath) \coloneqq\int_0^T\bra[\Bigg]{&\sum_{e\in\edges} \chi_{\{0\}}\bra[\Bigg]{\partial_x\sqrt{\frac{\rho^e}{\pi^e}}}\\
    &+ \sum_{\sfv\in\nodes}\sum_{e\in\edges(\sfv)}\pra[\Bigg]{\sfC\bra[\big]{\bar\jmath^e_\sfv| \sigma^e_\sfv(\gamma_\sfv,\rho^e|_\sfv)}
    + 2 \sigma^e_\sfv\bra{\pi^e|_\sfv, \omega_\sfv} \abs[\Bigg]{ \sqrt{\frac{\rho^e}{\pi^e}}\bigg|_\sfv -  \sqrt{\frac{\gamma_\sfv}{\omega_\sfv}}}^2}}\dx t.
\end{align*}
\end{definition}

\begin{theorem}[EDP limit]\label{thm:EDP_limit_edgefast}
    Consider a family of curves $(\mu^\eps,\rmj^\eps)_{\eps>0}\subset\CE$ satisfying the a priori bounds $\sup_{\eps>0}\sup_{t\in[0,T]}\calE(\mu^\eps(t))<\infty$ and $\sup_{\eps>0}\calD_\eps(\mu^\eps,\rmj^\eps)<\infty$. Let $\Pi:\CE\to\pCE$ be defined by $\Pi(\mu,\rmj)\coloneqq (\mu,\bar\jmath)$.
    
    Then, there exists $(\mu,\bar\jmath)\in\pCE$ such that $\Pi(\mu^\eps,\rmj^\eps) \rightharpoonup (\mu,\bar\jmath)$ in the sense of Proposition~\ref{prop:comp}, below, and we have 
    \begin{align*}
        \liminf_{\eps\to 0}\calE(\mu^\eps(t)) 
        &\ge \calE(\mu(t)) \quad\text{for all $t\in[0,T]$}
        \qquad\text{and}\qquad
        \liminf_{\eps\to 0} \calD_\eps(\mu^\eps,\rmj^\eps) 
        \ge \calD_0(\mu,\bar\jmath).
    \end{align*}
    In particular, if in addition we have the well-preparedness $\calE(\mu^\eps(0)) \to \calE(\mu(0))$ as $\eps\to 0$, then 
    \begin{align*}
        \liminf_{\eps\to 0} \calL_\eps(\mu^\eps,\rmj^\eps) 
        \ge \calL_0(\mu,\bar\jmath).
    \end{align*}
\end{theorem}
\begin{proof}
    The compactness is proved in Proposition~\ref{prop:comp}, after which the lower limit is established in Proposition~\ref{prop:lsc}, below.
\end{proof}
\begin{remark}
    Theorem~\ref{thm:EDP_limit_edgefast} characterizes the gradient system in continuity equation format $(\hat\nodes,\hat\edges,\overline\nabla,\calE,\calR_0^*)$ in the sense of Definition~\ref{def:GradSystCE}. 
    Here, we denoted $\overline\nabla:C(\hat\nodes) \to C(\hat\edges)$, $(\overline\nabla \Phi)^e_\sfv \coloneqq \phi_\sfv-\bar \varphi^e$ (see also~\eqref{eq:CE:edgefast:ODE}). Moreover, $\calE(\mu) = \calH(\mu|(\omega,\pi))$ and
    \begin{align*}
        \calR^\ast_0(\mu,\upxi) = \sum_{e\in\edges} \chi_{\{0\}}(\xi) 
        + \sum_{\sfv\in\nodes}\sum_{e\in\edges(\sfv)} \sigma^e_\sfv(\gamma_\sfv,\rho^e|_\sfv)\sfC^\ast\bra{\bar\xi}.
    \end{align*}
    This gradient system induces the gradient flow equation~\eqref{eq:system_intro_EdgePoints}.
\end{remark}

We will now prove Theorem~\ref{thm:EDP_limit_edgefast} in multiple steps, beginning with the compactness.

\begin{proposition}[Compactness]\label{prop:comp}
Let $(\mu^\eps, j^\eps)_\eps\subset\CE$ be such that $\sup_{\eps>0}\calD_\eps(\mu^\eps,j^\eps)<\infty$ and $\sup_{\eps>0}\sup_{t\in[0,T]}\calE_\eps(\mu^\eps(t))<\infty$. 
Then, there exist $\zeta\in L^\infty(0,T;\R^\edges)$, $\gamma\in L^1(0,T;\calM_{\ge 0}(\nodes))$, and $\bar\jmath\in L^1(0,T;\calM(\hat\edges))$ such that $((\gamma,\zeta\pi),\bar\jmath)\in\CE$ and such that (along a subsequence) we have the convergence
\begin{enumerate}[label=(\roman*)]
    \item $\int_\cdot \bar\jmath^\eps(t)\dx t \rightharpoonup^\ast \int_\cdot \bar\jmath(t)\dx t$ weakly-$^\ast$ in $\calM([0,T]\times \hat\edges)$;
    \item $\mu^\eps(t) \rightharpoonup (\gamma(t),\zeta(t)\pi)$ narrowly in $\calP(\Mgraph)$ for all $t\in[0,T]$;
    \item $\mu^\eps \to (\gamma,\zeta\pi)$ in strongly $L^1([0,T]\times\Mgraph)$.
\end{enumerate}
\end{proposition}
\begin{proof}

\medskip
\noindent
\emph{Step I: Spatial regularity.}\\
Regarding the convergence of $\mu^\eps$, we first note that for all $e\in\edges$ we have
\begin{align}\label{eq:ass:dx-bound}
        \norm[\bigg]{\partial_x\sqrt{\frac{\rho^{\eps,e}}{\pi^e}}}_{L^2(0,T;L^2({\pi^e}))} \le C\eps.
\end{align}
For each $e\in\edges$ and $t\in[0,T]$ we define 
\begin{align*}
    \zeta^{\eps,e}(t) \coloneqq \bra[\bigg]{\frac{1}{\pi^e([0,\ell^e])}\int_0^{\ell^e} \sqrt{\frac{\rho^{\eps,e}(t)}{\pi^e}}\dx \pi^e}^2,
\end{align*}
and (abusing notation) denote the constant-in-space function with value $\zeta^{\eps,e}$ also by $\zeta^{\eps,e}$. Then, \eqref{eq:ass:dx-bound} and the Poincaré inequality \eqref{eq:ass:PI} yield
\begin{align*}
    \norm[\bigg]{\sqrt{\frac{\rho^{\eps,e}}{\pi^e}}-\sqrt{ \zeta^{\eps,e}}}_{L^2(0,T; L^2(\pi^e))} 
    \le C\eps
\end{align*}
In particular, using Hölder's inequality, we obtain $\norm{\rho^\eps-\zeta^\eps\pi}_{L^1(0,T;W^{1,1}(\Medges))} \to 0$ as $\eps \to 0$.

On the other hand, Pinsker's inequality and the uniform energy bound imply the bound $\sup_{\eps>0} \norm{\zeta^\eps}_{L^\infty(0,T;\R^\edges)}<\infty$, so that there exists $\zeta\in L^\infty(0,T;\R^\edges)$ and a (not renamed) subsequence such that $\zeta^\eps\rightharpoonup \zeta$ weakly in $L^\infty(0,T;\R^\edges)$. 
Since by definition $\pi\in W^{1,1}(\Medges)$, this also implies $\zeta^\eps\pi\rightharpoonup \zeta\pi$ weakly in $L^\infty(0,T;W^{1,1}(\Medges))$. 
In total we have $\rho^\eps\rightharpoonup \zeta\pi$ weakly in $L^1(0,T;W^{1,1}(\Medges))$. For the vertex densities the energy bound directly yields for each $\sfv\in\nodes$ (along a subsequence) $\gamma_\sfv^\eps \rightharpoonup \gamma_\sfv$ weakly in $L^\infty(0,T)$ for some $\gamma_\sfv\in L^\infty(0,T)$. 

\medskip
\noindent
\emph{Step II: Convergence of fluxes.}\\ 
Regarding the fluxes, we note that the calculations from Step 1 also imply for all $e\in\edges$, $\sfv\in\nodes$ that
$\sup_{\eps>0}\norm[\Big]{\sqrt{\frac{\rho^{\eps,e}}{\pi^e}}\Big|_\sfv}_{L^2(0,T)} \le C$, and $\sup_{\eps>0}\norm[\Big]{\sqrt{\frac{\gamma^{\eps}_\sfv}{\omega_\sfv}}}_{L^\infty(0,T)} \le C$.
Consequently, we obtain for $A\in\calB([0,T])$, $e\in\edges$, and $\sfv\in\nodes$ that
\begin{equation}\label{eq:est_sigma}
    \int_A \!\! \sigma_\sfv^e(\gamma^\eps_\sfv,\rho^{\eps,e}|_\sfv)\dx t = \int_A\!\! \sigma_\sfv^e(\omega_\sfv,\pi^e|_\sfv)\sqrt{\frac{\rho^{\eps,e}}{\pi^e}\Big|_\sfv\frac{\gamma^\eps_\sfv}{\omega_\sfv}}\dx t
    \le C^2\sigma_\sfv^e(\omega_\sfv,\pi^e|_\sfv)\sqrt{\scrL^1(A)} = \tilde C \sqrt{\scrL^1(A)},
\end{equation}
for some constant $\tilde C$ depending on the $L^\infty$ norm of $(\pi,\omega)$.
Arguing as in the proof of Lemma~\ref{lem:finite_flux_comp}, this yields for every $s>0$
\begin{align*}
    \abs{\bar J^{\eps,e}_\sfv}(A) &\le \frac{\bar C\tilde C}{s} + \tilde C \sqrt{\scrL^1(A)}\frac{\sfC^\ast(s)}{s}.
\end{align*}
The right-hand side is independent of $\eps$ and, choosing $s$ large enough and $\scrL^1(A)$ small enough, can be made smaller than any $\delta>0$. Hence, it holds $\sup_{\eps>0}\abs{\bar J^{\eps,e}_\sfv} \ll \scrL^1$ and, choosing $A= [0,T]$, we find that $\sup_{\eps>0}\abs{\bar J^{\eps,e}_\sfv([0,T])} <\infty$. 
This implies that $(\bar J^\eps)_{\eps>0}$ has a subsequence converging weakly-$^\ast$in $\calM([0,T]\times\hat\edges)$ to some $\bar J\in\calM([0,T]\times\hat\edges)$ with $\bar J^e_\sfv\ll\scrL^1$ for all $e\in\edges$, $\sfv\in\nodes$.
By the disintegration theorem, there exists $\bar\jmath$ such that $\bar J^e_\sfv = \int_\cdot\bar\jmath^e_\sfv\dx t$. 

\medskip
\noindent
\emph{Step III: Strong convergence and narrow convergence for all times.}\\ 
Arguing as in Step III of the proof of Lemma~\ref{lem:strong_compactness}, we obtain that $\mu^\eps \to (\gamma,\zeta\pi)$ strongly in $L^1([0,T]\times\Mgraph)$.
This strong convergence, the continuity equation, and the uniform flux bound together also imply narrow convergence for all times, thereby ensuring that indeed $(\mu,\bar\jmath)\in\pCE$.
\end{proof}

\begin{proposition}[Lower limit inequality]\label{prop:lsc}
Let $(\mu^\eps, j^\eps)_\eps\subset\CE$ be such that $(\mu^\eps, \bar\jmath^\eps)^\eps\rightharpoonup ((\gamma,\zeta\pi),\bar\jmath)$ in the sense of Proposition~\ref{prop:comp}. 
Then, it holds
\begin{align*}
    \liminf_{\eps\to 0}\calE(\mu^\eps(t)) 
    &\ge \calE(\mu(t)) \quad\text{for all $t\in[0,T]$}
    \qquad\text{and}\qquad
    \liminf_{\eps\to 0} \calD_\eps(\mu^\eps,\rmj^\eps) 
    \ge \calD_0(\mu,\bar\jmath).
\end{align*}
\end{proposition}
\begin{proof}
We estimate
\begin{align*}
    \calD_\eps(\mu^\eps,j^\eps) &= \int_0^T\sum_{e\in\edges}\pra[\Bigg]{\frac{\eps}{2\diffedge^e}\int_0^{\ell^e} \frac{\abs{j^{\eps,e}}^2}{\rho^{\eps,e}}\dx x + \frac{2\diffedge^e}{\eps} \int_0^{\ell^e}\abs[\bigg]{\partial_x\sqrt{\frac{\rho^{\eps,e}}{\pi^e}}}^2\pi^e\dx x}\dx t\\
    &\quad+\int_0^T \sum_{\sfv\in\nodes}\sum_{e\in\edges(\sfv)}\pra[\Bigg]{\sfC\bra[\big]{\bar\jmath^{\eps,e}_\sfv|\sigma_\sfv^e(\gamma^\eps_\sfv,\rho^{\eps,e}|_\sfv)}
    + 2\sigma_\sfv^e\bra{\omega_\sfv,\pi^e|_\sfv} \abs[\bigg]{ \sqrt{\frac{\rho^{\eps,e}}{\pi^e}}\bigg|_\sfv -  \sqrt{\frac{\gamma^\eps_\sfv}{\omega_\sfv}}}^2}\dx t\\
    &\ge\int_0^T \sum_{\sfv\in\nodes}\sum_{e\in\edges(\sfv)}\pra[\Bigg]{\sfC\bra[\big]{\bar\jmath^{\eps,e}_\sfv|\sigma_\sfv^e(\gamma^\eps_\sfv,\rho^{\eps,e}|_\sfv)}
    + 2\sigma_\sfv^e\bra{\omega_\sfv,\pi^e|_\sfv} \abs[\bigg]{ \sqrt{\frac{\rho^{\eps,e}}{\pi^e}}\bigg|_\sfv -  \sqrt{\frac{\gamma^\eps_\sfv}{\omega_\sfv}}}^2}\dx t.
\end{align*}
Both the remaining integral functional and the energy are lower semicontinuous with respect to the convergences of $\mu^\eps$ and $ \bar\jmath^\eps$ (see e.g. \cite[Theorem~2.34]{AmbrosioFuscoPallara00}). Thus, since $\mu = ((\zeta^e \pi^e)_e,(\gamma_\sfv)_\sfv)$, we obtain both the claimed lower limits.
\end{proof}

\subsection{Combinatorial graph limit} \label{ssec:terminal}

We have shown in the previous section that the fast edge diffusion limit system is indeed equivalent to the ODE system~\eqref{eq:system_intro_EdgePoints} on the extended graph~\eqref{eq:def:ExtendedGraph} with extended vertex set $\hat\nodes = \nodes \cup \edges$ and extended edge set $\hat\edges = \set*{e\sfv : e\in \edges , \sfv\in \nodes(e)}$.
In this section, we aim to further reduce this system to obtain an ODE system on the original combinatorial graph $(\nodes,\edges)$ with vertex set $\nodes$ and edge set $\edges$.
To this end, we consider another rescaling, where the jump rates of the system~\eqref{eq:system_intro_EdgePoints} obtained in the Section~\ref{ssec:edgepoints} are accelerated to leave the edge nodes (see Figure~\ref{fig:MarkovTerminal}).
This is done by applying the results from~\cite[§7]{PeletierSchlichting2022}, where a similar problem is studied.
We first state the result in our notation, before establishing the link to~\cite[§7]{PeletierSchlichting2022}.
\begin{definition}[Limit continuity equation and embedding]\label{def:CE_terminaly}  
    The set $\widehat{\CE}$ is defined as the space of pairs $(\gamma, \bar\jmath):[0,T]\to \calP(\nodes)\times\calM(\nodes\times\edges)$ with $\gamma\in\AC(0,T;\calP(\nodes))$, which satisfy for all $\phi \in C(\nodes)$ 
    \begin{equation}\label{eq:def:CE:terminal}
        \pderiv{}{t} \sum_{\sfv \in \nodes} \phi_\sfv \gamma_\sfv(t)
        = \sum_{\sfv\in\nodes} \sum_{e\in\edges(\sfv)} \phi_\sfv \bar\jmath^e_\sfv(t)\,.
    \end{equation}
\end{definition}
\begin{remark}
    As before, this definition could be written without test functions as ODE: $\dot \gamma_\sfv(t)= \sum_{e\in\edges(\sfv)} \bar \jmath^e_{\sfv}$. However, the formulation~\eqref{eq:def:CE:terminal} makes it easy to see that \eqref{eq:def:CE:terminal} is a special case of \eqref{eq:def:CE:edgefast} with the choice $\varphi =0$. This ensures that the map $\Pi(\mu,\bar\jmath) \coloneqq (\gamma,\bar\jmath)$ is indeed an embedding $\Pi:\pCE\to\widehat{\CE}$.
\end{remark}
\begin{definition}[Reference measures, energies and dissipations]
    For $\eps \ge 0$ (including $\eps = 0)$ we introduce the rescaled reference measures
    \begin{align*}
        \pi^\eps &\coloneqq \frac{\eps}{Z^\eps}\pi
        \qquad\text{and}\qquad 
        \omega^\eps \coloneqq \frac{1}{Z^\eps}\omega, \qquad\text{where}\qquad Z^\eps \coloneqq \sum_{\sfv\in\nodes}\omega_\sfv + \eps\sum_{e\in\edges}\pi^e([0,\ell^e]).
    \end{align*}
    Moreover, the prelimit and limit energies are given by
    \begin{equation}\label{eq:def:terminal:energies}
    \calE_\eps(\mu)\coloneqq \calH(\mu|(\omega^\eps,\pi^\eps)), \quad\text{ for } \eps>0 
    \qquad\text{and}\qquad
    \calE_0(\gamma)\coloneqq \calH(\gamma|\omega^0) \,,\quad\text{respectively.}
    \end{equation}
    For the prelimit dissipations, the dynamics is sped up by replacing $\scrk^e_\sfv$ with $\scrk^e_\sfv/\sqrt{\eps}$ for all $e\in\edges$, $\sfv\in\nodes$ and define $\sigma^{\eps,e}_\sfv:\R_+\times\R_+\to \R_+$ by $\sigma^{\eps,e}_\sfv(a,b)=\scrk^e_\sfv\sqrt{ab}/\sqrt{\eps}$ and notice that
    \begin{equation}\label{eq:def:terminal:sigma0}
        \sigma^{\eps,e}_\sfv(\omega_\sfv^\eps,\pi^{\eps,e}|_\sfv) = \frac{\scrk^e_\sfv}{Z^\eps}\sqrt{\pi^e|_\sfv \omega_\sfv}\to \sigma^{0,e}_{\sfv} \coloneqq \frac{\scrk^e_\sfv}{Z^0}\sqrt{\pi^e|_\sfv \omega_\sfv}
        \qquad\text{ for every } e=\sfv\sfw \in \edges\,.
    \end{equation}
    For every $\eps> 0$ the rescaled prelimit dissipation functional is given by
    \begin{equation}\label{eq:terminal:Deps}
        \widehat\calD_\eps(\mu,\bar\jmath) \coloneqq \sum_{\sfv\in\nodes}\sum_{e\in\edges(\sfv)}
        \int_0^T
        \pra[\bigg]{
        \sfC\bra[\big]{\bar\jmath^e_\sfv | \sigma^{\eps}(\gamma_\sfv^\eps,\zeta^e\pi^{\eps,e}|_\sfv)}
        + 2 \sigma^{\eps}(\omega_\sfv^\eps,\pi^{\eps,e}|_\sfv) \abs[\bigg]{ \sqrt{\zeta^e} -  \sqrt{\frac{\gamma_\sfv}{\omega_\sfv^\eps}}}^2 
        } \dx{}t, %
    \end{equation}
    if $(\mu,\bar\jmath)\in\pCE$ and $\rho = \zeta\pi^\eps$ for some $\zeta:[0,T]\to\R^\edges$, setting $\widehat\calD_\eps(\mu,\bar\jmath) = \infty$, otherwise.
    The prelimit dissipation function is defined by $\widehat\calL_\eps(\mu,\bar\jmath) \coloneqq \calE_\eps(\mu(T)) - \calE_\eps(\mu(0)) + \widehat\calD_\eps(\mu,\bar\jmath)$.

    The tentative limit dissipation functional is defined for $(\gamma,\bar\jmath)\in\widehat{\CE}$ with $\bar\jmath^e_\sfv(t) = -\bar\jmath^e_\sfw(t)$ for all $e=\sfv\sfw\in \edges$ for a.e. $t\in [0,T]$ by
    \begin{align}\label{eq:def:terminal:limitD}
        \widehat{\calD}_0(\gamma,\bar\jmath) &\coloneqq 
        \sum_{e=\sfv\sfw\in\edges}
        \int_0^T
        \pra[\bigg]{
        \sfC\bra[\big]{\bar\jmath^e_\sfv | \overline{\scrk}_{\sfv\sfw} \sqrt{\omega^0_\sfv \omega^0_\sfw}}
        + 2  \overline{\scrk}_{\sfv\sfw} \sqrt{\omega^0_\sfv \omega^0_\sfw}\abs[\bigg]{\sqrt{\frac{\gamma_\sfv}{\omega_\sfv^0}} -  \sqrt{\frac{\gamma_\sfw}{\omega_\sfw^0}}}^2 
        } \dx{}t\,,  \\ %
        \text{with }\ \ \overline{\scrk}_{\sfv\sfw} &\coloneqq \frac{\mathsf{Harm}\bra*{\pi^e|_{\sfv} \,r(e,\sfv),\pi^e|_{\sfw} \, r(\sfw,e)}}{2\sqrt{\omega_\sfv \omega_\sfw}} \ \ \text{ and } \ \ \mathsf{Harm}(a,b)\coloneqq\frac{2}{\frac{1}{a}+\frac{1}{b}} \text{ for  } a,b>0 \,,
        \label{eq:def:terminal:k}
    \end{align}
    and we set $\widehat{\calD}(\gamma,\bar\jmath) =\infty$ if $(\gamma,\bar\jmath)\notin\widehat{\CE}$ or if $\bar\jmath^e_\sfv \ne -\bar\jmath^e_\sfw$ for any $e=\sfv\sfw\in \edges$ and on any subset of $[0,T]$ with positive Lebesgue measure. 

    The energy dissipation functional is then given as $\widehat{\calL}_0(\gamma,\bar\jmath) \coloneqq \calE_0(\gamma(T)) - \calE_0(\gamma(0)) + \widehat{\calD}_0(\gamma,\bar\jmath)$.
\end{definition}
Note that the above definition is chosen such that $\widehat\calD_1$ is equal to $\calD_0$ from Definition~\ref{def:CombGraph:D_epsD_0} and the same holds for the energy, i.e., the prelimit gradient system of this section is indeed the limit of the previous section. 

\begin{theorem}[EDP limit]\label{thm:EDP_terminal}
    Consider a family of curves $(\mu^\eps,\bar\jmath^\eps)_{\eps>0}\subset\pCE$ with $\rho^\eps = \zeta^\eps\pi^\eps$ satisfying the a priori bounds $\sup_{\eps>0}\sup_{t\in[0,T]}\calE_\eps(\mu^\eps(t))<\infty$ and $\sup_{\eps>0}\widehat\calD_\eps(\mu^\eps,\bar\jmath^\eps)<\infty$.
    
    Then, there exists $(\gamma,\bar\jmath)\in\overline\CE$ satisfying $\bar \jmath^e_{\sfv}(t)=-\bar\jmath^e_{\sfw}(t)$ for all $e=\sfv\sfw \in \edges$ and a.e. $t\in[0,T]$, such that
    \begin{enumerate}[label=(\roman*)]
		\item $\mu^\eps(t) \rightharpoonup (\gamma(t),0)$ narrowly in $\calP(\Mgraph)$ for all $t\in[0,T]$;
        \item $\int_\cdot\bar\jmath^\eps(t)\dx t \rightharpoonup^\ast \int_\cdot\bar\jmath(t)\dx t$ weakly-$^\ast$ in $\calM(\hat\edges)$;
    \end{enumerate}
    and it holds
    \begin{align*}
        \liminf_{\eps\to 0}\calE_\eps(\mu^\eps(t)) 
        &\ge \calE_0(\gamma(t)) \quad\text{for all $t\in[0,T]$}
        \qquad\text{and}\qquad
        \liminf_{\eps\to 0} \widehat\calD_\eps(\mu^\eps,\bar\jmath^\eps) 
        \ge \widehat{\calD}_0(\gamma,\bar\jmath).
    \end{align*}
    In particular, if in addition we have the well-preparedness $\calE_\eps(\mu^\eps(0)) \to \calE_0(\gamma(0))$ as $\eps\to 0$, then
    \begin{align*}
        \liminf_{\eps\to 0} \widehat\calL_\eps(\mu^\eps,\bar\jmath^\eps) 
        \ge \widehat{\calL}_0(\gamma,\bar\jmath).
    \end{align*}
\end{theorem}
\begin{proof}
    The compactness statements are a consequence of~\cite[Lemma 7.2]{PeletierSchlichting2022}, which also implies the antisymmetry property $\bar \jmath^e_{\sfv}=-\bar\jmath^e_{\sfw}$ (from the fact, that the graph divergence vanishes on $e\in\edges$) and that $\zeta^{\eps,e}$ converges narrowly to a finite measure $\zeta^e\in \calM([0,T]\times \edges)$ for all $e\in \edges$.
    To obtain the lower limit inequality, we change the order of summation in~\eqref{eq:terminal:Deps} and minimize with respect to $\zeta$, for which we do not have any control in the limit, by using the limit structure of the fluxes $\bar\jmath^e_\sfv=-\bar\jmath^e_\sfw$ for $e=\sfv\sfw\in \edges$. In this way, arguing as in the proof of~\cite[Lemma 7.4]{PeletierSchlichting2022}, we arrive at a lower limit, which still contains an inner minimization problem:
    \begin{align*}
        \liminf_{\eps\to0} \widehat\calD_\eps(\mu,\bar\jmath) \geq \sum_{e\in\edges} \int_0^T \inf_{\zeta^e\in \R_+} \set[\bigg]{\sum_{\sfv\in \nodes(e)} \bra[\bigg]{\sfC\bra[\big]{ \bar\jmath^e_\sfv | \sqrt{\zeta^e}\sigma^{0,e}_\sfv}
        + 2 \sigma^{0,e}_\sfv \abs[\bigg]{ \sqrt{\zeta^e} -  \sqrt{\frac{\gamma_\sfv}{\omega_\sfv^0}}}^2 } } \dx{t} \,.
    \end{align*}
    In our situation, the inner optimization problem is exactly the series law for the conductivities and instead of using~\cite[Lemma 7.3]{PeletierSchlichting2022}, we can directly apply~\cite[Corollary~3.4]{PeletierSchlichting2022} to conclude that $\liminf_{\eps\to0} \widehat\calD_\eps(\mu,\bar\jmath) \geq \widehat\calD_0(\gamma,\bar\jmath)$ as defined in~\eqref{eq:def:terminal:limitD}, where we note for $e=\sfv\sfw\in \edges$ that
    \[ 
        \overline\scrk_{\sfv\sfw} = \frac{\mathsf{Harm}(\sigma^{0,e}_{\sfv},\sigma^{0,e}_{\sfw})}{2\sqrt{\omega^0_\sfv \omega^0_\sfw}} = \frac{\mathsf{Harm}\bra*{ \scrk^{e}_\sfv \sqrt{\pi^e|_\sfv\,\omega_\sfv},\scrk^{e}_\sfw \sqrt{\pi^e|_\sfw\,\omega_\sfw}}}{2\sqrt{\omega_\sfv \omega_\sfw}}
        = \frac{\mathsf{Harm}\bra*{r(e,\sfv) \pi^e|_{\sfv}, r(\sfw,e) \pi^e|_{\sfw}}}{2\sqrt{\omega_\sfv \omega_\sfw}},
    \]
    where we used~\eqref{eq:def:terminal:sigma0} to derive the second equality and the definition of $\scrk^e_{\sfv\sfw}$ from~\eqref{eq:DBC:r} to derive the third equality.
\end{proof}

We can recover a limit gradient system in continuity equation format with respect to the standard graph gradient $\overline\nabla: C(\nodes)\to C(\edges)$ defined for $\phi:\nodes\to \R$ by $\overline\nabla\phi_{\sfv\sfw} \coloneqq \phi_\sfw-\phi_\sfv$. Indeed, for any $(\gamma,\bar\jmath)\in\widehat{\CE}$ satisfying $\bar \jmath^e_{\sfv}(t)=-\bar\jmath^e_{\sfw}(t)$, the continuity equation~\eqref{eq:def:CE:terminal} can be antisymmetrised as follows
\begin{equation}\label{eq:terminal:CE:antisym}
        \pderiv{}{t} \sum_{\sfv \in \nodes} \phi_\sfv \gamma_\sfv(t)
        = 
        \sum_{e\in\edges} 
        \sum_{\sfv\in\nodes(e)} 
        \phi_\sfv \bar\jmath^e_\sfv(t)
        =
        \sum_{e=\sfv\sfw\in\edges} \phi_\sfv \bar\jmath^e_\sfv(t)+ \sum_{e=\sfw\sfv\in\edges} \phi_\sfv \bar\jmath^e_\sfv(t) = \sum_{e=\sfv\sfw\in\edges}  \bra*{ \phi_\sfv - \phi_\sfw} \bar\jmath^{e}_\sfv(t) \,,
\end{equation}
where, for the second sum, we used the identity 
\[
\sum_{e=\sfw\sfv\in\edges} \phi_\sfv \bar\jmath^e_\sfv(t) 
= - \sum_{e=\sfw\sfv\in\edges} \phi_\sfv \bar\jmath^e_\sfw(t)
= - \sum_{e=\sfv\sfw\in\edges} \phi_\sfw \bar\jmath^e_\sfv(t) \,.
\]
Consequently, we can introduce the directed fluxes $f:[0,T]\to\calM(\edges)$ from $\sfv$ to $\sfw$ by 
\begin{equation}\label{eq:def:terminal:flux}
    f_{\sfv\sfw}(t)\coloneqq
    \begin{cases}
        \bar\jmath_{\sfw}^e(t)=-\bar\jmath^e_{\sfv}(t) , &\text{ if } \sfv\sfw=e \,;\\
        -\bar\jmath_{\sfv}^e(t)=\bar\jmath_{\sfw}^e(t) , & \text{ if }  \sfw\sfv=e \,.
    \end{cases}
\end{equation}
Hereby, we note that $\bar\jmath^e_\sfv$ corresponds in the prelimit model to the directed flux going from $e$ to $\sfv$.

The pair $(\gamma,f)$ solves the standard graph continuity equation on the directed graph $(\nodes,\edges)$
\begin{equation}\label{eq:terminal:GraphCE}
\pderiv{}{t} \sum_{\sfv \in \nodes} \phi_\sfv \gamma_\sfv(t) = \sum_{\sfv\sfw\in\edges} \overline\nabla \phi_{\sfv\sfw}(t) f_{\sfv\sfw}(t) 
\end{equation}
or equivalently in strong divergence form
\begin{equation}\label{eq:def:terminal:div}
     \dot \gamma_{\sfv}(t) + (\overline\div f(t))_{\sfv} = 0 \,,
     \quad\text{ where }\quad (\overline\div f)_{\sfv}\coloneqq \sum_{\sfw\in\nodes:\sfv\sfw\in \edges} f_{\sfv\sfw}-\sum_{\sfw\in\nodes:\sfw\sfv\in \edges} f_{\sfw\sfv} \,.
\end{equation}
In this case, we write $(\gamma,f)\in \overline\CE%
$ 
for such a curve. We use the same \emph{bar}-notation as in~\eqref{eq:CE_discrete}, which is the standard graph continuity equation on the extended graph $(\hat\nodes,\hat\edges)$ defined in~\eqref{eq:def:ExtendedGraph}.
Finally, we notice that the EDP limit $\widehat\calL_0$ has a gradient structure with respect to this continuity equation.
\begin{remark}\label{rem:terminal:GF}
    For any $(\gamma,\bar\jmath)\in \widehat{\CE}$ with $\widehat{\calL}_0(\gamma,\bar\jmath) <\infty$, it holds $(\gamma,f)\in\overline\CE$ with $f$ defined in~\eqref{eq:def:terminal:flux} and the identity
    \begin{equation}
        \widehat{\calL}_0(\gamma,\bar\jmath) = 
        \overline\calL%
        (\gamma,f) \coloneqq \calE_0(\gamma(T))-\calE_0(\gamma(0)) + \overline\calD%
        (\gamma,f) \,,
    \end{equation}
    with 
    \begin{equation}
        \overline\calD%
        (\gamma,f)\coloneqq\int_0^T
        \pra[\Bigg]{
        \sfC\bra[\big]{f_{\sfv\sfw}(t) | \overline{\scrk}_{\sfv\sfw} \sqrt{\omega^0_\sfv \omega^0_\sfw}}
        + 2  \overline{\scrk}_{\sfv\sfw} \sqrt{\omega^0_\sfv \omega^0_\sfw}\abs[\Bigg]{\sqrt{\frac{\gamma_\sfv(t)}{\omega_\sfv^0}} -  \sqrt{\frac{\gamma_\sfw(t)}{\omega_\sfw^0}}}^2 
        } \dx{}t\,.
    \end{equation}
    This defines the gradient system in continuity equation format $(\nodes,\edges,\overline\nabla,\calE_0,\overline\calR^*)$ in the sense of Definition~\ref{def:GradSystCE}, where $(\nodes,\edges,\overline\nabla)$ induces the continuity equation $\overline\CE$ as in~\eqref{eq:terminal:GraphCE}, $\calE_0$ is defined in~\eqref{eq:def:terminal:energies}, and the dual dissipation potential of $\cosh$-type is given for $\gamma\in\calP(\nodes)$ and $\Xi:\edges\to \R$ by
\begin{equation}\label{eq:def:terminal:limitR*}
    \overline{\calR}^*(\gamma,\Xi) \coloneqq \sum_{\sfv\sfw\in\edges} \overline{\scrk}_{\sfv\sfw} \sqrt{\omega^0_{\sfv}\,\omega^0_{\sfw}} \sfC^*(\Xi_{\sfv\sfw}) \,.
\end{equation}
In particular, we recover the gradient flow solution~\eqref{eq:GF:abstract} in strong form as
\begin{alignat*}{2}
    \forall \sfv\in\nodes:&\quad &&\dot \gamma_\sfv(t) +(\overline\div f(t))_\sfv = 0, \\ 
    \forall \sfv\sfw\in \edges:&\quad &&f_{\sfv\sfw}(t) = \rmD_2\overline{\calR}^*(\gamma,-\overline\nabla \calE_0'(\gamma))|_{\sfv\sfw}
    =\overline{\scrk}_{\sfv\sfw}\sqrt{\omega^0_{\sfv}\,\omega^0_{\sfw}} \bra*{ \frac{\gamma_\sfw(t)}{\omega_\sfw^0}- \frac{\gamma_\sfv(t)}{\omega_\sfv^0}} ,
\end{alignat*}
where we used the identity $(\sfC^*)'(\tfrac{1}{2}(\log b- \log a)) = b-a$, since $(\sfC^*)'(r)=2\sinh(r/2)$.
By recalling the definitions~\eqref{eq:def:terminal:k} and~\eqref{eq:def:terminal:div}, we arrive at the limit dynamic~\eqref{eq:intro:TerminalLimit} introduced in  Section~\ref{ssec:MS_intro}.
\end{remark}

\begin{remark}[Joint fast edge diffusion and combinatorial limit]\label{rem:joint_limit}
    We note that the limits considered in Section~\ref{ssec:edgepoints} and the present section are complementary to each other.
    Indeed, the fast diffusion limit is realized by accelerating the dynamics along the metric edges, but keeping the reference measures and the jump rates between metric edges and vertex reservoirs fixed. 
    In contrast to this, the combinatorial limit is obtained by rescaling the reference measures and jump coefficients.
    It is therefore plausible that these limits should commute with each other. 
    In light of the length of this manuscript, we refrain from adding yet another section to consider this joint limit analytically.
    Instead we present in Section~\ref{sec:Numerics} numerical results indicating that the two limits may indeed commute.
\end{remark}

\section{Numerical simulations}\label{sec:Numerics}

We present several numerical examples based on an implementation of the discrete scheme introduced in Section~\ref{sec:microscopic}.

This implementation is used to investigate the multiscale limits as considered in the analysis of Section~\ref{sec:multiscale}.   
Going beyond the scope of Section~\ref{sec:multiscale}, we study the rescaled systems for initial data that are not well-prepared, observing that the rescaled systems accommodate for this lack of well-preparedness after short times.
Furthermore, using relative entropies and Hellinger distances, rescaled prelimit systems and their respective limit systems are compared to each other for well-prepared initial data.
We conclude the section by numerically studying the simultaneous fast edge diffusion and combinatorial limit mentioned in Remark~\ref{rem:joint_limit}.

Due to the scope of the present work, we focus on the case of a single graph (triangular) and fixed reference measures. The interesting study of different measures, graph topologies or other driving functionals, such as for example nonlocal interactions, is postponed to future works.

\paragraph*{Prelimit system }
With the notation of Section~\ref{sec:microscopic}, the (space-)discrete version of \eqref{eq:system_intro_linear} reads as 

\begin{align*}
    \frac{\dx}{\dx t}\tilde\gamma^e_k &= \frac{\tilde d^e_k}{(h^e_n)^2}\bra[\bigg]{\frac{\tilde\gamma^e_{k+1}}{\tilde\omega^e_{n,k+1}}-\frac{\tilde\gamma^e_k}{\tilde\omega^e_{n,k}}} + \frac{\tilde d^e_{k-1}}{(h^e_n)^2}\bra[\bigg]{\frac{\tilde\gamma^e_{k-1}}{\tilde\omega^e_{n,k-1}}-\frac{\tilde\gamma^e_k}{\tilde\omega^e_{n,k}}},\quad k=2,\ldots, n-1,\, e\in \edges,
\end{align*}
as well as
\begin{align*}
    \frac{\dx}{\dx t}\tilde\gamma^e_1 
    &= \frac{\tilde d^e_1}{(h^e_n)^2}\bra[\bigg]{\frac{\tilde\gamma^e_2}{\tilde\omega^e_{n,2}}-\frac{\tilde\gamma^e_1}{\tilde\omega^e_{n,1}}} 
    +  d^{e,n}_{\sfv} \bra[\bigg]{\frac{\bar\gamma_\sfv}{\omega_\sfv}-\frac{\tilde\gamma^e_1}{\tilde\omega^e_{n,1}}},\quad e=\sfv\sfw\in \edges,\\
    \frac{\dx}{\dx t}\tilde\gamma^e_n 
    &= d^{e,n}_\sfw\bra[\bigg]{\frac{\bar\gamma_\sfw}{\omega_\sfw}-\frac{\tilde\gamma^e_n}{\tilde\omega^e_{n,n}}} 
    + \frac{\tilde d^e_{n-1}}{(h^e_n)^2}\bra[\bigg]{\frac{\tilde\gamma^e_{n-1}}{\tilde\omega^e_{n,n-1}}-\frac{\tilde\gamma^e_n}{\tilde\omega^e_{n,n}}},\quad e=\sfv\sfw\in \edges,    
\end{align*}
with $\tilde d^e_k \coloneqq \diffedge^e\sqrt{\tilde\omega^e_{n,k}\tilde\omega^e_{n,k+1}}$, $d^{e,n}_{\sfv} \coloneqq \scrk^e_\sfv\sqrt{\omega_\sfv\tilde\omega^e_{n,0}}$, and $d^{e,n}_\sfv \coloneqq \scrk^e_\sfv\sqrt{\tilde\omega^e_{n,n}\omega_\sfv}$, where the internal vertex measures $\tilde\omega^e_{n,k}$ are as in \eqref{eq:def_internal_vertex_measures}.
The vertex evolution is characterized by
\begin{align*}
    \partial_t\bar \gamma_\sfv &= \sum_{e=\sfv\sfw\in \edges(\sfv)}d^{e,n}_{\sfv}\bra[\bigg]{ \frac{\tilde\gamma^e_1}{\tilde\omega^e_{n,1}}-  \frac{\bar \gamma_\sfv}{\omega_\sfv} } + \sum_{e=\sfw\sfv\in \edges(\sfv)}d^{e,n}_\sfv\bra[\bigg]{ \frac{\tilde\gamma^e_n}{\tilde\omega^e_{n,n}}-  \frac{\bar \gamma_\sfv}{\omega_\sfv} },  &&\forall \sfv\in \nodes \,.
\end{align*}
We will perform all simulations with densities as unknowns, which we now introduce. For all $n\in\bbN$, $\sfv\in\nodes$, $e\in\edges$ and $k=1,\ldots,n$ we define $\tilde u_k^e = \frac{\tilde\gamma_k^e}{\tilde\omega_{n,k}^e}$ and $\bar u_\sfv = \frac{\bar\gamma_\sfv}{\omega_\sfv}$. Substituting these into the above equations, we obtain
\begin{align}\label{eq:discrete_u_interior}
    \tilde \omega_{n,k}^e\frac{\dx}{\dx t}\tilde u^e_k &= \frac{\tilde d^e_k}{(h^e_n)^2}\bra{\tilde u^e_{k+1}-\tilde u^e_k} + \frac{\tilde d^e_{k-1}}{(h^e_n)^2}\bra{\tilde u^e_{k-1}-\tilde u^e_k},\quad k=2,\ldots, n-1,\, e\in \edges,
\end{align}
as well as
\begin{subequations}\label{eq:discrete_u_bndries}
\begin{align}
    \tilde\omega^e_{n,1}\frac{\dx}{\dx t}\tilde u^e_1 
    &= \frac{\tilde d^e_1}{(h^e_n)^2}\bra{\tilde u^e_2-\tilde u^e_1}
    +  d^{e,n}_{\sfv} \bra{\bar u_\sfv-\tilde u^e_1},\quad e=\sfv\sfw\in \edges, \label{eq:discrete_u_bndries:v}\\
    \tilde\omega^e_{n,n}\frac{\dx}{\dx t} u^e_n 
    &= d^{e,n}_\sfw\bra{\bar u_\sfw-\tilde u^e_n} 
    + \frac{\tilde d^e_{n-1}}{(h^e_n)^2}\bra{\tilde u^e_{n-1}-\tilde u^e_n},\quad e=\sfv\sfw\in \edges, \label{eq:discrete_u_bndries:w}
\end{align}
\end{subequations}
and
\begin{align}\label{eq:discrete_vertex}
    \omega_\sfv\partial_t \bar u_\sfv &= \sum_{e=\sfv\sfw\in \edges(\sfv)}d^{e,n}_{\sfv}\bra*{ \tilde u^e_1-  \bar u_\sfv} + \sum_{e=\sfw\sfv\in \edges(\sfv)}d^{e,n}_{\sfv}\bra*{\tilde u^e_n-  \bar u_\sfv},  &&\forall \sfv\in \nodes \,.
\end{align}

The system of ODEs \eqref{eq:discrete_u_interior}--\eqref{eq:discrete_vertex} will be used for all numerical simulations below and we fix $n=100$ in all experiments except the ones displayed in Figure~\ref{fig:hellinger}.

All experiments are carried out on a triangle with three vertices and three edges as sketched in Figure~\ref{fig:graphs}. We also fix the length of all edges to one, i.e., $\ell^{e_1}=\ell^{e_2}=\ell^{e_3}=1$. In addition, for the diffusion coefficients as well as the symmetrised jump rates, we choose $\diffedge^{e_1}=\diffedge^{e_2}=\diffedge^{e_3}=1$ and $\scrk^e_\sfv=1$ for $e=e_1,\,e_2,\,e_3$ and $v=v_1,\,v_2,\,v_3$, respectively.

\begin{figure}[!ht]
    \centering
    \begin{minipage}{.4\textwidth}
    \centering
    \begin{tikzpicture}[scale=1, transform shape]
    \pgfdeclarelayer{background}
    \pgfsetlayers{background,main}
    \begin{scope}[every node/.style={draw,circle}]
            \node (v1) at (0,0) {$\sfv_1$};
            \node (v2) at (2,-1) {$\sfv_2$};
            \node (v3) at (2,1) {$\sfv_3$};
    \end{scope}
        \begin{scope}[>={Stealth[bluegray]},
            every edge/.style={draw=bluegray,line width=1pt}]
            \path [-] (v1) edge node[bluegray,below] {$e_1$} (v2);
            \path [-] (v2) edge node[bluegray,below right] {$e_2$} (v3);
            \path [-] (v3) edge node[bluegray,above left] {$e_3$} (v1);
    \end{scope}
    \end{tikzpicture}
    \end{minipage}
    \caption{Sketch of the triangular graph used in all numerical experiments.} %
    \label{fig:graphs}
\end{figure}
For the invariant measures, we choose 
\begin{align*}
\pi^{e_1} = \frac{x + 0.1}{Z}, \quad \pi^{e_2} = \frac{\sin(4\pi x) + 1.1}{Z}, \quad \pi^{e_3} = \frac{0.8x^2 - 1.8x + 1.1}{Z}
\end{align*}
and 
\begin{align*}
\omega_{\sfv_1} = \frac{0.1}{Z},\quad \omega_{\sfv_2} = \frac{1.1}{Z},\quad \omega_{\sfv_3} = \frac{1.1}{Z},
\end{align*}
where the normalization factor
$$
Z = \sum_{i=1}^3 \bra*{\int_0^1 \pi^{e_i}(x)\dx x + \omega_{\sfv_i}} = 4.47,
$$
ensures that the invariant measure on the whole metric graph remains a probability measure. 
Note that this choice is continuous not only along the edges, but also at the vertices. 

In accordance with Section~\ref{sec:microscopic}, we obtain discretized quantities on the edges by integrating over patches, see \eqref{eq:def_internal_vertex_measures}, i.e., $\tilde\omega^e_{n,k}\coloneqq \pi^e(I^e_{n,k})$ with $I^e_{n,\alpha} =[(\alpha-1) h^e_n,\alpha h^e_n)$ and $h^e_n = \ell^e/n$. 
Here, we stress that both $\tilde\gamma^e_k$ and $\gamma_v$ are measures (not densities) and to ensure that they form a probability vector on the graph, the sum over all of them needs to be one (without any additional scaling factors).

On the other hand, from the discrete quantities $\tilde \gamma_k^e$, we can recover the continuous objects $\rho^n$ by using the embedding defined in \eqref{eq:embedding_discr_to_cont}, i.e., by applying a piecewise constant interpolation and scaling with $n$ to preserve the total mass. 

Finally, unless stated otherwise, we prescribe the initial data
\begin{align*}
\tilde \gamma^{e_1}_k &= \tilde \gamma^{e_2}_k = \tilde \gamma^{e_3}_k = \frac{1}{6n}\qquad\text{and}\qquad
\bar \gamma_{\sfv_1} = \bar \gamma_{\sfv_2} = \bar \gamma_{\sfv_3} = \frac{1}{6},
\end{align*}
which correspond to uniform initial measures in the continuum.

Regarding the numerical solver, we note that upon rescaling the parameters to study the different limits, the system becomes stiff, meaning that some of the components evolve on a very different time scale than other. 
Therefore, a simple explicit time stepping scheme is no longer appropriate, as for a stable solution very small time steps would be necessary. 
Instead, we use an implicit multi-step variable-order (order from $1$ to $5$) method that is based on a backward differentiation formula for the derivative, \cite{Shampine1997}, and that is implemented in $\textrm{scipy}$.

\paragraph*{Kirchhoff-limit }
To obtain discrete limit dynamics for this case, we need to formulate the discrete version of the graph evolution without the vertex dynamics, ensuring continuity of the measures over the vertices as well as the Kirchhoff condition to preserve the total mass.
To this end, we introduce a discrete patch around each vertex, which can be thought of as enforcing that all edge degrees of freedom adjacent to this vertex are equal. 
We denote this patch vertex degree of freedom by $\tilde u_v$. %

The dynamics in the interior are then modified as follows, taking into account that the  first and last degree of freedom on each edge is removed:
\begin{align}\label{eq:discrete_u_interior_kirchhoff}
    \tilde \omega_{n,k}^e\frac{\dx}{\dx t}\tilde u^e_k &= n^2\tilde d^e_k\bra{\tilde u^e_{k+1}-\tilde u^e_k} + n^2\tilde d^e_{k-1}\bra{\tilde u^e_{k-1}-\tilde u^e_k},\quad k=3,\ldots, n-2,\, e\in \edges.
\end{align}
On the internal vertices adjacent to the vertex degree of freedom we have
\begin{align}\label{eq:discrete_u_bndries_kirchhoff}
    \tilde\omega^e_{n,2}\frac{\dx}{\dx t}\tilde u^e_2 
    &= n^2\tilde d^e_2\bra{\tilde u^e_3-\tilde u^e_2}
    +  n^2\tilde d^e_1\bra{\bar u_\sfv-\tilde u^e_2},\\
    \tilde\omega^e_{n,n-1}\frac{\dx}{\dx t} \tilde u^e_{n-1} 
    &= n^2d^e_{n-1}\bra{\bar u_\sfv-\tilde u^e_n} 
    + n^2\tilde d^e_{n-2}\bra{\tilde u^e_{n-2}-\tilde u^e_{n-1}},\quad e=\sfv\sfw\in \edges.
\end{align}
Finally, for the vertex patch we use the notation
\begin{align*}
\tilde\omega_\sfv = \sum_{e=\sfv\sfw\in \edges(\sfv)} \tilde\omega_{n,1}^e + \sum_{e=\sfw\sfv\in \edges(\sfv)}\tilde \omega_{n,n}^e,
\end{align*}
and obtain 
\begin{align}\label{eq:discrete_vertex_kirchhoff}
    \tilde \omega_\sfv\partial_t \bar u_\sfv &= \sum_{e=\sfv\sfw\in \edges(\sfv)}n^2\tilde d^e_{1}\bra*{ \tilde u^e_2-  \bar u_\sfv} + \sum_{e=\sfw\sfv\in \edges(\sfv)}n^2\tilde d^{e}_{n-1}\bra*{\tilde u^e_{n-1}-  \bar u_\sfv},  &&\forall \sfv\in \nodes \,.
\end{align}
We remark that we expect the system above to converge, as $n\to\infty$, to the continuous Kirchhoff system. Yet, the rigorous proof of this statement would need similar arguments as in Section~\ref{sec:microscopic} and is postponed to future research. 

In Figure \ref{fig:Kirchhoff_limit}, we present a comparison of the measures $\upgamma^\eps = u^\eps\upomega^\eps$ for the (rescaled) system \eqref{eq:discrete_u_interior}--\eqref{eq:discrete_vertex}, i.e., with $\tilde \omega^\eps_\sfv \coloneqq \eps  \tilde \omega_\sfv/Z^\eps$ (with $Z^\eps$ a renormalization factor of order 1), $\scrk^{\eps,e}_v = \scrk^{e}_v / \eps$, and the Kirchhoff system \eqref{eq:discrete_u_interior_kirchhoff}--\eqref{eq:discrete_vertex_kirchhoff}. 

As expected, we observe that continuity at the vertices is enforced more and more as $\eps$ approaches zero. 
Additionally, the fact that our initial data are not well-prepared leads to the formation of boundary layers at short times, see inlets (b) and (c). They occur as the rescaled invariant vertex measures together with the rescaled jump rates allow for a very fast transport of mass off the vertices. 
For later times, the dynamics on the edges become dominant, driving the shapes of the profiles towards the invariant (edge) measures. 
At large times, we see that the rescaled systems almost coincide with the limit system, while the unscaled systems remain distinct.

\begin{figure}[!ht]
    \centering
    \subfloat[$t=0$]{\includegraphics[width=.47\textwidth]{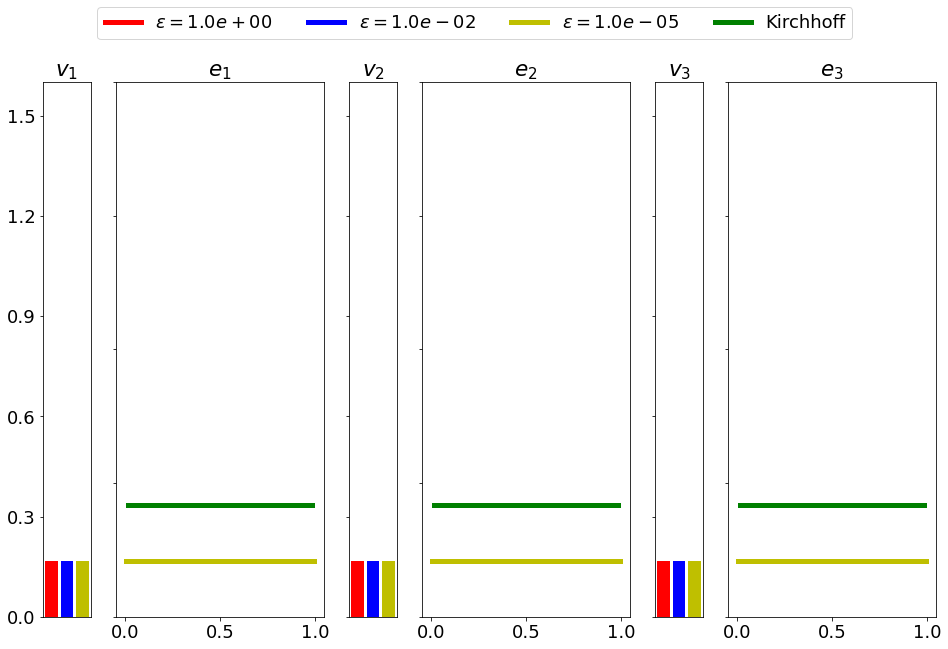}}\hspace{5mm}
    \subfloat[$t=0.005$]{\includegraphics[width=.47\textwidth]{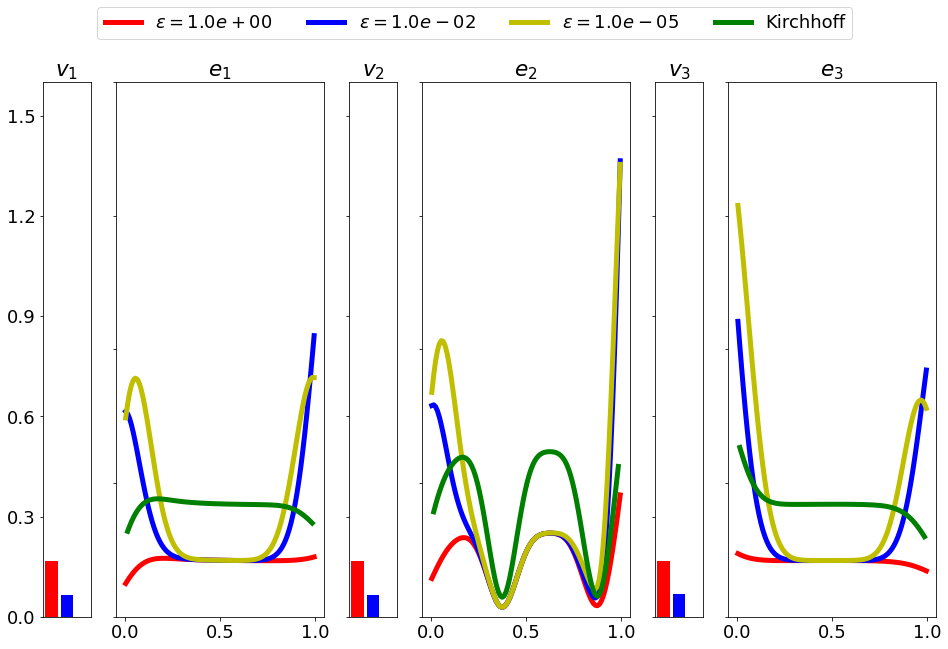}}
    \\
    \subfloat[$t=0.01$]{\includegraphics[width=.47\textwidth]{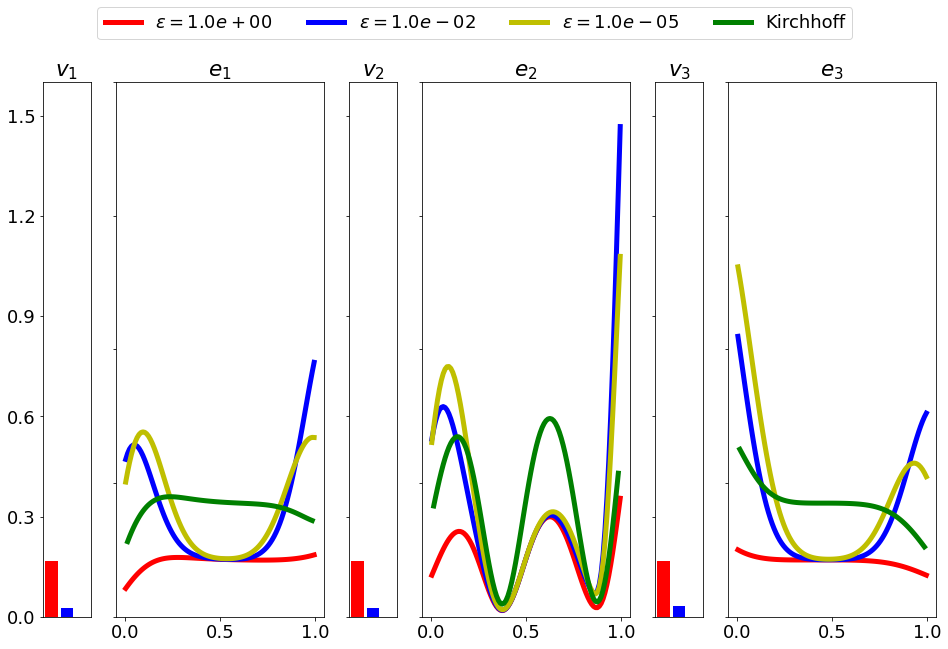}}\hspace{5mm}
    \subfloat[$t=0.1$]{\includegraphics[width=.47\textwidth]{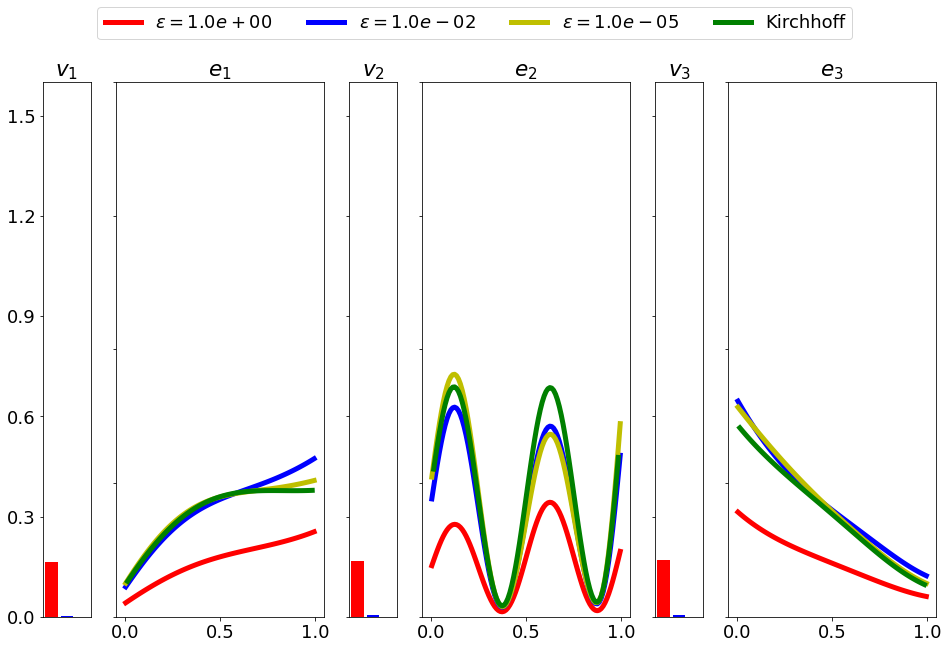}}
    \\
    \subfloat[$t=1.0$]{\includegraphics[width=.47\textwidth]{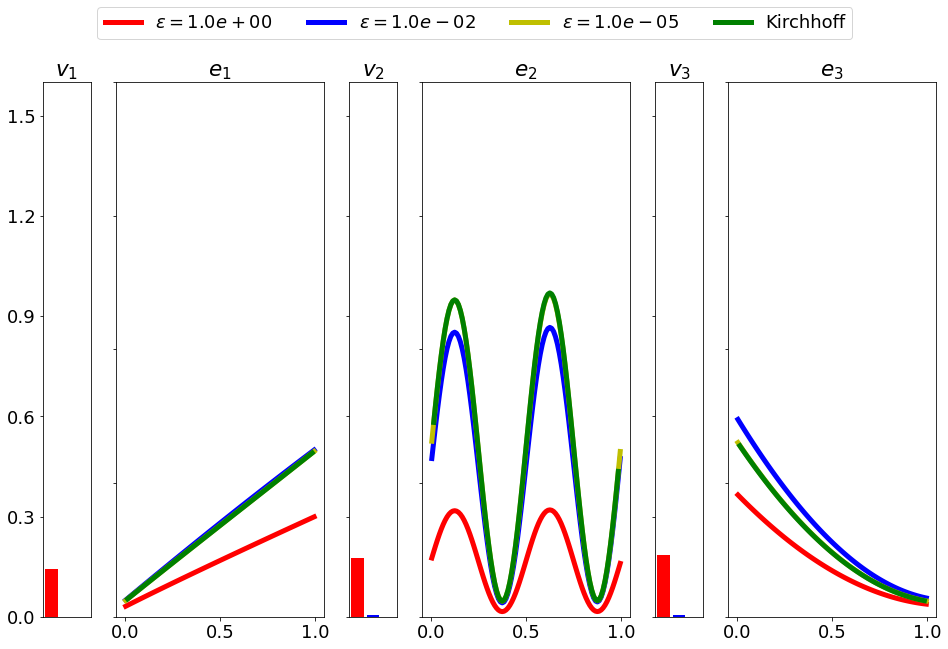}}\hspace{5mm}
    \subfloat[$t=40.0$]{\includegraphics[width=.47\textwidth]{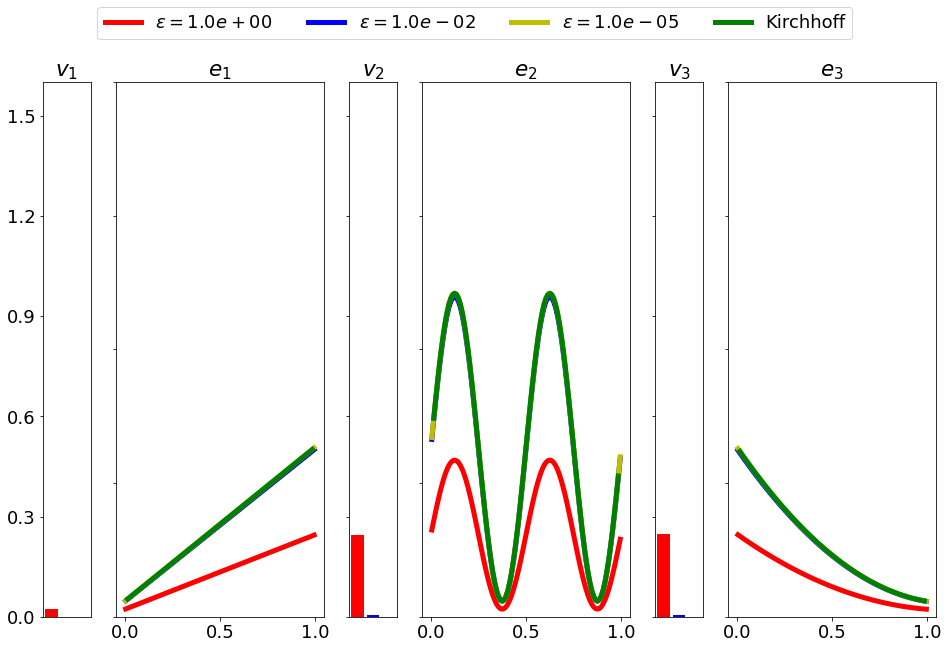}}
    \caption{Evolution of the measures $\upgamma^\eps = u^\eps\upomega^\eps$ for the prelimit system with rescaling $\omega^\eps_\sfv \coloneqq \eps \omega_\sfv$ and $\scrk^{\eps,e}_v = \scrk^{e}_v / \eps$ with different values of $\eps$ as well as for the discrete Kirchhoff system \eqref{eq:discrete_u_interior_kirchhoff}--\eqref{eq:discrete_vertex_kirchhoff}.
    }
    \label{fig:Kirchhoff_limit}
\end{figure}

\paragraph*{Fast edge diffusion limit }
Next, we consider the fast diffusion limit $\diffedge^{\eps,e} = \eps^{-1} \diffedge^e = \eps^{-1}$ for all $e\in \edges$. The resulting limit system on the graph $(\hat\nodes,\hat\edges)$ introduced in \eqref{eq:def:ExtendedGraph} reads as
\begin{equation}\label{eq:discrete_fast}
    \begin{aligned}
    		\pi^e[0,1]\partial_t\zeta^e(t) &= \sum_{\sfv\in \nodes(e)}\scrk^e_\sfv\sqrt{\pi^e|_{\sfv}\omega_\sfv}\bra*{\bar u_\sfv(t) - \zeta^e(t)},  &&\forall e\in \edges,\\
		\omega_v\partial_t \bar u_\sfv(t) &= \sum_{e\in \edges(\sfv)}\scrk^e_\sfv\sqrt{\pi^e|_{\sfv}\omega_\sfv}\bra*{ \zeta^e(t)  - \bar u_\sfv},  &&\forall \sfv\in \nodes \,,
    \end{aligned}
\end{equation}
where $u_v$ again denotes the quotient $\gamma_v / \omega_v$ and where we kept the notation $\zeta^e(t)$ from Proposition~\ref{prop:comp} as it already denotes a density. 
Figure \ref{fig:Fast_limit} displays the dynamics of the measures $\upgamma^\eps = u^\eps\upomega$ obtained from the (rescaled) system \eqref{eq:discrete_u_interior}--\eqref{eq:discrete_vertex} and the fast edge system \eqref{eq:discrete_fast}. 
Note that we visualize the limit system by presenting the product $\zeta^e\pi^e$ on each edge.

In this case, no boundary layers occur and the stationary profiles of prelimit and limit coincide, since no rescaling of the reference measures takes place. 
We do, however, see that with decreasing $\eps$ the profile of the respective invariant measures on the edges is assumed more quickly.

\begin{figure}[!ht]
    \centering
    \subfloat[$t=0$]{\includegraphics[width=.47\textwidth]{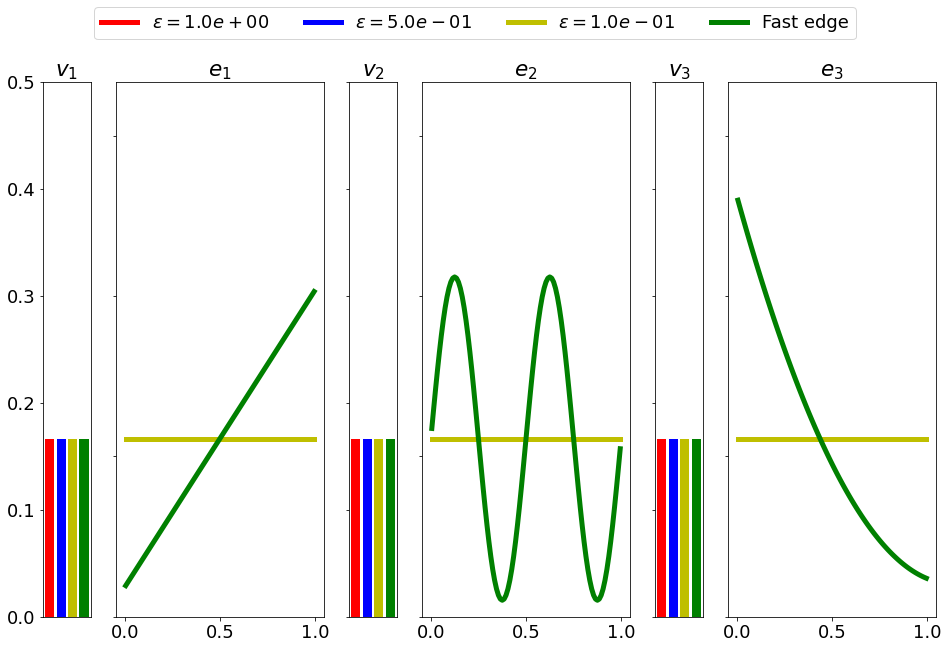}}\hspace{5mm}
    \subfloat[$t=0.005$]{\includegraphics[width=.47\textwidth]{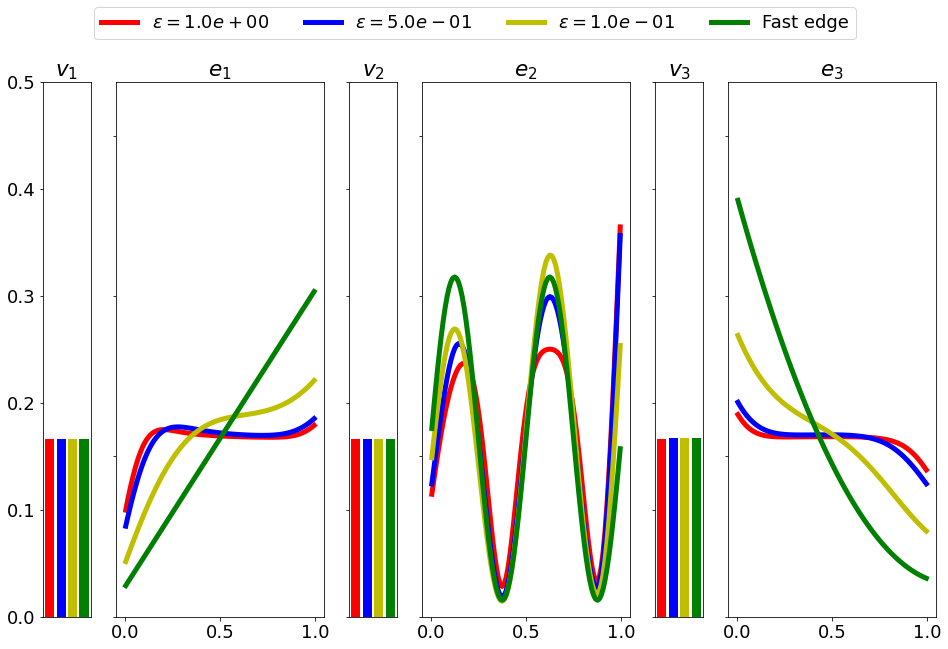}}
    \\
    \subfloat[$t=0.01$]{\includegraphics[width=.47\textwidth]{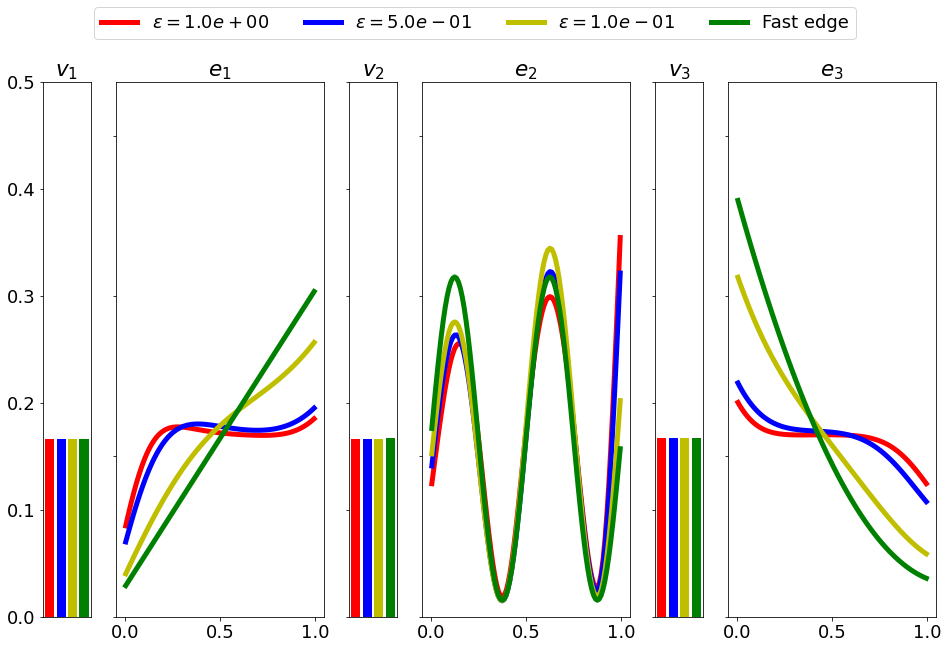}}\hspace{5mm}
    \subfloat[$t=0.1$]{\includegraphics[width=.47\textwidth]{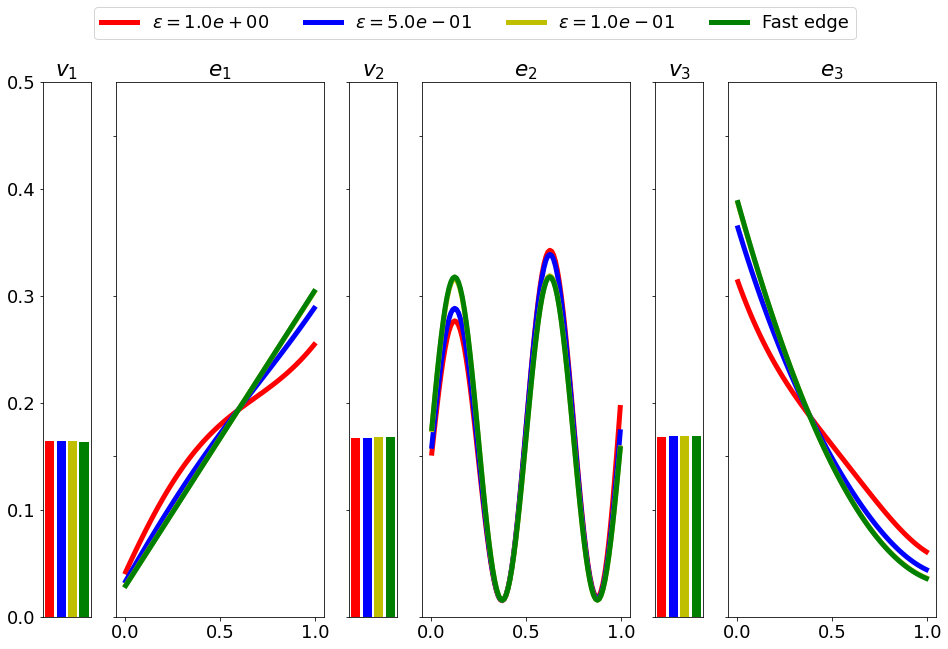}}
    \\
    \subfloat[$t=1.0$]{\includegraphics[width=.47\textwidth]{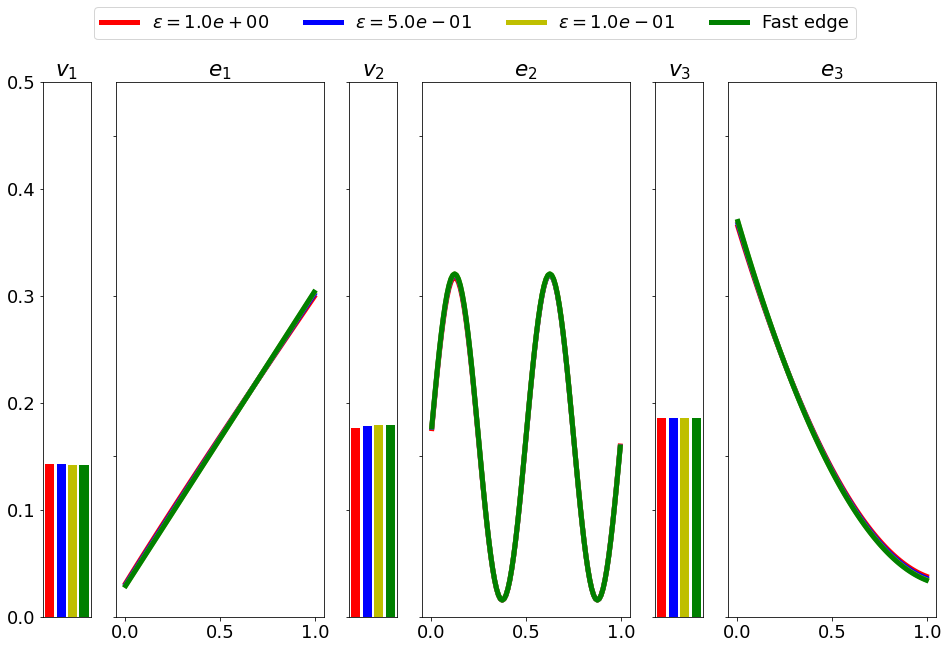}}\hspace{5mm}
    \subfloat[$t=40.0$]{\includegraphics[width=.47\textwidth]{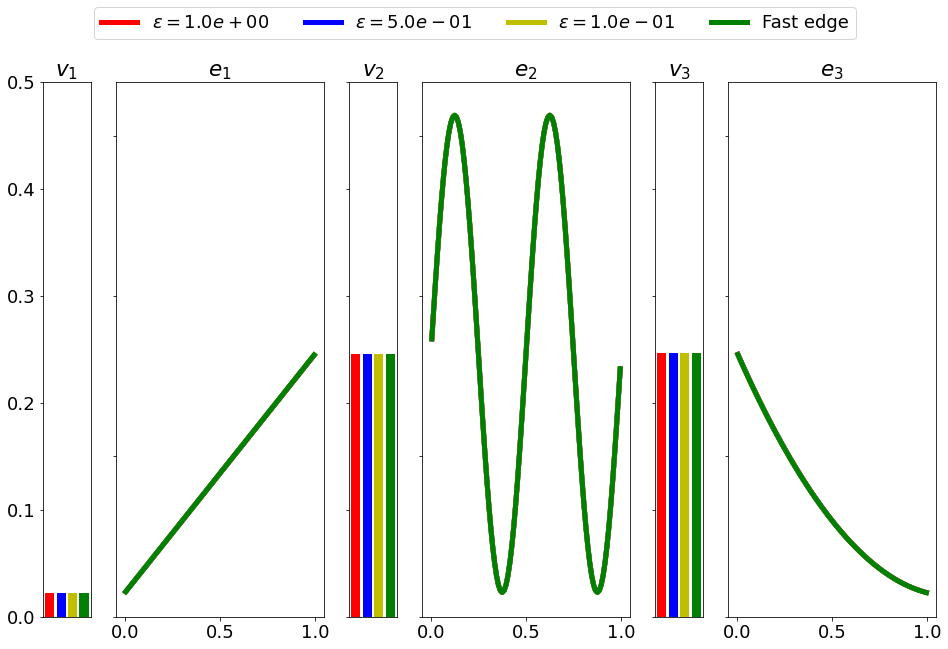}}
    \caption{Evolution of the measures $\upgamma^\eps = u^\eps\upomega$ for the prelimit system with rescaling $d^{\eps,e} = \eps^{-1}d^{e}$ with different values of $\eps$ as well as the for combinatorial system~\eqref{eq:discrete_fast}
    on the extended graph $(\hat\nodes,\hat\edges)$.}
    \label{fig:Fast_limit}
\end{figure}

\paragraph*{Combinatorial graph limit } 
We now address the combinatorial graph limit going from \eqref{eq:discrete_fast} to the system on original combinatorial graph $(\nodes,\edges)$ given by
\begin{align}
    \label{eq:system_numeris_EdgePoints}
    \omega_\sfv^0\partial_t u^0_\sfv(t) &= \sum_{\sfv\sfw \in \edges \text{ or } \sfw\sfv \in \edges}\bar\scrk_{\sfv\sfw}\sqrt{\omega_\sfv^0\omega_\sfw^0}\bra*{u_\sfw^0 - u_\sfv^0},  &&\forall \sfv\in \nodes  \,,
\end{align}
with $u_\sfv = \gamma_\sfv(t) / \omega_\sfv^0$, $u_\sfw = \gamma_\sfw(t) / \omega_\sfw^0$ and where
\[ 
    \overline\scrk_{\sfv\sfw} = \frac{\mathsf{Harm}\bra[\big]{ \scrk^{e}_\sfv \sqrt{\pi^e|_\sfv\,\omega_\sfv},\scrk^{e}_\sfw \sqrt{\pi^e|_\sfw\,\omega_\sfw}}}{2\sqrt{\omega_\sfv \omega_\sfw}}.
\]

As discussed in Section~\ref{ssec:terminal}, this limit is achieved by the rescaling $\scrk^{\eps,e}_\sfv = \eps^{-\frac12} \scrk^e_\sfv = \eps^{-\frac12}$ as well as  $\pi^{\eps,e} = \eps\pi^e/Z^\eps$ and  $\omega_\sfv^\eps = \omega_\sfv/Z^\eps$, where we recall $Z^\eps = \eps\sum_{e\in\edges}\pi^e([0,1]) + \sum_{\sfv\in\nodes}\omega_\sfv$ in \eqref{eq:discrete_fast}.

Since both the prelimit and the limit systems are discrete, we refrain from presenting the evolution of their profiles.
We do, however, consider the convergence of the relative entropies and the Hellinger distances for well-prepared initial data. 
Indeed, this is discussed next.

\paragraph*{Relative entropies and Hellinger distances }
We now study the relative entropies $\calE(\upgamma) = \calH(\upgamma|\upomega)$
for different values of $\eps$ and each of the limits discussed in Section~\ref{sec:multiscale}. 
Furthermore, we calculate for each case the time integrals of Hellinger-type distances between the respective prelimit and limit systems. They are given by
\begin{align}\label{eq:dist_compare}
    \mathrm{H}(\upgamma_1, \upgamma_2) = \frac{1}{2}\int_0^T\bra[\Bigg]{\sum_{e \in \edges} \sum_{k=1}^n\bra[\Bigg]{\sqrt{\frac{\tilde \gamma_{1,k}^e}{\tilde \omega_{k}^e}}-\sqrt{\frac{\tilde \gamma_{2,k}^e}{\tilde \omega_{k}^e}}}^2\;dx + \sum_{\sfv\in \nodes} \bra*{\sqrt{\frac{\bar\gamma_{1,\sfv}}{\omega_v}}-\sqrt{\frac{\bar\gamma_{2,\sfv}}{\omega_v}}}^2} \,,
\end{align}
with obvious modifications if we only compare vertex-based quantities.
In order to have a reasonable comparison, we use well-prepared initial data for all these studies. 
To this end, we chose the initial densities 
\begin{align*}
\tilde u^{e_1}_k &= k/n, \;\tilde u^{e_2}_k = 1/n, \; \tilde u^{e_3}_k = 1-k/n \qquad\text{and}\qquad \gamma_{\sfv_1} = 0,\, \gamma_{\sfv_2} = \gamma_{\sfv_3} = 1.
\end{align*}
This results in the discrete approximation of a continuous function over the graph, which is necessary to ensure well-preparedness for the Kirchhoff limit.

The results for the relative entropies for Kirchhoff, fast edge diffusion, and combinatorial limits are displayed in Figure \ref{fig:entropies}. 
We observe exponential convergence for all systems, yet with rates depending on $\eps$.
In all cases the limit systems converge faster than the rescaled ones for $\eps > 0$, which can be explained by the additional states that are still active (e.g. $\zeta$ in the combinatorial limit) in the prelimit systems and need to be dissipated.
What is more, as shown in Section~\ref{sec:multiscale}, as $\eps\to 0$, the entropies converge to the limit entropy. 
Similarly, we observe that for all cases the Hellinger distances between prelimit and limit decrease with $\eps$ as shown in Table~\ref{tab:Hellinger}.
Here, we stress that the dependence of the distances on $\eps$ is not comparable across different limits, since the role of $\eps$ is distinct for each limit.

\begin{figure}[ht]
    \centering
    \subfloat[Kirchhoff limit]{\includegraphics[width=.33\textwidth]{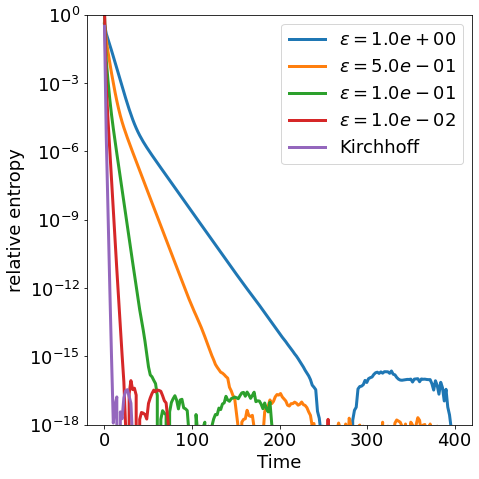}}
    \subfloat[Fast edge limit]{
    \includegraphics[width=.33\textwidth]{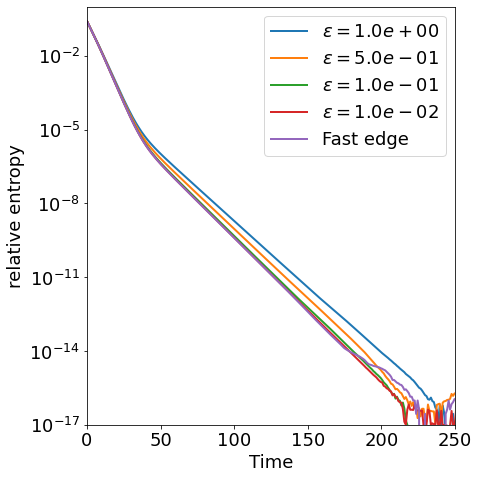}
    }
      \subfloat[Combinatorial limit]{\includegraphics[width=.33\textwidth]{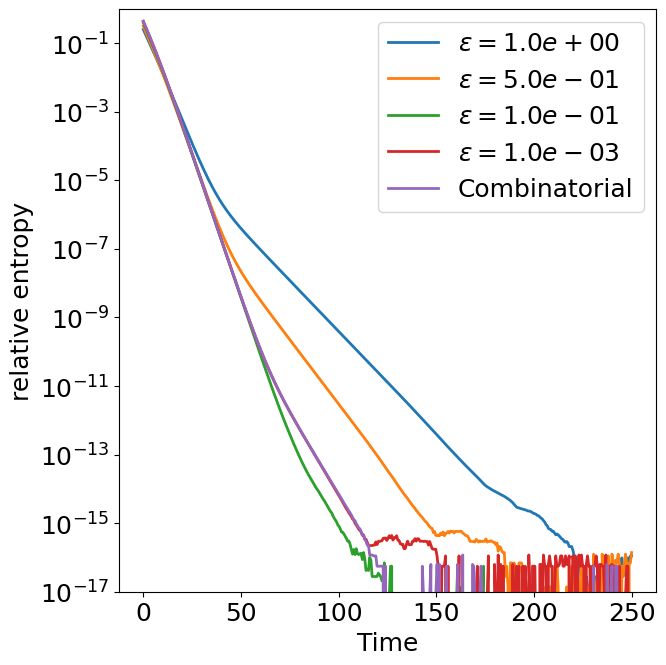}}
    \caption{Evolution of the relative entropies (logarithmic axes) as a function of time for different limits and different values of $\eps$.}
    \label{fig:entropies}
\end{figure}

\begin{table}
\centering
\begin{tabular}{|l|l|l|l|l|}
\hline
               & $\eps=1$ & $\eps=0.1$ & $\eps=0.01$ & $\eps=0.001$ \\ \hline
Kirchhoff      & 4.37 & 4.18e-1 & 1.17e-2 & 5.64e-5 \\ \hline
Fast diffusion & 2.14e-3 &5.43e-4 & 2.2e-5 & 3.92e-8   \\ \hline
Combinatorial  & 8.59e-5 & 2.14e-5 & 8.46e-7 & 8.44e-11  \\ \hline
\end{tabular}
\caption{Values of the Hellinger distance defined in \eqref{eq:dist_compare} for different values of $\eps$. For fast diffusion and combinatorial, convergence is already observed at $\eps=0.001$.}
\label{tab:Hellinger}
\end{table}

\begin{figure}[h!]
    \centering
    \subfloat[Joint limit at $n=50$]{\includegraphics[width=.33\textwidth]{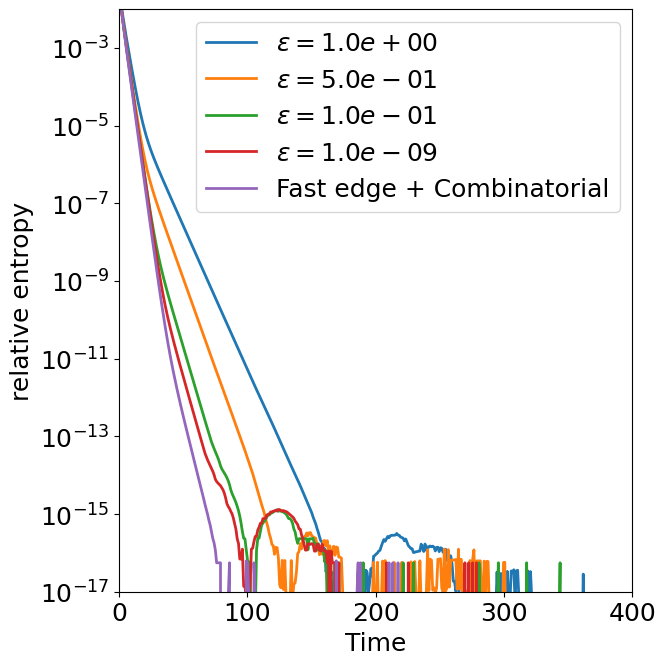}}
    \subfloat[Joint limit at $n=400$]{
    \includegraphics[width=.33\textwidth]{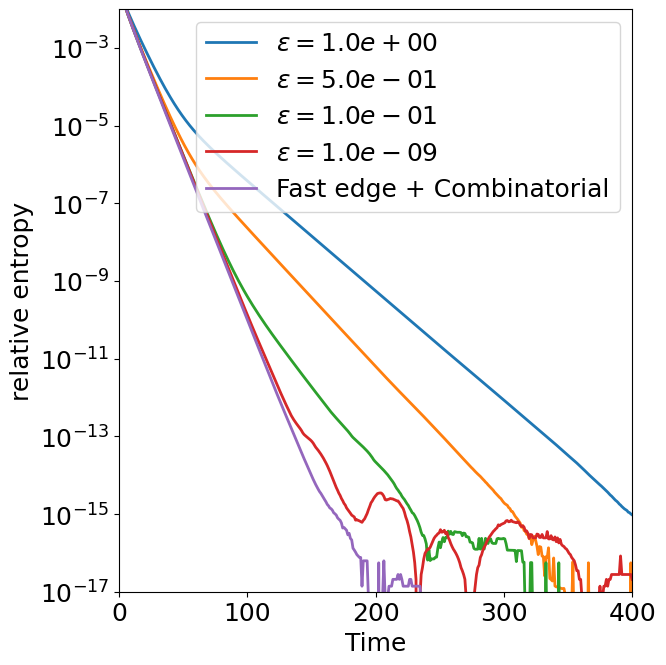}
    }
      \subfloat[Hellinger distances]{\includegraphics[width=.33\textwidth]{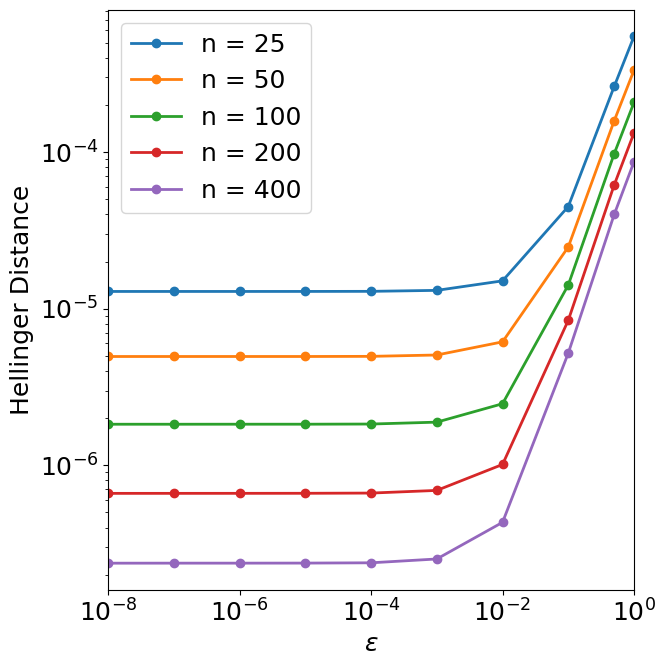}}    
    \caption{(a),(b): Evolution of the relative entropies (logarithmic axes) as a function of time for different values of $\eps$ for $n=50$ and $n=400$, respectively.\\
    (c): Time integrated Hellinger distances between prelimit and limit curves for different values of $n$ and $\eps$ (both with logarithmic axes).}
    \label{fig:hellinger}
\end{figure}

\paragraph{Joint fast edge diffusion and combinatorial limit }
Recalling Remark~\ref{rem:joint_limit}, we now address the joint fast edge diffusion ($\diffedge^{\eps,e} = \eps^{-1} \diffedge^e$) and combinatorial ($\scrk^{\eps,e}_\sfv = \eps^{-\frac12} \scrk^e_\sfv = \eps^{-\frac12}$ as well as  $\pi^{\eps,e} = \eps\pi^e/Z^\eps$ and  $\omega_\sfv^\eps = \omega_\sfv/Z^\eps$) limit. 
In Figure~\ref{fig:hellinger} (a) and (b), we observe that 
as $\eps$ decreases, there remains a visible difference in the rate with which entropy decreases, when compared to the limit system. However, we recall that we employ a discrete approximation for the edge diffusion, i.e., we perform in fact a combinatorial limit in the spirit of~\cite[§7]{PeletierSchlichting2022}.
For finite $n$, the approximation error of the discrete approximation of the edge diffusion starts to dominate the total error as $\eps$ becomes small.
Indeed, Figure~\ref{fig:hellinger} (c) shows that the Hellinger distances between the prelimit and limit curves decrease as $n$ increases, suggesting a convergence rate of order $1/n$.

\section{Conclusion and outlook}\label{sec:conclusion}

\paragraph*{Incorporating nonlocal interaction energies into the framework }
    As in~\cite{ErbarFathiLaschosSchlichting2016,ErbarForkertMaasMugnolo2022} it is 
    possible to generalize our approach to systems satisfying a \emph{local detailed balance} condition, which relaxes the detailed balance condition~\eqref{eq:DBC:r} as follows: For all $\mu=(\gamma,\rho)\in \mathcal{P}(\Mgraph)$ there exist jump rates $r[\mu]: \nodes \times \edges \to [0,\infty)$, $r[\mu]: \edges \times \nodes \to [0,\infty)$ and a \emph{local equilibrium} $(\omega[\mu],\pi[\mu])\in \calP_+(\Mgraph)$ such that it holds
    \begin{equation}\label{eq:localDBC}
        r[\mu](e,\sfv) \, \pi^e[\mu]|_{\sfv} = r[\mu](\sfv,e) \, \omega_\sfv[\mu] \qquad \forall (\sfv,e)\in \nodes\times \edges \,.
    \end{equation}
    A state $\mu^*\in \calP(\Mgraph)$ is stationary for the according dynamic~\eqref{eq:system_intro_linear} with $r$ replaced by $r[\mu]$ and $P$ replaced by $-\log \pi[\mu]$ if and only if $\mu^* = (\omega[\mu^*],\pi[\mu^*])$. In particular, this allows to treat nonlocal interaction energies by considering $\pi[\mu] = e^{-P[\mu]}$ with $P^e[\mu](x) \coloneqq \int_0^{\ell^e} K^e(x-y) \rho^e(\dx{x})$ for some kernel $K^e:\R \to \R$ for all $e\in\edges$. In this case, the driving energy~\eqref{eq:def:free_energy} has to be redefined such that $\calE'_\Medges(\mu) = \log \frac{\rho}{\pi(\mu)}$ and $\calE'_\nodes(\mu) = \log \frac{\gamma}{\omega(\mu)}$ (for details of the construction of the free energy see~\cite[Definition 2.3]{ErbarFathiLaschosSchlichting2016}). 
    
    The crucial a priori estimates required for obtaining existence as well as the main EDP convergence results can still be obtained as described before if the kernels $K^e$ are Lipschitz. Indeed, in this context we highlight \cite[Lemma~5.6]{ErbarForkertMaasMugnolo2022} stating that the Fisher information with Lipschitz interaction kernels is finite if and only if the standard Fisher information without such kernels is finite.

\paragraph*{Scaling limits in the discrete setting }
The numerical implementation of the combinatorial graph limit suggests that the rescalings of Section~\ref{ssec:edgepoints} and Section~\ref{ssec:terminal} commute with each other. It might be interesting to further investigate this observation. 

\paragraph*{Numerical studies }
In Section~\ref{sec:Numerics}, we restricted our considerations only to a single graph topology, a triangle shape. It would be interesting to numerically study more complicated graph topologies to develop a deeper understanding of possible phenomena that do not occur on a triangle-shaped graph.

\section*{Acknowledgements} %

The authors thank Martin Burger (Hamburg) and Matthias Erbar (Bielefeld) for fruitful discussion about gradient flows on metric graphs. 
GH was financed in part by a PhD scholarship of the German National Academic Foundation (Studienstiftung des deutschen Volkes).
JFP thanks the Deutsche Forschungsgemeinschaft (DFG) for support via the Research Unit FOR 5387 POPULAR, Project No. 511588106.
The research of AS is partially based upon work from COST Action 24122 mSPACE, supported
by COST (European Cooperation in Science and Technology), www.cost.eu.

\section*{Statements and declarations}

The code and data used to generate the numerical simulations presented in this work are available from the authors upon reasonable request.

\medskip\noindent
The authors declare that they have no conflict of interest to disclose.

\bibliographystyle{alphaabbr}
\bibliography{bib}
\end{document}

%% file: header.tex
\usepackage{latexsym,amssymb}
\usepackage{amsmath,amsthm,amsxtra,amsbsy,accents}
\usepackage{mathtools}
\usepackage{bm}
\usepackage[scr=boondox]{mathalpha}
\usepackage{hyperref}
\usepackage{subfig}

\usepackage{enumitem}

\usepackage{upgreek}

\usepackage[only,llbracket,rrbracket]{stmaryrd}

\DeclareMathAlphabet{\mathup}{OT1}{\familydefault}{m}{n}
\newcommand{\dx}[1]{\mathop{}\!\mathup{d} #1}
\newcommand{\pderiv}[3][]{\frac{\mathop{}\!\mathup{d}^{#1} #2}{\mathop{}\!\mathup{d} #3^{#1}}}

\newcommand{\eps}{\varepsilon}
\DeclarePairedDelimiter{\abs}{\lvert}{\rvert}
\DeclarePairedDelimiter{\norm}{\lVert}{\rVert}
\DeclarePairedDelimiter{\bra}{(}{)}
\DeclarePairedDelimiter{\pra}{[}{]}
\DeclarePairedDelimiter{\set}{\{}{\}}
\DeclarePairedDelimiter{\skp}{\langle}{\rangle}
\makeatletter
\newcommand{\customlabel}[2]{%
   \protected@write \@auxout {}{\string \newlabel {#1}{{#2}{\thepage}{#2}{#1}{}} }%
   \hypertarget{#1}{}
}
\makeatother

\numberwithin{figure}{section}

\makeatletter
\newcommand{\dashover}[2][\mathop]{#1{\mathpalette\df@over{{\dashfill}{#2}}}}
\newcommand{\fillover}[2][\mathop]{#1{\mathpalette\df@over{{\solidfill}{#2}}}}
\newcommand{\df@over}[2]{\df@@over#1#2}
\newcommand\df@@over[3]{%
	\vbox{
		\offinterlineskip
		\ialign{##\cr
			#2{#1}\cr
			\noalign{\kern1pt}
			$\m@th#1#3$\cr
		}
	}%
}
\newcommand{\dashfill}[1]{%
	\kern-.5pt
	\xleaders\hbox{\kern.5pt\vrule height.4pt width \dash@width{#1}\kern.5pt}\hfill
	\kern-.5pt
}
\newcommand{\dash@width}[1]{%
	\ifx#1\displaystyle
	2pt
	\else
	\ifx#1\textstyle
	1.5pt
	\else
	\ifx#1\scriptstyle
	1.25pt
	\else
	\ifx#1\scriptscriptstyle
	1pt
	\fi
	\fi
	\fi
	\fi
}
\newcommand{\solidfill}[1]{\leaders\hrule\hfill}
\makeatother
 \def\calB{{\mathcal B}} 
\def\calD{{\mathcal D}} \def\calE{{\mathcal E}} 
 \def\calH{{\mathcal H}} \def\calI{{\mathcal I}}
  \def\calL{{\mathcal L}}
\def\calM{{\mathcal M}}  
\def\calP{{\mathcal P}}  \def\calR{{\mathcal R}}

  \def\rmf{{\mathrm f}}
  
\def\rmj{{\mathrm j}}

 \def\rmB{{\mathrm B}} 
\def\rmD{{\mathrm D}}

\def\rmV{{\mathrm V}}

 \def\sfn{{\mathsf n}}

\def\sfv{{\mathsf v}} \def\sfw{{\mathsf w}}

  \def\sfC{{\mathsf C}}
 \def\sfE{{\mathsf E}} 
\def\sfG{{\mathsf G}}  \def\sfI{{\mathsf I}}
  \def\sfL{{\mathsf L}}
\def\sfM{{\mathsf M}}  
  \def\sfR{{\mathsf R}}
  
\def\sfV{{\mathsf V}}  \def\sfX{{\mathsf X}}
\def\sfY{{\mathsf Y}}

  \def\scrL{{\mathscr  L}}

 \def\scrk{{\mathscr  k}}

 \def\bbN{{\mathbb N}} 
  \def\bbR{{\mathbb R}}
 \def\bbT{{\mathbb T}}

\newcommand{\AC}{\mathrm{AC}}
\newcommand{\Mgraph}{\sfM}
\newcommand{\Medges}{\sfL}
\newcommand{\nodes}{\sfV}
\newcommand{\edges}{\sfE}
\newcommand{\normal}{\sfn}
\newcommand{\R}{\mathbb{R}}
\newcommand{\EDP}{{\mathsf{EDP}}}
\newcommand{\CE}{{\mathsf{CE}}}

\newcommand{\pCE}{{\overline{\mathsf{CE}}_{\hat\nodes}}}

\renewcommand{\div}{\mathop{\mathrm{div}}\nolimits}
\newcommand{\bdiv}{\mathop{\mathsf{div}}\nolimits}
\newcommand{\gdiv}{\mathop{\overline{\mathrm{div}}}\nolimits}

\newcommand{\gnabla}{\overline{\nabla}}

\newcommand{\diffedge}{d}

\DeclareMathOperator{\arsinh}{arsinh}
\makeatletter
\newcommand{\bnabla}{{\mathpalette\b@nabla\relax}}
\newcommand\b@nabla[2]{%
        \setbox\z@=\hbox{$\m@th#1\bigtriangledown$}%
        \ht\z@.7\ht\z@
        \raise\dp\z@\box\z@
}
\makeatother

\newtheorem{theorem}{Theorem}[section]
\newtheorem{assumption}[theorem]{Assumption}
\newtheorem{lemma}[theorem]{Lemma}
\newtheorem{corollary}[theorem]{Corollary}
\newtheorem{definition}[theorem]{Definition}

\newtheorem{proposition}[theorem]{Proposition}
\newtheorem{remark}[theorem]{Remark}
\newtheorem{remark*}{Remark}

\newtheorem{property}[theorem]{Property}
\newtheorem{principle}[theorem]{Principle}

\numberwithin{equation}{section}

\makeatletter
\newcommand{\oset}[3][0ex]{%
  \mathrel{\mathop{#3}\limits^{
    \vbox to#1{\kern-2\ex@
    \hbox{$\scriptstyle#2$}\vss}}}}
\makeatother
\usepackage{trimclip}
\makeatletter
\NewDocumentCommand{\rmjmath}{}{\mathbin{\mathpalette\eplus@\relax\mspace{1mu}}}
\newcommand{\eplus@}[2]{\clipbox{-.5 -.5 0 {.35\height}}{$\m@th#1\rmj$}}
\makeatother

\newcommand{\one}{1\!\!1}